\documentclass[11pt,oneside,english]{article}
\usepackage{mathtools}
\usepackage{amsthm}
\usepackage{amssymb}
\usepackage{stmaryrd}
\usepackage[unicode=true,pdfusetitle,
 bookmarks=true,bookmarksnumbered=false,bookmarksopen=false,
 breaklinks=false,pdfborder={0 0 1},backref=false,colorlinks=false]
 {hyperref}

\makeatletter
\numberwithin{equation}{section}
\numberwithin{figure}{section}
\theoremstyle{plain}
\newtheorem{thm}{\protect\theoremname}[section]
\theoremstyle{definition}
\newtheorem{defn}[thm]{\protect\definitionname}
\theoremstyle{remark}
\newtheorem{rem}[thm]{\protect\remarkname}
\theoremstyle{plain}
\newtheorem{assumption}[thm]{\protect\assumptionname}
\theoremstyle{plain}
\newtheorem{prop}[thm]{\protect\propositionname}
\theoremstyle{plain}
\newtheorem{lem}[thm]{\protect\lemmaname}
\theoremstyle{plain}
\newtheorem{cor}[thm]{\protect\corollaryname}

\usepackage[a4paper]{geometry}
\usepackage{changes}

\usepackage{times}

\allowdisplaybreaks[3]

\makeatother

\providecommand{\assumptionname}{Assumption}
\providecommand{\corollaryname}{Corollary}
\providecommand{\definitionname}{Definition}
\providecommand{\lemmaname}{Lemma}
\providecommand{\propositionname}{Proposition}
\providecommand{\remarkname}{Remark}
\providecommand{\theoremname}{Theorem}

\begin{document}
\title{Well-posedness of the obstacle problem for stochastic nonlinear diffusion
equations: an entropy formulation\thanks{This work is supported
by the National Science and Technology Major Project (2022ZD0116401).
The first author is also supported by the National Natural Science Foundation of China (No.~12222103),
and by Shanghai Institute for Mathematics and Interdisciplinary Sciences (Grand No.~SIMIS-ID-2024-WE),
and by LMNS at Fudan University.}}
\author{Kai Du\thanks{Shanghai Center for Mathematical Sciences,
Fudan University, Shanghai, China (Email: {\tt kdu@fudan.edu.cn}).
The author is a member of LMNS, Fudan University.}
\and 
Ruoyang Liu\thanks{Corresponding Author. Department of Mathematics, Shanghai Normal University, Shanghai, China (Email: {\tt ryliu@shnu.edu.cn}).}
}
\date{}
\maketitle

\vspace{-1ex}
\begin{abstract}
In this paper, we establish the existence, uniqueness and stability results for the obstacle problem
associated with a degenerate nonlinear diffusion equation perturbed by conservative
gradient noise. 
Our approach revolves round introducing a new entropy formulation for stochastic variational inequalities.
As a consequence, we obtain a novel well-posedness result for 
the obstacle problem of deterministic porous medium equations with nonlinear reaction terms.
\medskip 

\medskip 

\end{abstract}

\section{Introduction}

The aim of this paper is to establish the well-posedness of the obstacle
problem for a \emph{degenerate} nonlinear diffusion equation perturbed
by nonlinear conservative multiplicative noise. 
Specifically, we prove existence
and uniqueness results for a random field $\{u(t,x):t\ge0,\,x\in\mathbb{T}^{d}\}$
which lies above a given obstacle~$\psi$ and obeys the stochastic
partial differential equation (SPDE)
\begin{equation}
\partial_{t}u=\Delta\Phi(u)+f(t,x,u)+\nabla\cdot\sigma^{k}(x,u)\circ\dot{W}_{t}^{k}\label{eq:main}
\end{equation}
on the region $\{(t,x):u(t,x)>\psi(t,x)\}$ in a suitable sense. Here,
$\mathbb{T}^{d}$ denotes the $d$-dimensional torus, $\Phi:\mathbb{R}\to\mathbb{R}$
is an increasing function with $\Phi'(0)=0$, $\{W^{k}:k=1,2,\dots\}$ are independent
Wiener processes, and the Einstein summation convention is used over  $k=1,2,\dots$ in~\eqref{eq:main}.
The stochastic integral is interpreted in the Stratonovich sense.
Stratonovich noise arises naturally in modeling large-scale turbulence by using the Wong--Zakai principle (e.g. \cite{fedrizzi2017regularity, munteanu2018total, flandoli20212d, du2024entropy}). For further applications of gradient noise in SPDEs, see \cite{fehrman2019well, fehrman2024well, sauerbrey2024martingale}, etc.

There is a vast literature concerning the obstacle problems for deterministic
PDEs, usually under the name of variational inequalities, see e.g.
\cite{lions1967variational,brezis1972problemes,bensoussan1982applications,alt1983quasilinear,caffarelli1998obstacle}
among the pioneers, and \cite{bogelein2011degenerate,bogelein2015obstacle,scheven2015existence,caffarelli2017obstacle,athanasopoulos2019parabolic}
as more recent references. The obstacle problem for SPDE was first
studied by Haussmann and Pardoux~\cite{haussmann1989stochastic}, using a framework of Skorohod problem from
the study of reflected diffusion (see e.g. \cite{skorokhod1961stochastic,tanaka1979stochastic,lions1984stochastic}).
The key point of this framework is
to impose a ``minimal forcing'' on the equation to keep the solution
in the constraint. In the context of (\ref{eq:main}), this means
that, along with the random field $u$, one also needs to seek a random Radon measure $\nu$ satisfying the \emph{Skorohod condition}:
\begin{equation}\label{eq:skorohod}
    \langle u-\psi,\nu\rangle=0\quad\text{(a.s.)}
\end{equation}
such that the equation
\begin{equation}
\partial_{t}u=\Delta\Phi(u)+f(t,x,u)+\nabla\cdot\sigma^{k}(x,u)\circ\dot{W}_{t}^{k}+\nu\label{eq:SVI}
\end{equation}
holds everywhere (usually in the sense of distributions). 
The random measure $\nu$ plays a role of the ``minimal forcing'' that enforces the solution to stay above the obstacle.

In the framework of Skorohod, the obstacle problem for SPDE has been extensively investigated
under various scenarios, see e.g. \cite{nualart1992white,donati1993white,xu2009white,denis2014obstacle,brzezniak2023reflection}
for well-posedness, and \cite{dalang2006hitting,zambotti2017random,matoussi2021large}
for solution properties. 
We remark that all those works primarily address
\emph{non-degenerate} leading operators. 
However, the literature on degenerate equations is relatively sparse.
Yang and Zhang \cite{yang2019obstacle} discussed a degenerate equation with Lipschitz nonlinearity. 
Regarding stochastic porous medium equations, R\"ockner, Wang, and Zhang \cite{rockner2013stochastic} considered a constant obstacle $\psi\equiv 0$ and proved the existence of solutions in the space $H^{-1}$.
Liu and Tang \cite{liu2024obstacle} obtained the well-posedness, but imposing a special assumption on the barrier $\psi$. 
Remarkably, even for deterministic degenerate diffusion equations like the porous medium equation,  the complete well-posedness of the corresponding obstacle problem remains elusive (cf.~\cite{levi2007obstacle,levi2008mathematical,bogelein2015obstacle}).

The lack of regularity products the main difficulty in our problem.
Specifically, the compensation term $\nu$ in (\ref{eq:SVI}) is only a Radon measure on $Q_{T}:=[0,T)\times\mathbb{T}^{d}$
from a priori estimates (see Proposition \ref{prop:estimate for u_n,epsilon}),
and as a result, the continuity in space and time of the solution
$u$ is necessary to ensure the Skorohod condition \eqref{eq:skorohod} well-defined. To our knowledge,
such a continuity property seems not to be always available of the 
solutions to degenerate diffusion equations \cite{gess2020optimal,bruno2022optimal}.
Therefore, it is imperative to construct a novel framework
to comprehend the obstacle problem of SPDE.

In this paper, we propose an entropy formulation for the obstacle problem for SPDE \eqref{eq:main},
based on the theory of entropy solutions of stochastic nonlinear diffusion
equations recently developed by \cite{dareiotis2019entropy,dareiotis2020nonlinear,du2024entropy}.
Roughly speaking, the random field $u$ along with a random Radon measure $\nu$ is an \emph{entropy solution} to the equation (\ref{eq:main}) subject to the constraint $u\geq\psi$, if for each convex $\eta\in C^{2}(\mathbb{R})$
with compact $\mathrm{supp}\,(\eta^{\prime\prime})$, the inequality
\begin{equation}
\partial_{t}\eta(u)\le\eta^{\prime}(u)\big[\Delta\Phi(u)+f(t,x,u)+\nabla\cdot\sigma^{k}(x,u)\circ\dot{W}_{t}^{k}\big]+\eta^{\prime}(\psi)\nu\label{eq:entropy-solution}
\end{equation}
holds in the sense of distribution (see Definition \ref{def:entropy solution with ob}
for the precise formulation). 
Remarkably, we do not impose the Skorohod condition explicitly into this definition. 

At first glance, it may seem more intuitive to replace the term $\eta^{\prime}(\psi)\nu$ in \eqref{eq:entropy-solution} with $\eta^{\prime}(u)\nu$,
because the essence of entropy solution is to test the inequality satisfied by $\eta(u)$, 
and the chain rule obviously suggests $\eta^{\prime}(u)\nu$ rather than $\eta^{\prime}(\psi)\nu$.
However, retaining this term is justifiable due to two observations. Firstly,
the ``right'' term $\eta^{\prime}(u)\nu$ may not be well-defined
due to the unmatched regularity of $u$ and $\nu$, while $\eta^{\prime}(\psi)\nu$
is well-defined as long as the obstacle $\psi$ is properly regular.
Secondly, we observe (formally) that $\langle\eta^{\prime}(u)-\eta^{\prime}(\psi),\nu\rangle=0$
from the Skorohod condition, and $\eta^{\prime}(u)\geq\eta^{\prime}(\psi)$
as $\eta^{\prime}$ is increasing, so we actually have $\eta^{\prime}(u)\nu=\eta^{\prime}(\psi)\nu$.
In a nutshell, by leveraging the characteristics of entropy solutions,
we embed the Skorohod condition naturally into the entropy inequality
(\ref{eq:entropy-solution}) that can be well-defined. This replacement of $\eta^{\prime}(u)\nu$ with $\eta^{\prime}(\psi)\nu$ is a key novelty of our formulation, providing a new perspective on handling regularity issues in entropy-based frameworks.

The main results of this paper, Theorems \ref{thm:existence_main thm} and \ref{thm:unque_main thm}, establish the existence, uniqueness, and stability
of entropy solutions to the obstacle problem for SPDE \eqref{eq:main}, 
subject to mild assumptions on the obstacles. 
Analogous to the theory of entropy solutions for PDEs \cite{chen2003well,chen2004quasilinear},
proper regularity is indispensable for effectively defining entropy solutions in our problem.
Specifically, the solution must possess the regularity $\sqrt{\Phi^\prime(u)}\nabla {u}\in L_2(Q_T)$.
Fortunately, this requisite regularity property is automatically satisfied by the solution constructed through the method of vanishing viscosity in our proof of the existence. 

The outcomes of this paper yield new results for deterministic nonlinear diffusion equations.
As an example, we will show how our results apply to porous medium equations with reaction terms (see Section~\ref{sec:Deterministic}).
The obstacle problem for porous medium equations finds applications in modeling the flow of substance moving in subsoil with saturation \cite{gagneux1995analyse,levi2005singular,levi2007obstacle,levi2008mathematical,amorim2017obstacle}, with its well-posedness explored in previous works such as  \cite{levi2007obstacle,levi2008mathematical,bogelein2015obstacle,avelin2017comparison,bogelein2018doubly}. 
L\'{e}vi and Vallet~\cite{levi2007obstacle,levi2008mathematical} 
obtained a well-posedness result for the case of constant barriers
in an entropy formulation that differs from ours.
In the framework of variational inequality, B\"ogelein, Lukkari, and Scheven~\cite{bogelein2015obstacle} investigated general barriers and demonstrated the existence of strong solutions, albeit leaving the uniqueness as a conjecture.
Utilizing the entropy formulation proposed here, one can ascertain not only the existence but also the uniqueness of the solution for a general form of equations and a wide range of obstacles.
Evidently, the entropy solution has an intimate relation with the variational solutions.
We will demonstrate in Section~\ref{sec:Deterministic} that
the entropy solution actually satisfies a variational inequality in the weak sense, and conversely, a strong variational solution is also an entropy solution if the obstacle $\psi$ is greater than a positive number.

Our strategy of proving the existence is to construct an entropy solution by a penalization method. Specifically, we consider the penalized equation
\[
\partial_{t}u_{\epsilon}=\Delta\Phi(u_{\epsilon})+f(t,x,u_{\epsilon})+\nabla\cdot\sigma^{k}(x,u_{\epsilon})\circ\dot{W}_{t}^{k}+\epsilon^{-1}(u_{\epsilon}-\psi)^{-},\quad\epsilon>0,
\]
whose solvability is provided by \cite{dareiotis2019entropy,du2024entropy}.
We prove the monotonicity of $u_{\epsilon}$ and a uniform energy
estimate for $u_{\epsilon}$ by further mollifying the equations (see
Sections \ref{sec:Approximation}--\ref{sec:Comparison-principle}),
and by these we obtain a limit $u$ and verify the convergence of
each terms in the definition of entropy solution, except the term
involving the measure $\nu$. 
For the latter, we derive a uniform estimate

\[
\sup_{\epsilon>0}\,\mathbb{E}\bigg(\iint\epsilon^{-1}(u_{\epsilon}-\psi)^{-}\mathrm{d}x\mathrm{d}t\bigg)^{(m+1)/2}<\infty.
\]
This along with a compactness argument yields the existence of $\nu$. We also shows that $\nu$ is almost surely a finite Radon measure
on $Q_{T}$. {For existence, we require the obstacle to be H\"older continuous.}

The uniqueness of the entropy solution follows from an $L_1$-stability estimate.
To prove the latter, we apply Kruzhkov's doubling-variables methods to 
estimate the difference between an entropy solution and another one that satisfies an extra condition called the $(\star)$-property (see Definition \ref{def:star property} and Proposition \ref{prop:uniqueentropy}). 
On the other hand, we verify that such a property is automatically satisfied by the solutions of approximating equations in our proof of the existence.
This along with the $L_1$-stability yields the uniqueness. 
Technically, the low regularity of $\nu$ brings new difficulties in proving the $L_1$-estimates in contrast to the previous work \cite{liu2024obstacle} where the measure $\nu$ lies in the space $L_2(Q_T)$ due to a specific structure of the obstacles they considered. 
To overcome the challenge, we carefully analyze the term involving the compensation measures (see \eqref{eq:FGdiffL1-1-1-1}):
with the help of the Taylor expansion of the barrier function $\psi$, we utilize the symmetry of the specific test function to cancel out lower-order terms and strategically take limits to eliminate the higher-order terms. 
For this reason, we require the obstacle $\psi\in C^{2}_{x}(\bar{Q}_T)$
in our uniqueness result.

This paper is organized as follows. In Section \ref{sec:entropy formula}, we state the main results
after introducing notations and assumptions. 
Section \ref{sec:-stability} is devoted to deriving an $L_1$-stabilty estimate and the uniqueness of solutions under
an extra condition called the $(\star)$-property. 
In Section \ref{sec:Approximation}, we introduce the approximating penalized equations, derive crucial estimates for the solutions $u_{n,\epsilon}$ 
and the penalty terms $\epsilon^{-1}(u_{n,\epsilon}-\psi)^{-}$,
and validate the $(\star)$-property of $u_{n,\epsilon}$.
In Section \ref{sec:Comparison-principle}, we prove a comparison principle, which indicates that $u_{n,\epsilon}$ is increasing as $\epsilon\downarrow0$. 
In Section \ref{sec:Existence}, we finalize the proof of the existence of
entropy solutions to the original problem by passing to the limit of $u_{n,\epsilon}$ 
as $n\rightarrow\infty$ and $\epsilon\downarrow 0$. 
In Section \ref{sec:Deterministic}, we apply our result to the obstacle problem for porous medium equations, and discuss the relation between entropy solutions and variational solutions in this case. 
Two auxiliary lemmas are proved in the final section.

\section{Entropy formulation and main results\label{sec:entropy formula}}

Let $(\Omega,\mathcal{F},\{\mathcal{F}_{t}\},\mathbb{P})$ be a complete
filtered probability space and $\{W^{k}\}$ be a sequence of independent
$\mathcal{F}_{t}$-adapted Wiener processes. 
Denote $x=(x_{1},x_{2},\ldots,x_{d})$ for $x\in\mathbb{R}^{d}$.
For a fixed $T>0$, define
$Q_{T}:=[0,T)\times\mathbb{T}^{d}$ and $\Omega_{T}:=\Omega\times[0,T]$.
Let $L_{p}$ and $H_{p}^{s}$ be the usual Lebesgue and Sobolev spaces
with $p\geq2$ and $s\in\mathbb{R}$; we simplify $H_{2}^{s}$ as
$H^{s}$ (cf. \cite{evans1998partial}).
Moreover, we denote by $C^\kappa_x(\bar{Q}_T)$ the space of continuous functions $f$ on $\bar{Q}_T$
such that $f(t,\cdot)\in C^\kappa(\mathbb{T}^d)$ and its $C^\kappa$-norm is bounded uniformly in $t\in[0,T]$.

For simplicity, we denote by $\Pi_{\psi}(\Phi,f,\xi)$ the obstacle
problem for (\ref{eq:main}) with the (lower) obstacle $\psi$ and with
the initial data $u(0,x)=\xi(x)$, and by $\Pi(\Phi,f,\xi)$ the Cauchy
problem for (\ref{eq:main}) without obstacles. 
To formulate the entropy solution explicitly, we rewrite (\ref{eq:main})
into It\^o's form: 
\begin{align}
\mathrm{d}u & =\big[\Delta\Phi(u)+\partial_{x_{i}}\big(a^{ij}(x,u)\partial_{x_{j}}u+b^{i}(x,u)\big)+f(t,x,u)\big]\mathrm{d}t\label{eq:ito}\\
&\quad+\nabla\cdot\sigma^{k}(x,u)\mathrm{d}W_{t}^{k},\nonumber
\end{align}
where
\begin{align*}
a^{ij}(x,r) & :=\frac{1}{2}\sigma_{r}^{ik}(x,r)\sigma_{r}^{jk}(x,r),\quad b^{i}(x,r):=\frac{1}{2}\sigma_{r}^{ik}(x,r)\sigma_{x_{j}}^{jk}(x,r).
\end{align*}
Hereafter, we use the Einstein summation convention over $k=1, 2, \dots$ and $i,j=1,\dots,d$.
Moreover, we define the set
\[
\mathcal{E}:=\{\eta\in C^{2}(\mathbb{R}):\eta\ \mathrm{is\ convex\ with}\ \eta^{\prime\prime}\ \mathrm{compactly\ supported}\};
\]
and for a function $g:\mathbb{T}^{d}\times\mathbb{R}\rightarrow\mathbb{R}$,
we write
\[
\llbracket g\rrbracket(x,r):=\int_{0}^{r}g(x,s)\mathrm{d}s,\quad r\in\mathbb{R}.
\]
In some circumstances, we write $g_{x_{i}}(x,r)=\partial_{x_{i}}g(x,r)$
and $g_{r}(x,r)=\partial_{r}g(x,r)$. Fix $m>1$.
\begin{defn}
\label{def:entropy solution with ob}An entropy solution to the obstacle
problem $\Pi_{\psi}(\Phi,f,\xi)$ is a pair $(u,\nu)$ such that
\begin{itemize}
\item[(i)] $u$ is a predictable random field on $\bar{Q}_{T}$ such that $u\ge\psi$
for almost all $(\omega,t,x)\in\Omega\times Q_{T}$ and
\[
u\in L_{m+1}(\Omega_{T};L_{m+1}(\mathbb{T}^{d})),\quad\llbracket\sqrt{\Phi^{\prime}}\rrbracket(u)\in L_{2}(\Omega_{T};H^{1}(\mathbb{T}^{d}));
\]
\item[(ii)] $\nu$ is a predictable Radon measure on $Q_{T}$ and $\mathbb{E}[\nu(Q_T)]<\infty$;
\item[(iii)] for all $(\eta,\varphi,\varrho)\in\mathcal{E}\times C_{c}^{\infty}([0,T))\times C^{\infty}(\mathbb{T}^{d})$
and $\phi:=\varphi\varrho\geq0$, it holds almost surely that
\begin{equation}
\begin{aligned} & -\int_{Q_{T}}\eta(u)\partial_{t}\phi\mathrm{d}x\mathrm{d}t-\int_{Q_{T}}\eta^{\prime}(\psi)\phi\nu(\mathrm{d}x\mathrm{d}t)\\
 & \leq\int_{\mathbb{T}^{d}}\eta(\xi)\phi(0)\mathrm{d}x+\int_{Q_{T}}\llbracket\Phi^{\prime}\eta^{\prime}\rrbracket(u)\Delta\phi\mathrm{d}x\mathrm{d}t+\int_{Q_{T}}\llbracket a^{ij}\eta^{\prime}\rrbracket(u)\phi_{x_{i}x_{j}}\mathrm{d}x\mathrm{d}t\\
 & \quad+\int_{Q_{T}}\Big(\llbracket a_{x_{j}}^{ij}\eta^{\prime}+b_{r}^{i}\eta^{\prime}\rrbracket(x,u)-2\eta^{\prime}(u)b^{i}(x,u)\Big)\phi_{x_{i}}\mathrm{d}x\mathrm{d}t\\
 & \quad+\int_{Q_{T}}\Big(-\eta^{\prime}(u)b_{x_{i}}^{i}(x,u)+\llbracket b_{rx_{i}}^{i}\eta^{\prime}\rrbracket(x,u)+\eta^{\prime}(u)f(t,x,u)\Big)\phi\mathrm{d}x\mathrm{d}t\\
 & \quad+\int_{Q_{T}}\Big(\frac{1}{2}\eta^{\prime\prime}(u)\sum_{k=1}^{\infty}|\sigma_{x_{i}}^{ik}(x,u)|^{2}-\eta^{\prime\prime}(u)|\nabla\llbracket\sqrt{\Phi^{\prime}}\rrbracket(u)|^{2}\Big)\phi\mathrm{d}x\mathrm{d}t\\
 & \quad+\int_{Q_{T}}\Big(\eta^{\prime}(u)\phi\sigma_{x_{i}}^{ik}(x,u)-\llbracket\sigma_{rx_{i}}^{ik}\eta^{\prime}\rrbracket(x,u)\phi-\llbracket\sigma_{r}^{ik}\eta^{\prime}\rrbracket(x,u)\phi_{x_{i}}\Big)\mathrm{d}x\mathrm{d}W_{t}^{k}.
\end{aligned}
\label{eq:entropy formula-1}
\end{equation}
\end{itemize}
\end{defn}


To establish the well-posedness of the obstacle problem,
we need some structural conditions of the equations.
The following assumptions are used in \cite{dareiotis2020nonlinear,du2024entropy}.
Fix constants $K\geq1$ and $\kappa\in(0,1]$.
\begin{assumption}
\label{assu:assumption for phi} $\Phi\in C^{2}(\mathbb{R})$ is strictly
increasing and odd, $\sqrt{\Phi^{\prime}}$ is away from the origin,
and
\[
\begin{gathered}|\sqrt{\Phi^{\prime}}(0)|\leq K,\quad|(\sqrt{\Phi^{\prime}})^{\prime}(r)|\leq K|r|^{\frac{m-3}{2}},\quad\sqrt{\Phi^{\prime}}(r)\geq K^{-1}\mathbf{1}_{\{|r|\geq1\}},\\
mK|\llbracket\sqrt{\Phi^{\prime}}\rrbracket(r)-\llbracket\sqrt{\Phi^{\prime}}\rrbracket(s)|\geq\begin{cases}
|r-s|, & \textrm{if}\ |r|\lor|s|\geq1;\\
|r-s|^{\frac{m+1}{2}}, & \textrm{if}\ |r|\lor|s|<1.
\end{cases}
\end{gathered}
\]
\end{assumption}

A typical choice of $\Phi$ is $\Phi(r)=|r|^{m-1}r$, associated with
porous medium equations.

\begin{assumption}
\label{assu:assumption for sigma} $\sigma^{k}\in C^{3}(\bar{Q}_{T};\mathbb{R}^{d})$
for each $k=1,2,\dots$, and
\[
\sum_{k=1}^{\infty}\Vert\sigma^{k}\Vert_{C^{3}(\bar{Q}_{T};\mathbb{R}^{d})}^{2}\leq K;
\]
$f\in C(\bar{Q}_{T}\times\mathbb{R})$ and satisfies
\begin{align}\label{eq:assumption for f}
 & \sup_{t,r}\,\Vert f(t,\cdot,r)\Vert_{C^{\kappa}(\mathbb{T}^d)}+\Vert f_{r}\Vert_{L_{\infty}(Q_{T}\times\mathbb{R})}\leq K.
\end{align}
\end{assumption}

Our main results are the following two theorems.

\begin{thm}[existence]
\label{thm:existence_main thm}Let Assumptions \ref{assu:assumption for phi}
and \ref{assu:assumption for sigma} be satisfied
and $\psi\in C^{\kappa}_x(\bar{Q}_{T})$ with $\kappa\in (0,1]$. 
Then for any  $\xi\in L_{m+1}(\Omega,\mathcal{F}_{0};L_{m+1}(\mathbb{T}^{d}))$
satisfying 
\begin{equation}
\xi(x)\geq\psi(0,x),\quad\forall(\omega,x)\in\Omega\times\mathbb{T}^{d},\label{eq:assumption on xi}
\end{equation}
there exists an entropy solution $(u,\nu)$ to $\Pi_{\psi}(\Phi,f,\xi)$.
\end{thm}

\begin{thm}[$L_1$-stability and uniqueness]
\label{thm:unque_main thm}Let Assumptions \ref{assu:assumption for phi}
and \ref{assu:assumption for sigma} be satisfied and $\psi\in C^{2}_x(\bar{Q}_{T})$. 
Let $(u,\nu)$ and $(\tilde{u},\tilde{\nu})$ be the entropy solutions to $\Pi_{\psi}(\Phi,f,{\xi})$ and $\Pi_{\psi}(\Phi,f,\tilde{\xi})$, respectively,  with $\xi,\tilde{\xi}\in L_{m+1}(\Omega,\mathcal{F}_{0};L_{m+1}(\mathbb{T}^{d}))$ satisfying (\ref{eq:assumption on xi}).
Then there is a constant $C$ depending only on $K$, $d$, and $T$
such that
\[
\underset{t\in[0,T]}{\mathrm{ess\ sup}}\,\mathbb{E}\Vert u(t,\cdot)-\tilde{u}(t,\cdot)\Vert_{L_{1}(\mathbb{T}^{d})}\leq C\, \mathbb{E}\Vert\xi-\tilde{\xi}\Vert_{L_{1}(\mathbb{T}^{d})}.
\]
Consequently, the obstacle problem $\Pi_{\psi}(\Phi,f,{\xi})$ admits at most one entropy solution.
\end{thm}

\begin{rem}
The advantage of our definition is its ability to address cases where the Skorohod condition~\eqref{eq:skorohod} may not be well-defined due to the weak regularity of $u$. Importantly, our definition guarantees the validity of the Skorohod condition, provided it is well-posed (such as when $u$ is continuous or quasi-continuous, as in \cite{denis2014obstacle}). 
To see this, one can derive the entropy inequality~\eqref{eq:entropy formula-1} for $\eta(r)=r^2$ by approximation, and compare this with the direct application of It\^{o}'s formula to $u^2$. This comparison naturally establishes the validity of the Skorohod condition.
\end{rem}

Our proof of the above theorems relies on an $L_1$-stability estimate, Proposition~\ref{prop:uniqueentropy}.
Compared with Theorem~\ref{thm:unque_main thm}, 
Proposition~\ref{prop:uniqueentropy} imposes an extra condition called the ($\star$)-property (see Definition \ref{def:star property}) on one of the solutions that are compared.

The proof of Theorem \ref{thm:existence_main thm} is finalized in Section~\ref{sec:Existence}.
Here we conclude the proof of Theorem~\ref{thm:unque_main thm}.

\begin{proof}[Proof of Theorem~\ref{thm:unque_main thm}]
Define $\xi_{n}:=(-n)\lor(\xi\land n)$. 
It is shown in Section \ref{sec:Existence} that
the problem $\Pi_{\psi}(\Phi,f,\xi_{n})$ admits an entropy solution $(u_{n},\nu_{n})$ that has the $(\star)$-property.
Then by Proposition \ref{prop:uniqueentropy}, one has
\[
\begin{aligned}
& \underset{t\in[0,T]}{\mathrm{ess\ sup}}\,\mathbb{E}\Vert u(t,\cdot)-\tilde{u}(t,\cdot)\Vert_{L_{1}(\mathbb{T}^{d})} \\
&\le \underset{t\in[0,T]}{\mathrm{ess\ sup}}\,\mathbb{E}\Vert u_{n}(t,\cdot)-{u}(t,\cdot)\Vert_{L_{1}(\mathbb{T}^{d})} 
+ \underset{t\in[0,T]}{\mathrm{ess\ sup}}\,\mathbb{E}\Vert u_{n}(t,\cdot)-\tilde{u}(t,\cdot)\Vert_{L_{1}(\mathbb{T}^{d})}\\
& \leq C \big(\mathbb{E}\Vert\xi_{n}-{\xi}\Vert_{L_{1}(\mathbb{T}^{d})}
+ \mathbb{E}\Vert\xi_{n}-\tilde{\xi}\Vert_{L_{1}(\mathbb{T}^{d})}
\big).
\end{aligned}
\]
Passing to the limit yields the desired estimate.
The uniqueness of the measure $\nu$ is verified in Proposition~\ref{prop:Uniquenessfornu}.
The proof is complete.
\end{proof}

\begin{rem}
In view of the proof of Theorem~\ref{thm:existence_main thm}, one may prove that the entropy solution also has the regularity $\nabla\Phi(u)\in L_2(\Omega_T;H^1(\mathbb{T}^d))$.
In order to control the length of the article, we show this rigorously only for the deterministic porous medium equation in Lemma~\ref{lem:more-regularity} when we discuss the relation between the entropy solution and the variational solution.
\end{rem}

\begin{rem}
In this paper, we confine our discussion to the lower obstacle problem.
Nonetheless, it is worth noting that the methods employed and the results obtained are equally applicable to the investigation of upper obstacle problems, leading to analogous conclusions.
We also believe that our methods can address the bilateral obstacle problem, 
albeit with potentially increased technical complexity.
\end{rem}

\begin{rem}
Our results can be extended to a general equation 
\begin{align*}
\mathrm{d}u & =\big[\Delta\Phi(u)+\nabla\cdot G(x,u)+f(t,x,u)\big]\mathrm{d}t+\nabla\cdot\sigma^{k}(x,u)\circ\mathrm{d}W_{t}^{k},
\end{align*}
where $\sigma$ and $G$ satisfy \cite[Assumption 2.3]{dareiotis2020nonlinear}.
\end{rem}

\section{$L_{1}$-stability\label{sec:-stability}}

In this section, we derive an $L_{1}$-estimate (Proposition~\ref{prop:uniqueentropy})
for the difference between two entropy solutions to the obstacle problem,
which naturally leads to the uniqueness result. The strategy of the proof
is based on Kruzhkov's doubling-variables method, and is also used in
proving the comparison principle, Proposition~\ref{prop:comparison}. 

At this stage, we assume a priori assumption that there exists an
entropy solution possessing an additional property that is called
the strong entropy condition in \cite{feng2008stochastic} or the
$(\star)$-property in \cite{dareiotis2019entropy,dareiotis2020ergodicity,du2024entropy}.
Fortunately, this property is satisfied by the solution that is constructed
in the proof of the existence under a stronger integrability condition of $\xi$. Therefore, the $L_{1}$-estimate is valid
actually for any entropy solutions, as stated in Theorem \ref{thm:unque_main thm}.
For the detail, refer to the proof of Theorem \ref{thm:unque_main thm} in Section \ref{sec:entropy formula}.

Now we introduce the $(\star)$-property. For simplicity, we denote
\[
\int_{t}\cdot:=\int_{0}^{T}\cdot\,\mathrm{d}t\quad\text{and}\quad\int_{x}\cdot:=\int_{\mathbb{T}^{d}}\cdot\,\mathrm{d}x.
\]
Given functions $(u,\tilde{u},\phi,h)\in L_{1}(\Omega_{T}\times\mathbb{T}^{d})\times L_{1}(\Omega_{T}\times\mathbb{T}^{d})\times C^{\infty}(\bar{Q}_{T}\times\bar{Q}_{T})\times C^{\infty}(\mathbb{R})$
and real number $z$, we define
\begin{align*}
H(t,x,z|\tilde{u},\phi,h) & :=\int_{0}^{T}\int_{y}h(\tilde{u}-z)\sigma_{y_{i}}^{ik}(y,\tilde{u})\phi(t,x,s,y)\mathrm{d}W_{s}^{k}\\
 & \quad-\int_{0}^{T}\int_{y}\llbracket\sigma_{ry_{i}}^{ik}h(\cdot-z)\rrbracket(y,\tilde{u})\phi(t,x,s,y)\mathrm{d}W_{s}^{k}\\
 & \quad-\int_{0}^{T}\int_{y}\llbracket\sigma_{r}^{ik}h(\cdot-z)\rrbracket(y,\tilde{u})\partial_{y_{i}}\phi(t,x,s,y)\mathrm{d}W_{s}^{k},
\end{align*}
and 
\begin{align*}
\mathcal{B}(u,\tilde{u}|\phi,h) & :=-\mathbb{E}\int_{t,x,s,y}\partial_{y_{i}x_{j}}\phi(t,x,s,y)\int_{\tilde{u}}^{u}\int_{r}^{\tilde{u}}h^{\prime}(\tilde{r}-r)\sigma_{r}^{jk}(x,r)\sigma_{r}^{ik}(y,\tilde{r})\mathrm{d}\tilde{r}\mathrm{d}r\\
 & \quad-\mathbb{E}\int_{t,x,s,y}\partial_{y_{i}}\phi(t,x,s,y)\int_{\tilde{u}}^{u}\int_{r}^{\tilde{u}}h^{\prime}(\tilde{r}-r)\sigma_{rx_{j}}^{jk}(x,r)\sigma_{r}^{ik}(y,\tilde{r})\mathrm{d}\tilde{r}\mathrm{d}r\\
 & \quad+\mathbb{E}\int_{t,x,s,y}\partial_{y_{i}}\phi(t,x,s,y)\int_{u}^{\tilde{u}}h^{\prime}(\tilde{r}-u)\sigma_{x_{j}}^{jk}(x,u)\sigma_{r}^{ik}(y,\tilde{r})\mathrm{d}\tilde{r}\\
 & \quad-\mathbb{E}\int_{t,x,s,y}\partial_{x_{j}}\phi(t,x,s,y)\int_{\tilde{u}}^{u}\int_{r}^{\tilde{u}}h^{\prime}(\tilde{r}-r)\sigma_{r}^{jk}(x,r)\sigma_{ry_{i}}^{ik}(y,\tilde{r})\mathrm{d}\tilde{r}\mathrm{d}r\\
 & \quad-\mathbb{E}\int_{t,x,s,y}\phi(t,x,s,y)\int_{\tilde{u}}^{u}\int_{r}^{\tilde{u}}h^{\prime}(\tilde{r}-r)\sigma_{rx_{j}}^{jk}(x,r)\sigma_{ry_{i}}^{ik}(y,\tilde{r})\mathrm{d}\tilde{r}\mathrm{d}r\\
 & \quad+\mathbb{E}\int_{t,x,s,y}\phi(t,x,s,y)\int_{u}^{\tilde{u}}h^{\prime}(\tilde{r}-u)\sigma_{x_{j}}^{jk}(x,u)\sigma_{ry_{i}}^{ik}(y,\tilde{r})\mathrm{d}\tilde{r}\\
 & \quad+\mathbb{E}\int_{t,x,s,y}\partial_{x_{j}}\phi(t,x,s,y)\int_{\tilde{u}}^{u}h^{\prime}(\tilde{u}-r)\sigma_{r}^{jk}(x,r)\sigma_{y_{i}}^{ik}(y,\tilde{u})\mathrm{d}r\\
 & \quad+\mathbb{E}\int_{t,x,s,y}\phi(t,x,s,y)\int_{\tilde{u}}^{u}h^{\prime}(\tilde{u}-r)\sigma_{rx_{j}}^{jk}(x,r)\sigma_{y_{i}}^{ik}(y,\tilde{u})\mathrm{d}r\\
 & \quad-\mathbb{E}\int_{t,x,s,y}\phi(t,x,s,y)h^{\prime}(\tilde{u}-u)\sigma_{x_{j}}^{jk}(x,u)\sigma_{y_{i}}^{ik}(y,\tilde{u}),
\end{align*}
where on the right-hand side $u=u(t,x)$ and $\tilde{u}=\tilde{u}(s,y)$.
Moreover, given a $C^{\infty}$-function $\rho:\mathbb{R}\rightarrow[0,2]$
with the properties
\[
\mathrm{supp}\,\rho\subset(0,1),\quad\int_{\mathbb{R}}\rho(x)\mathrm{d}x=1.
\]
we set $\rho_{\theta}(r):=\theta^{-1}\rho(\theta^{-1}r)$ for $\theta>0$. 
\begin{defn}[Definition 3.6 in \cite{dareiotis2020nonlinear}]
\label{def:star property} We say that a function $u\in L_{m+1}(\Omega_{T}\times\mathbb{T}^{d})$
has the ($\star$)-property, if for any $(\tilde{u},g,\varphi,h)\in L_{m+1}(\Omega_{T}\times\mathbb{T}^{d})\times C^{\infty}(\mathbb{T}^{d}\times\mathbb{T}^{d})\times C_{c}^{\infty}((0,T))\times C^{\infty}(\mathbb{R})$
with $h^{\prime}\in C_{c}^{\infty}(\mathbb{R})$ and for sufficiently
small $\theta>0$, it holds that $H(\cdot,\cdot,u(\cdot,\cdot)|\tilde{u},\phi{}_{\theta},h)\in L_{1}(\Omega_{T}\times\mathbb{T}^{d})$
with 
\[
\phi{}_{\theta}(t,x,s,y):=g(x,y)\rho_{\theta}(t-s)\varphi(\frac{t+s}{2}),\quad(t,x,s,y)\in\bar{Q}_{T}\times\bar{Q}_{T},
\]
and that
\[
\mathbb{E}\int_{t,x}H(t,x,u(t,x)|\tilde{u},\phi{}_{\theta},h)\leq C\theta^{1-\mu}+\mathcal{B}(u,\tilde{u}|\phi{}_{\theta},h)
\]
for a constant $C$ independent of $\theta$, where $\mu:=(3m+5)/(4m+4)<1$.
\end{defn}

Define $\varrho_{\varsigma}:=\rho_{\varsigma}^{\otimes d}$ for all
$\varsigma>0$. 
\begin{prop}
\label{prop:uniqueentropy}Let Assumptions \ref{assu:assumption for phi}
and \ref{assu:assumption for sigma} be satisfied, and $\psi\in C^{2}_x(\bar{Q}_{T})$.
Given $\psi(0,\cdot)\le\xi,\tilde{\xi}\in L_{m+1}(\Omega,\mathcal{F}_{0};L_{m+1}(\mathbb{T}^{d}))$,
we suppose that $(u,\nu)$ and $(\tilde{u},\tilde{\nu})$ are entropy
solutions to $\Pi_{\psi}(\Phi,f,\xi)$ and $\Pi_{\psi}(\Phi,f,\tilde{\xi})$,
respectively. Assume that $u$ has the ($\star$)-property. Then,
\begin{equation}
\underset{t\in[0,T]}{\mathrm{ess\ sup}}\,\Vert u(t,\cdot)-\tilde{u}(t,\cdot)\Vert_{L_{1}(\Omega\times\mathbb{T}^{d})}\leq C\Vert\xi-\tilde{\xi}\Vert_{L_{1}(\Omega\times\mathbb{T}^{d})},\label{eq:L1_for diff solution u}
\end{equation}
where the constant $C$ depends only on $d$, $K$, and $T$.
\end{prop}

\begin{proof}
For $\theta>0$ small enough and non-negative $\varphi\in C_{c}^{\infty}((0,T))$,
we introduce
\[
\phi_{\varsigma}(t,x,y)=\varrho_{\varsigma}(x-y)\varphi(t),\quad\phi_{\theta,\varsigma}(t,x,s,y):=\phi_{\varsigma}(\frac{t+s}{2},x,y)\rho_{\theta}(t-s).
\]
Furthermore, for each $\delta>0$, we define the function $\eta_{\delta}\in C^{2}(\mathbb{R})$
by
\[
\eta_{\delta}(0)=\eta_{\delta}^{\prime}(0)=0,\quad\eta_{\delta}^{\prime\prime}(r)=\rho_{\delta}(|r|).
\]
Thus, we have
\[
\big|\eta_{\delta}(r)-|r|\big|\leq\delta,\quad\mathrm{supp}\,\eta_{\delta}^{\prime\prime}\subset[-\delta,\delta],\quad\int_{\mathbb{R}}|\eta_{\delta}^{\prime\prime}(r)|\mathrm{d}r\leq2,\quad|\eta_{\delta}^{\prime\prime}|\leq2\delta^{-1}.
\]

Fix $(z,s,y)\in\mathbb{R}\times Q_{T}$. Since $(u,\nu)$ is an entropy
solution to $\Pi_{\psi}(\Phi,f,\xi)$, using the entropy inequality
of $(u,\nu)$ with $\big(\eta_{\delta}(\cdot-z),\phi_{\theta,\varsigma}(\cdot,\cdot,s,y)\big)$
instead of $\big(\eta,\phi\big)$, we have
\begin{align*}
 & -\int_{t,x}\eta_{\delta}(u-z)\partial_{t}\phi_{\theta,\varsigma}-\int_{Q_{T}}\eta_{\delta}^{\prime}(\psi(t,x)-z)\phi_{\theta,\varsigma}\nu(\mathrm{d}x\mathrm{d}t)\\
 & \leq\int_{t,x}\llbracket\Phi^{\prime}\eta_{\delta}^{\prime}(\cdot-z)\rrbracket(u)\Delta_{x}\phi_{\theta,\varsigma}\\
 & \quad+\int_{t,x}\llbracket a^{ij}\eta_{\delta}^{\prime}(\cdot-z)\rrbracket(u)\partial_{x_{i}x_{j}}\phi_{\theta,\varsigma}\\
 & \quad+\int_{t,x}\Big(\llbracket a_{x_{j}}^{ij}\eta_{\delta}^{\prime}(\cdot-z)+b_{r}^{i}\eta_{\delta}^{\prime}(\cdot-z)\rrbracket(x,u)-2\eta_{\delta}^{\prime}(u-z)b^{i}(x,u)\Big)\partial_{x_{i}}\phi_{\theta,\varsigma}\\
 & \quad+\int_{t,x}\Big(-\eta_{\delta}^{\prime}(u-z)b_{x_{i}}^{i}(x,u)+\llbracket b_{rx_{i}}^{i}\eta_{\delta}^{\prime}(\cdot-z)\rrbracket(x,u)\Big)\phi_{\theta,\varsigma}\\
 & \quad+\int_{t,x}\eta_{\delta}^{\prime}(u-z)f(t,x,u)\phi_{\theta,\varsigma}\\
 & \quad+\int_{t,x}\Big(\frac{1}{2}\eta_{\delta}^{\prime\prime}(u-z)\sum_{k=1}^{\infty}|\sigma_{x_{i}}^{ik}(x,u)|^{2}-\eta_{\delta}^{\prime\prime}(u-z)|\nabla_{x}\llbracket\sqrt{\Phi^{\prime}}\rrbracket(u)|^{2}\Big)\phi_{\theta,\varsigma}\\
 & \quad+\int_{0}^{T}\int_{x}\Big(\eta_{\delta}^{\prime}(u-z)\phi_{\theta,\varsigma}\sigma_{x_{i}}^{ik}(x,u)\\
 & \quad-\llbracket\sigma_{rx_{i}}^{ik}\eta_{\delta}^{\prime}(\cdot-z)\rrbracket(x,u)\phi_{\theta,\varsigma}-\llbracket\sigma_{r}^{ik}\eta_{\delta}^{\prime}(\cdot-z)\rrbracket(x,u)\partial_{x_{i}}\phi_{\theta,\varsigma}\Big)\mathrm{d}W_{t}^{k},
\end{align*}
where $u=u(t,x)$. Notice that each term in the inequality is continuous
in $(z,s,y)$. We take $z=\tilde{u}(s,y)$ by convolution and integrate
over $(s,y)\in Q_{T}$. By taking expectations, we have 

\begin{align} 
& -\mathbb{E}\int_{t,x,s,y}\eta_{\delta}(u-\tilde{u})\partial_{t}\phi_{\theta,\varsigma}-\mathbb{E}\int_{s,y}\int_{Q_{T}}\eta_{\delta}^{\prime}(\psi(t,x)-\tilde{u})\phi_{\theta,\varsigma}\nu(\mathrm{d}x\mathrm{d}t)\nonumber\\
 & \leq\mathbb{E}\int_{t,x,s,y}\llbracket\Phi^{\prime}\eta_{\delta}^{\prime}(\cdot-\tilde{u}\rrbracket(u)\Delta_{x}\phi_{\theta,\varsigma}\nonumber\\
 & \quad+\mathbb{E}\int_{t,x,s,y}\llbracket a^{ij}\eta_{\delta}^{\prime}(\cdot-\tilde{u})\rrbracket(u)\partial_{x_{i}x_{j}}\phi_{\theta,\varsigma}\nonumber\\
 & \quad+\mathbb{E}\int_{t,x,s,y}\llbracket a_{x_{j}}^{ij}\eta_{\delta}^{\prime}(\cdot-\tilde{u})+b_{r}^{i}\eta_{\delta}^{\prime}(\cdot-\tilde{u})\rrbracket(x,u)\partial_{x_{i}}\phi_{\theta,\varsigma}\nonumber\\
 & \quad-\mathbb{E}\int_{t,x,s,y}2\eta_{\delta}^{\prime}(u-\tilde{u})b^{i}(x,u)\partial_{x_{i}}\phi_{\theta,\varsigma}\nonumber\\
 & \quad+\mathbb{E}\int_{t,x,s,y}\Big(-\eta_{\delta}^{\prime}(u-\tilde{u})b_{x_{i}}^{i}(x,u)+\llbracket b_{rx_{i}}^{i}\eta_{\delta}^{\prime}(\cdot-\tilde{u})\rrbracket(x,u)\Big)\phi_{\theta,\varsigma} \label{eq:entropy doubling}\\
 & \quad+\mathbb{E}\int_{t,x,s,y}\eta_{\delta}^{\prime}(u-\tilde{u})f(t,x,u)\phi_{\theta,\varsigma}\nonumber\\
 & \quad+\mathbb{E}\int_{t,x,s,y}\Big(\frac{1}{2}\eta_{\delta}^{\prime\prime}(u-\tilde{u})\sum_{k=1}^{\infty}|\sigma_{x_{i}}^{ik}(x,u)|^{2}-\eta_{\delta}^{\prime\prime}(u-\tilde{u})|\nabla_{x}\llbracket\sqrt{\Phi^{\prime}}\rrbracket(u)|^{2}\Big)\phi_{\theta,\varsigma}\nonumber\\
&\quad+\mathbb{E}\int_{s,y}\Big[\int_{0}^{T}\int_{x}\Big(\eta_{\delta}^{\prime}(u-z)\phi_{\theta,\varsigma}\sigma_{x_{i}}^{ik}(x,u)\nonumber\\
 & \quad-\llbracket\sigma_{rx_{i}}^{ik}\eta_{\delta}^{\prime}(\cdot-z)\rrbracket(x,u)\phi_{\theta,\varsigma}-\llbracket\sigma_{r}^{ik}\eta_{\delta}^{\prime}(\cdot-z)\rrbracket(x,u)\partial_{x_{i}}\phi_{\theta,\varsigma}\Big)\mathrm{d}W_{t}^{k}\Big]_{z=\tilde{u}}, \nonumber
\end{align}
where $u=u(t,x)$ and $\tilde{u}=\tilde{u}(s,y)$. Since the integrand
of the stochastic integral on the right-hand side of (\ref{eq:entropy doubling})
vanish on $[0,s]$, the stochastic integral is zero. Estimate \eqref{eq:entropy doubling} is derived from the aspect of $(u,\nu)$.

Similarly, for each $(z,t,x)\in\mathbb{R}\times Q_{T}$, we apply
the entropy inequality of $(\tilde{u},\tilde{\nu})$ with $\eta(r):=\eta_{\delta}(r-z)$
and $\phi(s,y):=\phi_{\theta,\varsigma}(t,x,s,y)$. After substituting
$z=u(t,x)$ by convolution, integrating over $(t,x)\in Q_{T}$ and
taking expectations, we obtain a similar estimate from the aspect
of $(\tilde{u},\tilde{\nu})$. Notice that the stochastic integral in the estimate can be rewrite as 
$$
\mathbb{E}\int_{t,x}H(t,x,u(t,x)|\tilde{u},\phi_{\theta,\varsigma},\eta_\delta^\prime).
$$
Applying the ($\star$)-property of $u$,
this term is controlled by $C\theta^{1-\mu}+\mathcal{B}(u,\tilde{u}|\phi_{\theta,\varsigma},\eta_{\delta}^{\prime})$,
where $\mu=(3m+5)/(4m+4)<1$ and the constant $C$ is independent
of $\theta$. 

Adding these two estimates together, 
we have
\begin{equation}\label{eq:add doubling before theta}
-\mathbb{E}\int_{t,x,s,y}\eta_{\delta}(u-\tilde{u})\partial_{t}\phi_{\theta,\varsigma}\leq\sum_{\ell=1}^{7}J_{\ell},
\end{equation}
where
\begin{align*}
J_{1} & :=\mathbb{E}\int_{t,x,s,y}\eta_{\delta}^{\prime}(u-\tilde{u})f(t,x,u)\phi_{\theta,\varsigma}+\mathbb{E}\int_{s,y}\int_{Q_{T}}\eta_{\delta}^{\prime}(\psi(t,x)-\tilde{u})\phi_{\theta,\varsigma}\nu(\mathrm{d}x\mathrm{d}t)\\
 & \quad-\mathbb{E}\int_{t,x,s,y}\eta_{\delta}^{\prime}(u-\tilde{u})f(s,y,\tilde{u})\phi_{\theta,\varsigma}-\mathbb{E}\int_{t,x}\int_{Q_{T}}\eta_{\delta}^{\prime}(u-\psi(s,y))\phi_{\theta,\varsigma}\tilde{\nu}(\mathrm{d}y\mathrm{d}s),\\
J_{2} & :=\mathbb{E}\int_{t,x,s,y}\llbracket\Phi^{\prime}\eta_{\delta}^{\prime}(\cdot-\tilde{u})\rrbracket(u)\Delta_{x}\phi_{\theta,\varsigma}-\mathbb{E}\int_{t,x,s,y}\llbracket\Phi^{\prime}\eta_{\delta}^{\prime}(u-\cdot)\rrbracket(\tilde{u})\Delta_{y}\phi_{\theta,\varsigma},\\
J_{3} & :=\mathbb{E}\int_{t,x,s,y}\llbracket a^{ij}\eta_{\delta}^{\prime}(\cdot-\tilde{u})\rrbracket(x,u)\partial_{x_{i}x_{j}}\phi_{\theta,\varsigma}-\mathbb{E}\int_{t,x,s,y}\llbracket a^{ij}\eta_{\delta}^{\prime}(u-\cdot)\rrbracket(y,\tilde{u})\partial_{y_{i}y_{j}}\phi_{\theta,\varsigma}\\
 & \quad+\mathbb{E}\int_{t,x,s,y}\llbracket a_{x_{j}}^{ij}\eta_{\delta}^{\prime}(\cdot-\tilde{u})\rrbracket(x,u)\partial_{x_{i}}\phi_{\theta,\varsigma}-\mathbb{E}\int_{t,x,s,y}\llbracket a_{y_{j}}^{ij}\eta_{\delta}^{\prime}(u-\cdot)\rrbracket(y,\tilde{u})\partial_{y_{i}}\phi_{\theta,\varsigma}\\
 & \quad-2\mathbb{E}\int_{t,x,s,y}\eta_{\delta}^{\prime}(u-\tilde{u})b^{i}(x,u)\partial_{x_{i}}\phi_{\theta,\varsigma}+2\mathbb{E}\int_{t,x,s,y}\eta_{\delta}^{\prime}(u-\tilde{u})b^{i}(y,\tilde{u})\partial_{y_{i}}\phi_{\theta,\varsigma},\\
J_{4} & :=\mathbb{E}\int_{t,x,s,y}\llbracket b_{r}^{i}\eta_{\delta}^{\prime}(\cdot-\tilde{u})\rrbracket(x,u)\partial_{x_{i}}\phi_{\theta,\varsigma}-\mathbb{E}\int_{t,x,s,y}\llbracket b_{r}^{i}\eta_{\delta}^{\prime}(u-\cdot)\rrbracket(y,\tilde{u})\partial_{y_{i}}\phi_{\theta,\varsigma}\\
 & \quad-\mathbb{E}\int_{t,x,s,y}\eta_{\delta}^{\prime}(u-\tilde{u})b_{x_{i}}^{i}(x,u)\phi_{\theta,\varsigma}+\mathbb{E}\int_{t,x,s,y}\eta_{\delta}^{\prime}(u-\tilde{u})b_{y_{i}}^{i}(y,\tilde{u})\phi_{\theta,\varsigma}\\
 & \quad+\mathbb{E}\int_{t,x,s,y}\llbracket b_{rx_{i}}^{i}\eta_{\delta}^{\prime}(\cdot-\tilde{u})\rrbracket(x,u)\phi_{\theta,\varsigma}-\mathbb{E}\int_{t,x,s,y}\llbracket b_{ry_{i}}^{i}\eta_{\delta}^{\prime}(u-\cdot)\rrbracket(y,\tilde{u})\phi_{\theta,\varsigma},\\
J_{5} & :=\mathbb{E}\int_{t,x,s,y}\frac{1}{2}\eta_{\delta}^{\prime\prime}(u-\tilde{u})\sum_{k=1}^{\infty}\Big(|\sigma_{x_{i}}^{ik}(x,u)|^{2}+|\sigma_{y_{i}}^{ik}(y,\tilde{u})|^{2}\Big)\phi_{\theta,\varsigma},\\
J_{6} & :=\mathbb{E}\int_{t,x,s,y}-\eta_{\delta}^{\prime\prime}(u-\tilde{u})\Big(|\nabla_{x}\llbracket\sqrt{\Phi^{\prime}}\rrbracket(u)|^{2}+|\nabla_{y}\llbracket\sqrt{\Phi^{\prime}}\rrbracket(\tilde{u})|^{2}\Big)\phi_{\theta,\varsigma},\\
J_{7} & :=C\theta^{1-\mu}+\mathcal{B}(u,\tilde{u}|\phi_{\theta,\varsigma},\eta_{\delta}^{\prime}),
\end{align*}
and $u=u(t,x)$ and $\tilde{u}=\tilde{u}(s,y)$. For the terms involving the measures $\nu$ and $\tilde{\nu}$ in $J_{1}$, with the fact that $u,\tilde{u}\geq\psi$ and $\eta_{\delta}^{\prime\prime}\geq0$,
we have
\begin{align*}
 & \mathbb{E}\int_{s,y}\int_{Q_{T}}\phi_{\theta,\varsigma}\eta_{\delta}^{\prime}(\psi(t,x)-\tilde{u}(s,y))\nu(\mathrm{d}x\mathrm{d}t)\\
 & \quad-\mathbb{E}\int_{t,x}\int_{Q_{T}}\phi_{\theta,\varsigma}\eta_{\delta}^{\prime}(u(t,x)-\psi(s,y))\tilde{\nu}(\mathrm{d}y\mathrm{d}s)\\
 & \leq\mathbb{E}\int_{s,y}\int_{Q_{T}}\phi_{\theta,\varsigma}\eta_{\delta}^{\prime}(\psi(t,x)-\psi(s,y))\nu(\mathrm{d}x\mathrm{d}t)\\
 & \quad-\mathbb{E}\int_{t,x}\int_{Q_{T}}\phi_{\theta,\varsigma}\eta_{\delta}^{\prime}(\psi(t,x)-\psi(s,y))\tilde{\nu}(\mathrm{d}y\mathrm{d}s)\\
 & \leq\mathbb{E}\int_{s,y}\int_{Q_{T}}\phi_{\theta,\varsigma}\eta_{\delta}^{\prime}(\psi(t,x)-\psi(t,y))\nu(\mathrm{d}x\mathrm{d}t)\\
 & \quad-\mathbb{E}\int_{t,x}\int_{Q_{T}}\phi_{\theta,\varsigma}\eta_{\delta}^{\prime}(\psi(s,x)-\psi(s,y))\tilde{\nu}(\mathrm{d}y\mathrm{d}s)\\
 & \quad+\int_{s,y}\int_{Q_T}\Vert \eta_\delta^{\prime\prime}\Vert_{L_{\infty}}\phi_{\theta,\varsigma}|\psi(s,y)-\psi(t,y)|\phi_{\theta,\varsigma}\nu(\mathrm{d}x\mathrm{d}t)\\
 & \quad+\int_{t,x}\int_{Q_T}\Vert \eta_\delta^{\prime\prime}\Vert_{L_{\infty}}\phi_{\theta,\varsigma}|\psi(t,x)-\psi(s,x)|\phi_{\theta,\varsigma}\tilde{\nu}(\mathrm{d}y\mathrm{d}s)\\ 
 & \leq\mathbb{E}\int_{s,y}\int_{Q_{T}}\phi_{\theta,\varsigma}\eta_{\delta}^{\prime}(\psi(t,x)-\psi(t,y))\nu(\mathrm{d}x\mathrm{d}t)\\
 & \quad-\mathbb{E}\int_{t,x}\int_{Q_{T}}\phi_{\theta,\varsigma}\eta_{\delta}^{\prime}(\psi(s,x)-\psi(s,y))\tilde{\nu}(\mathrm{d}y\mathrm{d}s)\\
 & \quad+C\delta^{-1}\varpi(\theta)\Big[\mathbb{E}\nu(Q_T)+\mathbb{E}\tilde{\nu}(Q_T)\Big],
\end{align*}
where $\varpi$ is the modulus of continuity of $\psi$. Taking $\theta\rightarrow0^{+}$ in \eqref{eq:add doubling before theta}
and using \cite[Proposition 3.5]{dareiotis2019entropy}, with the symmetry of $\eta_{\delta}$,
we have
\begin{equation}
-\mathbb{E}\int_{t,x,y}\eta_{\delta}(u-\tilde{u})\partial_{t}\phi_{\varsigma}\leq\sum_{\ell=1}^{7}N_{\ell},\label{eq:inequalitywitht-1-1}
\end{equation}
where
\begin{align*}
N_{1} & :=\mathbb{E}\int_{t,x,y}\eta_{\delta}^{\prime}(u-\tilde{u})f(t,x,u)\phi_{\varsigma}+\mathbb{E}\int_{y}\int_{Q_{T}}\eta_{\delta}^{\prime}(\psi(t,x)-\psi(t,y))\phi_{\varsigma}\nu(\mathrm{d}x\mathrm{d}t)\\
 & \quad-\mathbb{E}\int_{t,x,y}\eta_{\delta}^{\prime}(u-\tilde{u})f(t,y,\tilde{u})\phi_{\varsigma}-\mathbb{E}\int_{x}\int_{Q_{T}}\eta_{\delta}^{\prime}(\psi(t,x)-\psi(t,y))\phi_{\varsigma}\tilde{\nu}(\mathrm{d}y\mathrm{d}t),\\
N_{2} & :=\mathbb{E}\int_{t,x,y}\llbracket\Phi^{\prime}\eta_{\delta}^{\prime}(\cdot-\tilde{u})\rrbracket(u)\Delta_{x}\phi_{\varsigma}-\mathbb{E}\int_{t,x,y}\llbracket\Phi^{\prime}\eta_{\delta}^{\prime}(u-\cdot)\rrbracket(\tilde{u})\Delta_{y}\phi_{\varsigma},\\
N_{3} & :=\mathbb{E}\int_{t,x,y}\llbracket a^{ij}\eta_{\delta}^{\prime}(\cdot-\tilde{u})\rrbracket(x,u)\partial_{x_{i}x_{j}}\phi_{\varsigma}-\mathbb{E}\int_{t,x,y}\llbracket a^{ij}\eta_{\delta}^{\prime}(u-\cdot)\rrbracket(y,\tilde{u})\partial_{y_{i}y_{j}}\phi_{\varsigma}\\
 & \quad+\mathbb{E}\int_{t,x,y}\llbracket a_{x_{j}}^{ij}\eta_{\delta}^{\prime}(\cdot-\tilde{u})\rrbracket(x,u)\partial_{x_{i}}\phi_{\varsigma}-\mathbb{E}\int_{t,x,y}\llbracket a_{y_{j}}^{ij}\eta_{\delta}^{\prime}(u-\cdot)\rrbracket(y,\tilde{u})\partial_{y_{i}}\phi_{\varsigma}\\
 & \quad-2\mathbb{E}\int_{t,x,y}\eta_{\delta}^{\prime}(u-\tilde{u})b^{i}(x,u)\partial_{x_{i}}\phi_{\varsigma}+2\mathbb{E}\int_{t,x,y}\eta_{\delta}^{\prime}(u-\tilde{u})b^{i}(y,\tilde{u})\partial_{y_{i}}\phi_{\varsigma},\\
N_{4} & :=\mathbb{E}\int_{t,x,y}\llbracket b_{r}^{i}\eta_{\delta}^{\prime}(\cdot-\tilde{u})\rrbracket(x,u)\partial_{x_{i}}\phi_{\varsigma}-\mathbb{E}\int_{t,x,y}\llbracket b_{r}^{i}\eta_{\delta}^{\prime}(u-\cdot)\rrbracket(y,\tilde{u})\partial_{y_{i}}\phi_{\varsigma}\\
 & \quad-\mathbb{E}\int_{t,x,y}\eta_{\delta}^{\prime}(u-\tilde{u})b_{x_{i}}^{i}(x,u)\phi_{\varsigma}+\mathbb{E}\int_{t,x,y}\eta_{\delta}^{\prime}(u-\tilde{u})b_{y_{i}}^{i}(y,\tilde{u})\phi_{\varsigma}\\
 & \quad+\mathbb{E}\int_{t,x,y}\llbracket b_{rx_{i}}^{i}\eta_{\delta}^{\prime}(\cdot-\tilde{u})\rrbracket(x,u)\phi_{\varsigma}-\mathbb{E}\int_{t,x,y}\llbracket b_{ry_{i}}^{i}\eta_{\delta}^{\prime}(u-\cdot)\rrbracket(y,\tilde{u})\phi_{\varsigma},\\
N_{5} & :=\mathbb{E}\int_{t,x,y}\frac{1}{2}\eta_{\delta}^{\prime\prime}(u-\tilde{u})\sum_{k=1}^{\infty}\Big(|\sigma_{x_{i}}^{ik}(x,u)|^{2}+|\sigma_{y_{i}}^{ik}(y,\tilde{u})|^{2}\Big)\phi_{\varsigma},\\
N_{6} & :=\mathbb{E}\int_{t,x,y}-\eta_{\delta}^{\prime\prime}(u-\tilde{u})\Big(|\nabla_{x}\llbracket\sqrt{\Phi^{\prime}}\rrbracket(u)|^{2}+|\nabla_{y}\llbracket\sqrt{\Phi^{\prime}}\rrbracket(\tilde{u})|^{2}\Big)\phi_{\varsigma},\\
N_{7} & =\sum_{l=1}^{9}N_{7,l}:=-\mathbb{E}\int_{t,x,y}\partial_{y_{i}x_{j}}\phi_{\varsigma}\int_{\tilde{u}}^{u}\int_{r}^{\tilde{u}}\eta_{\delta}^{\prime\prime}(r-\tilde{r})\sigma_{r}^{jk}(x,r)\sigma_{r}^{ik}(y,\tilde{r})\mathrm{d}\tilde{r}\mathrm{d}r\\
 & \quad-\mathbb{E}\int_{t,x,y}\partial_{y_{i}}\phi_{\varsigma}\int_{\tilde{u}}^{u}\int_{r}^{\tilde{u}}\eta_{\delta}^{\prime\prime}(r-\tilde{r})\sigma_{rx_{j}}^{jk}(x,r)\sigma_{r}^{ik}(y,\tilde{r})\mathrm{d}\tilde{r}\mathrm{d}r\\
 & \quad+\mathbb{E}\int_{t,x,y}\partial_{y_{i}}\phi_{\varsigma}\int_{u}^{\tilde{u}}\eta_{\delta}^{\prime\prime}(u-\tilde{r})\sigma_{x_{j}}^{jk}(x,u)\sigma_{r}^{ik}(y,\tilde{r})\mathrm{d}\tilde{r}\\
 & \quad-\mathbb{E}\int_{t,x,y}\partial_{x_{j}}\phi_{\varsigma}\int_{\tilde{u}}^{u}\int_{r}^{\tilde{u}}\eta_{\delta}^{\prime\prime}(r-\tilde{r})\sigma_{r}^{jk}(x,r)\sigma_{ry_{i}}^{ik}(y,\tilde{r})\mathrm{d}\tilde{r}\mathrm{d}r\\
 & \quad-\mathbb{E}\int_{t,x,y}\phi_{\varsigma}\int_{\tilde{u}}^{u}\int_{r}^{\tilde{u}}\eta_{\delta}^{\prime\prime}(r-\tilde{r})\sigma_{rx_{j}}^{jk}(x,r)\sigma_{ry_{i}}^{ik}(y,\tilde{r})\mathrm{d}\tilde{r}\mathrm{d}r\\
 & \quad+\mathbb{E}\int_{t,x,y}\phi_{\varsigma}\int_{u}^{\tilde{u}}\eta_{\delta}^{\prime\prime}(u-\tilde{r})\sigma_{x_{j}}^{jk}(x,u)\sigma_{ry_{i}}^{ik}(y,\tilde{r})\mathrm{d}\tilde{r}\\
 & \quad-\mathbb{E}\int_{t,x,y}\partial_{x_{j}}\phi_{\varsigma}\int_{\tilde{u}}^{u}\eta_{\delta}^{\prime\prime}(r-\tilde{u})\sigma_{r}^{jk}(x,r)\sigma_{y_{i}}^{ik}(y,\tilde{u})\mathrm{d}r\\
 & \quad+\mathbb{E}\int_{t,x,y}\phi_{\varsigma}\int_{\tilde{u}}^{u}\eta_{\delta}^{\prime\prime}(r-\tilde{u})\sigma_{rx_{j}}^{jk}(x,r)\sigma_{y_{i}}^{ik}(y,\tilde{u})\mathrm{d}r\\
 & \quad-\mathbb{E}\int_{t,x,y}\phi_{\varsigma}\eta_{\delta}^{\prime\prime}(u-\tilde{u})\sigma_{x_{j}}^{jk}(x,u)\sigma_{y_{i}}^{ik}(y,\tilde{u}),
\end{align*}
and $u=u(t,x)$ and $\tilde{u}=\tilde{u}(t,y)$. 

The term $N_{1}$ involving the measures $\nu$ and $\tilde{\nu}$ is derived from the obstacle problem.
We notice that
\begin{align}
 & N_{1}=\mathbb{E}\int_{t,x,y}\eta_{\delta}^{\prime}(u-\tilde{u})\Big(f(t,x,u)-f(t,x,\tilde{u})\Big)\phi_{\varsigma}\nonumber\\
 & \quad+\mathbb{E}\int_{t,x,y}\eta_{\delta}^{\prime}(u-\tilde{u})\Big(f(t,x,\tilde{u})-f(t,y,\tilde{u})\Big)\phi_{\varsigma}\nonumber\\
 & \quad+\mathbb{E}\int_{y}\int_{Q_{T}}\phi_{\varsigma}\eta_{\delta}^{\prime}(\psi(t,x)-\psi(t,y))\nu(\mathrm{d}x\mathrm{d}t)\nonumber\\
 & \quad-\mathbb{E}\int_{x}\int_{Q_{T}}\phi_{\varsigma}\eta_{\delta}^{\prime}(\psi(t,x)-\psi(t,y))\tilde{\nu}(\mathrm{d}y\mathrm{d}t)\label{eq:FGdiffL1-1-1-1}\\
 & \leq C\varsigma+C\mathbb{E}\int_{t,x,y}|u-\tilde{u}|\varrho_{\varsigma}(x-y)\varphi(t)\nonumber\\
 & \quad+\mathbb{E}\int_{y}\int_{Q_{T}}\phi_{\varsigma}\eta_{\delta}^{\prime}(\psi(t,x)-\psi(t,y))\nu(\mathrm{d}x\mathrm{d}t)\nonumber\\
 & \quad-\mathbb{E}\int_{x}\int_{Q_{T}}\phi_{\varsigma}\eta_{\delta}^{\prime}(\psi(t,x)-\psi(t,y))\tilde{\nu}(\mathrm{d}y\mathrm{d}t).\nonumber
\end{align}
Thanks to $\psi\in C_{x}^{2}(\bar{Q}_{T})$, we apply the Lagrange's
mean value theorem to compute
\begin{equation}
\begin{aligned} & \mathbb{E}\int_{y}\int_{Q_{T}}\phi_{\varsigma}\eta_{\delta}^{\prime}(\psi(t,x)-\psi(t,y))\nu(\mathrm{d}x\mathrm{d}t)\\
 & \quad-\mathbb{E}\int_{x}\int_{Q_{T}}\phi_{\varsigma}\eta_{\delta}^{\prime}(\psi(t,x)-\psi(t,y))\tilde{\nu}(\mathrm{d}y\mathrm{d}t)\\
 & =\mathbb{E}\int_{y}\int_{Q_{T}}\phi_{\varsigma}\eta_{\delta}^{\prime}\Big(-\partial_{x_{i}}\psi(t,x)(y_{i}-x_{i})-\frac{1}{2}\partial_{x_{i}x_{j}}\psi(t,\Theta)(y_{i}-x_{i})(y_{j}-x_{j})\Big)\nu(\mathrm{d}x\mathrm{d}t)\\
 & \quad-\mathbb{E}\int_{x}\int_{Q_{T}}\phi_{\varsigma}\eta_{\delta}^{\prime}\Big(\partial_{y_{i}}\psi(t,y)(x_{i}-y_{i})+\frac{1}{2}\partial_{y_{i}y_{j}}\psi(t,\tilde{\Theta})(y_{i}-x_{i})(y_{j}-x_{j})\Big)\tilde{\nu}(\mathrm{d}y\mathrm{d}t),
\end{aligned}
\label{eq:L1_nu}
\end{equation}
where $\Theta$ and $\tilde{\Theta}$ are points on the segment between $x$ and $y$. Keeping
in mind that $\eta_{\delta}^{\prime}$ is odd, we observe that 
\[
\int_{y}\varrho_{\varsigma}(x-y)\eta_{\delta}^{\prime}\big(-\partial_{x_{i}}\psi(t,x)(y_{i}-x_{i})\big)=0
\]
for each $(t,x)\in\bar{Q}_T$, and
\[
\int_{x}\varrho_{\varsigma}(x-y)\eta_{\delta}^{\prime}\big(\partial_{y_{i}}\psi(t,y)(x_{i}-y_{i})\big)=0
\]
for each $(t,y)\in\bar{Q}_T$. Combining with (\ref{eq:L1_nu}), we can continue to compute
\begin{align*}
 & \mathbb{E}\int_{y}\int_{Q_{T}}\phi_{\varsigma}\eta_{\delta}^{\prime}(\psi(t,x)-\psi(t,y))\nu(\mathrm{d}x\mathrm{d}t)\\
 & \quad-\mathbb{E}\int_{x}\int_{Q_{T}}\phi_{\varsigma}\eta_{\delta}^{\prime}(\psi(t,x)-\psi(t,y))\tilde{\nu}(\mathrm{d}y\mathrm{d}t)\\
 & =\mathbb{E}\int_{y}\int_{Q_{T}}\phi_{\varsigma}\Big[\eta_{\delta}^{\prime}(-\partial_{x_{i}}\psi(t,x)(y_{i}-x_{i})-\frac{1}{2}\partial_{x_{i}x_{j}}\psi(t,\Theta)(y_{i}-x_{i})(y_{j}-x_{j}))\\
 & \quad\qquad-\eta_{\delta}^{\prime}(-\partial_{x_{i}}\psi(t,x)(y_{i}-x_{i}))\Big]\nu(\mathrm{d}x\mathrm{d}t)\\
 & \quad-\mathbb{E}\int_{x}\int_{Q_{T}}\phi_{\varsigma}\Big[\eta_{\delta}^{\prime}(\partial_{y_{i}}\psi(t,y)(x_{i}-y_{i})+\frac{1}{2}\partial_{y_{i}y_{j}}\psi(t,\tilde{\Theta})(y_{i}-x_{i})(y_{j}-x_{j}))\\
 & \quad\qquad-\eta_{\delta}^{\prime}(\partial_{y_{i}}\psi(t,y)(x_{i}-y_{i}))\Big]\tilde{\nu}(\mathrm{d}y\mathrm{d}t)\\
 & \leq C\mathbb{E}\int_{y}\int_{Q_{T}}\phi_{\varsigma}\Vert\eta_{\delta}^{\prime\prime}\Vert_{L_{\infty}}\cdot\big|\partial_{x_{i}x_{j}}\psi(t,\Theta)(y_{i}-x_{i})(y_{j}-x_{j})\big|\nu(\mathrm{d}x\mathrm{d}t)\\
 & \quad+C\mathbb{E}\int_{x}\int_{Q_{T}}\phi_{\varsigma}\Vert\eta_{\delta}^{\prime\prime}\Vert_{L_{\infty}}\cdot\big|\partial_{y_{i}y_{j}}\psi(t,\tilde{\Theta})(y_{i}-x_{i})(y_{j}-x_{j})\big|\tilde{\nu}(\mathrm{d}y\mathrm{d}t)\\
 & \leq C\delta^{-1}\varsigma^{2}\Big[\mathbb{E}\int_{y}\int_{Q_{T}}\phi_{\varsigma}\nu(\mathrm{d}x\mathrm{d}t)+\mathbb{E}\int_{x}\int_{Q_{T}}\phi_{\varsigma}\tilde{\nu}(\mathrm{d}y\mathrm{d}t)\Big]\\
 & \leq C\delta^{-1}\varsigma^{2}\Big[\mathbb{E}\nu(Q_{T})+\mathbb{E}\tilde{\nu}(Q_{T})\Big]\leq C\delta^{-1}\varsigma^{2}.
\end{align*}
Therefore, we have
\[
N_{1}\leq C(\varsigma+\delta^{-1}\varsigma^{2})+C\mathbb{E}\int_{t,x,y}|u-\tilde{u}|\varrho_{\varsigma}(x-y)\varphi(t).
\]

The method to deal with the terms $N_{i}$, $i=2,\ldots,7$ is similar to
that in \cite{dareiotis2020nonlinear,du2024entropy}. With Assumption
\ref{assu:assumption for sigma}, we have
\begin{align*}
N_{5}+N_{7,9} & =\mathbb{E}\int_{t,x,y}\frac{1}{2}\eta_{\delta}^{\prime\prime}(u-\tilde{u})\sum_{k=1}^{\infty}\Big(\sigma_{x_{i}}^{ik}(x,u)-\sigma_{y_{i}}^{ik}(y,\tilde{u})\Big)^{2}\phi_{\varsigma}\\
 & \leq C\varsigma^{2}\delta^{-1}+C\delta.
\end{align*}
For the estimate of $N_{2}=N_{2,1}+N_{2,2}$, since $\partial_{x_{i}}\phi_{\varsigma}=-\partial_{y_{i}}\phi_{\varsigma}$,
we have
\begin{align*}
N_{2,1} & :=\mathbb{E}\int_{t,x,y}\llbracket\Phi^{\prime}\eta_{\delta}^{\prime}(\cdot-\tilde{u})\rrbracket(u)\Delta_{x}\phi_{\varsigma}\\
 & =-\mathbb{E}\int_{t,x,y}\partial_{x_{i}y_{i}}\phi_{\varsigma}\int_{\tilde{u}}^{u}\eta_{\delta}^{\prime}(r-\tilde{u})\Phi^{\prime}(r)\mathrm{d}r\\
 & =-\mathbb{E}\int_{t,x,y}\mathbf{1}_{\{\tilde{u}\leq u\}}\partial_{x_{i}y_{i}}\phi_{\varsigma}\int_{\tilde{u}}^{u}\int_{\tilde{u}}^{u}\mathbf{1}_{\{\tilde{r}\leq r\}}\eta_{\delta}^{\prime\prime}(r-\tilde{r})\Phi^{\prime}(r)\mathrm{d}\tilde{r}\mathrm{d}r\\
 & \quad-\mathbb{E}\int_{t,x,y}\mathbf{1}_{\{u\leq\tilde{u}\}}\partial_{x_{i}y_{i}}\phi_{\varsigma}\int_{u}^{\tilde{u}}\int_{u}^{\tilde{u}}\mathbf{1}_{\{r\leq\tilde{r}\}}\eta_{\delta}^{\prime\prime}(r-\tilde{r})\Phi^{\prime}(r)\mathrm{d}\tilde{r}\mathrm{d}r.
\end{align*}
Similarly, we have
\begin{align*}
N_{2,2} & :=-\mathbb{E}\int_{t,x,y}\llbracket\Phi^{\prime}\eta_{\delta}^{\prime}(u-\cdot)\rrbracket(\tilde{u})\Delta_{y}\phi_{\varsigma}\\
 & =\mathbb{E}\int_{t,x,y}\partial_{x_{i}y_{i}}\phi_{\varsigma}\int_{u}^{\tilde{u}}\eta_{\delta}^{\prime}(u-\tilde{r})\Phi^{\prime}(\tilde{r})\mathrm{d}\tilde{r}\\
 & =-\mathbb{E}\int_{t,x,y}\mathbf{1}_{\{\tilde{u}\leq u\}}\partial_{x_{i}y_{i}}\phi_{\varsigma}\int_{\tilde{u}}^{u}\int_{\tilde{u}}^{u}\mathbf{1}_{\{\tilde{r}\leq r\}}\eta_{\delta}^{\prime\prime}(r-\tilde{r})\Phi^{\prime}(\tilde{r})\mathrm{d}r\mathrm{d}\tilde{r}\\
 & \quad-\mathbb{E}\int_{t,x,y}\mathbf{1}_{\{u\leq\tilde{u}\}}\partial_{x_{i}y_{i}}\phi_{\varsigma}\int_{u}^{\tilde{u}}\int_{u}^{\tilde{u}}\mathbf{1}_{\{r\leq\tilde{r}\}}\eta_{\delta}^{\prime\prime}(r-\tilde{r})\Phi^{\prime}(\tilde{r})\mathrm{d}r\mathrm{d}\tilde{r}.
\end{align*}
Notice that
\begin{align*}
N_{6} & :=-\mathbb{E}\int_{t,x,y}\eta_{\delta}^{\prime\prime}(u-\tilde{u})\Big(|\nabla_{x}\llbracket\sqrt{\Phi^{\prime}}\rrbracket(u)|^{2}+|\nabla_{y}\llbracket\sqrt{\Phi^{\prime}}\rrbracket(\tilde{u})|^{2}\Big)\phi_{\varsigma}\\
 & \leq-2\mathbb{E}\int_{t,x,y}\eta_{\delta}^{\prime\prime}(u-\tilde{u})\nabla_{x}\llbracket\sqrt{\Phi^{\prime}}\rrbracket(u)\cdot\nabla_{y}\llbracket\sqrt{\Phi^{\prime}}\rrbracket(\tilde{u})\phi_{\varsigma}\\
 & =-2\mathbb{E}\int_{t,x,y}\phi_{\varsigma}\partial_{x_{i}}\llbracket\sqrt{\Phi^{\prime}}\rrbracket(u)\partial_{y_{i}}\int_{0}^{\tilde{u}}\eta_{\delta}^{\prime\prime}(u-\tilde{r})\sqrt{\Phi^{\prime}}(\tilde{r})\mathrm{d}\tilde{r}\\
 & =2\mathbb{E}\int_{t,x,y}\partial_{y_{i}}\phi_{\varsigma}\partial_{x_{i}}\llbracket\sqrt{\Phi^{\prime}}\rrbracket(u)\int_{u}^{\tilde{u}}\eta_{\delta}^{\prime\prime}(u-\tilde{r})\sqrt{\Phi^{\prime}}(\tilde{r})\mathrm{d}\tilde{r}.
\end{align*}
 Using \cite[Remark 3.1]{dareiotis2019entropy}, we have
\begin{align*}
N_{6} & \leq2\mathbb{E}\int_{t,x,y}\partial_{x_{i}y_{i}}\phi_{\varsigma}\int_{u}^{\tilde{u}}\int_{r}^{\tilde{u}}\eta_{\delta}^{\prime\prime}(r-\tilde{r})\sqrt{\Phi^{\prime}}(\tilde{r})\sqrt{\Phi^{\prime}}(r)\mathrm{d}\tilde{r}\mathrm{d}r\\
 & =2\mathbb{E}\int_{t,x,y}\mathbf{1}_{\{\tilde{u}\leq u\}}\partial_{x_{i}y_{i}}\phi_{\varsigma}\int_{\tilde{u}}^{u}\int_{\tilde{u}}^{u}\mathbf{1}_{\{\tilde{r}\leq r\}}\eta_{\delta}^{\prime\prime}(r-\tilde{r})\sqrt{\Phi^{\prime}}(\tilde{r})\sqrt{\Phi^{\prime}}(r)\mathrm{d}\tilde{r}\mathrm{d}r\\
 & \quad+2\mathbb{E}\int_{t,x,y}\mathbf{1}_{\{u\leq\tilde{u}\}}\partial_{x_{i}y_{i}}\phi_{\varsigma}\int_{u}^{\tilde{u}}\int_{u}^{\tilde{u}}\mathbf{1}_{\{r\leq\tilde{r}\}}\eta_{\delta}^{\prime\prime}(r-\tilde{r})\sqrt{\Phi^{\prime}}(\tilde{r})\sqrt{\Phi^{\prime}}(r)\mathrm{d}\tilde{r}\mathrm{d}r.
\end{align*}
Then,
\[
N_{2}+N_{6}\leq\mathbb{E}\int_{t,x,y}|\partial_{x_{i}y_{i}}\phi_{\varsigma}|\int_{\tilde{u}}^{u}\int_{\tilde{u}}^{u}\eta_{\delta}^{\prime\prime}(r-\tilde{r})|\sqrt{\Phi^{\prime}}(r)-\sqrt{\Phi^{\prime}}(\tilde{r})|^{2}\mathrm{d}\tilde{r}\mathrm{d}r.
\]
Based on the estimates on $|\sqrt{\Phi^{\prime}}(r)-\sqrt{\Phi^{\prime}}(\tilde{r})|^{2}$
in the proof of~\cite[Theorem 4.1]{dareiotis2019entropy} which using
Assumption~\ref{assu:assumption for phi}, we have
\[
N_{2}+N_{6}\leq C\varsigma^{-2}\delta^{2\alpha}\mathbb{E}(1+\Vert u\Vert_{L_{m}(Q_{T})}^{m}+\Vert\tilde{u}\Vert_{L_{m}(Q_{T})}^{m}),
\]
where $\alpha\in(0,1\land(m/2))$. On the other hand, note that
\begin{align*}
N_{3} & =\sum_{i=1}^{3}N_{3,i}:=-\mathbb{E}\int_{t,x,y}\partial_{y_{i}x_{j}}\phi_{\varsigma}\int_{\tilde{u}}^{u}\int_{\tilde{u}}^{u}\big(a^{ij}(x,r)+a^{ij}(y,\tilde{r})\big)\eta_{\delta}^{\prime\prime}(r-\tilde{r})\mathrm{d}\tilde{r}\mathrm{d}r\\
 & \quad+\mathbb{E}\int_{t,x,y}\partial_{x_{i}}\phi_{\varsigma}\int_{\tilde{u}}^{u}a_{x_{j}}^{ij}(x,r)\eta_{\delta}^{\prime}(r-\tilde{u})\mathrm{d}r-\mathbb{E}\int_{t,x,y}\partial_{y_{i}}\phi_{\varsigma}\int_{u}^{\tilde{u}}a_{y_{j}}^{ij}(y,\tilde{r})\eta_{\delta}^{\prime}(u-\tilde{r})\mathrm{d}\tilde{r}\\
 & \quad-2\mathbb{E}\int_{t,x,y}\eta_{\delta}^{\prime}(u-\tilde{u})b^{i}(x,u)\partial_{x_{i}}\phi_{\varsigma}+2\mathbb{E}\int_{t,x,y}\eta_{\delta}^{\prime}(u-\tilde{u})b^{i}(y,\tilde{u})\partial_{y_{i}}\phi_{\varsigma}.
\end{align*}
If $|r-\tilde{r}|\leq\delta$, from the definition of $a^{ij}$ and
Assumption \ref{assu:assumption for sigma}, we have
\begin{align*}
 & \partial_{y_{i}x_{j}}\phi_{\varsigma}(a^{ij}(x,r)+a^{ij}(y,\tilde{r})-{\sum_{k=1}^{\infty}}\sigma_{r}^{ik}(x,r)\sigma_{r}^{jk}(y,\tilde{r}))\\
 & =\frac{1}{2}\partial_{y_{i}x_{j}}\phi_{\varsigma}\sum_{k=1}^{\infty}\big(\sigma_{r}^{ik}(x,r)-\sigma_{r}^{ik}(y,\tilde{r})\big)\big(\sigma_{r}^{jk}(x,r)-\sigma_{r}^{jk}(y,\tilde{r})\big)\\
 & \leq C\varphi(t)(\varsigma^{2}+\delta^{2})\sum_{i,j=1}^{d}|\partial_{y_{i}x_{j}}\varrho_{\varsigma}(x-y)|.
\end{align*}
Therefore,
\begin{align*}
 & N_{3,1}+N_{7,1}\\
 & \leq C\mathbb{E}\int_{t,x,y}\varphi(t)\varsigma^{2}\sum_{i,j=1}^{d}|\partial_{y_{i}x_{j}}\varrho_{\varsigma}(x-y)|\cdot|u-\tilde{u}|\\
 & \quad+C\delta^{2}\varsigma^{-2}\mathbb{E}[\Vert u\Vert_{L_{1}(Q_{T})}+\Vert\tilde{u}\Vert_{L_{1}(Q_{T})}].
\end{align*}
To estimate $N_{3,2}+N_{7,2}+N_{7,4}$, we have
\begin{align*}
 & N_{3,2}+N_{7,2}+N_{7,4}\\
 & =-\mathbb{E}\int_{t,x,y}\mathbf{1}_{\{\tilde{u}\leq u\}}\partial_{y_{i}}\phi_{\varsigma}\int_{\tilde{u}}^{u}\int_{\tilde{u}}^{u}\mathbf{1}_{\{\tilde{r}\leq r\}}\Big(a_{x_{j}}^{ij}(x,r)-a_{y_{j}}^{ij}(y,\tilde{r})\Big)\eta_{\delta}^{\prime\prime}(r-\tilde{r})\mathrm{d}\tilde{r}\mathrm{d}r\\
 & \quad-\mathbb{E}\int_{t,x,y}\mathbf{1}_{\{u\leq\tilde{u}\}}\partial_{y_{i}}\phi_{\varsigma}\int_{u}^{\tilde{u}}\int_{u}^{\tilde{u}}\mathbf{1}_{\{r\leq\tilde{r}\}}\Big(a_{x_{j}}^{ij}(x,r)-a_{y_{j}}^{ij}(y,\tilde{r})\Big)\eta_{\delta}^{\prime\prime}(r-\tilde{r})\mathrm{d}\tilde{r}\mathrm{d}r\\
 & \quad+\mathbb{E}\int_{t,x,y}\mathbf{1}_{\{\tilde{u}\leq u\}}\partial_{y_{i}}\phi_{\varsigma}\\
 & \quad\quad\cdot\int_{\tilde{u}}^{u}\int_{\tilde{u}}^{u}\mathbf{1}_{\{\tilde{r}\leq r\}}\eta_{\delta}^{\prime\prime}(r-\tilde{r})\big(\sigma_{rx_{j}}^{jk}(x,r)\sigma_{r}^{ik}(y,\tilde{r})-\sigma_{r}^{ik}(x,r)\sigma_{ry_{j}}^{jk}(y,\tilde{r})\big)\mathrm{d}\tilde{r}\mathrm{d}r\\
 & \quad+\mathbb{E}\int_{t,x,y}\mathbf{1}_{\{u\leq\tilde{u}\}}\partial_{y_{i}}\phi_{\varsigma}\\
 & \quad\quad\cdot\int_{u}^{\tilde{u}}\int_{u}^{\tilde{u}}\mathbf{1}_{\{r\leq\tilde{r}\}}\eta_{\delta}^{\prime\prime}(r-\tilde{r})\big(\sigma_{rx_{j}}^{jk}(x,r)\sigma_{r}^{ik}(y,\tilde{r})-\sigma_{r}^{ik}(x,r)\sigma_{ry_{j}}^{jk}(y,\tilde{r})\big)\mathrm{d}\tilde{r}\mathrm{d}r.
\end{align*}
From Assumption \ref{assu:assumption for sigma}, we have
\[
|a_{x_{j}}^{ij}(x,r)-a_{x_{j}}^{ij}(y,\tilde{r})|\leq C\big(|r-\tilde{r}|+|x-y|\big),
\]
and
\[
|\sigma_{r}^{ik}(x,r)\sigma_{ry_{j}}^{jk}(y,\tilde{r})-\sigma_{r}^{ik}(y,\tilde{r})\sigma_{rx_{j}}^{jk}(x,r)|\leq C\big(|r-\tilde{r}|+|x-y|\big).
\]
Then, we have
\begin{align*}
N_{3,2}+N_{7,2}+N_{7,4} & \leq C\delta\varsigma^{-1}\mathbb{E}(\Vert u\Vert_{L_{1}(Q_{T})}+\Vert\tilde{u}\Vert_{L_{1}(Q_{T})})\\
 & \quad+C\mathbb{E}\int_{t,x,y}\varphi(t)\Big(\varsigma\sum_{i=1}^{d}|\partial_{y_{i}}\varrho_{\varsigma}(x-y)|\Big)|u-\tilde{u}|.
\end{align*}
Similarly, to estimate $N_{3,3}+N_{7,3}+N_{7,7}$, we first define
\begin{align*}
N_{3,3}=\sum_{i=1}^{2}N_{3,3,i} & :=-2\mathbb{E}\int_{t,x,y}\eta_{\delta}^{\prime}(u-\tilde{u})b^{i}(x,u)\partial_{x_{i}}\phi_{\varsigma}\\
 & \quad+2\mathbb{E}\int_{t,x,y}\eta_{\delta}^{\prime}(u-\tilde{u})b^{i}(y,\tilde{u})\partial_{y_{i}}\phi_{\varsigma}.
\end{align*}
Using
\begin{align*}
 & N_{3,3,2}+N_{7,7}\\
 & =\mathbb{E}\int_{t,x,y}\partial_{y_{i}}\phi_{\varsigma}\int_{\tilde{u}}^{u}\eta_{\delta}^{\prime\prime}(r-\tilde{u})\sigma_{y_{j}}^{jk}(y,\tilde{u})\Big(\sigma_{r}^{ik}(y,\tilde{u})-\sigma_{r}^{ik}(x,\tilde{u})\Big)\mathrm{d}r\\
 & \quad+\mathbb{E}\int_{t,x,y}\partial_{y_{i}}\phi_{\varsigma}\int_{\tilde{u}}^{u}\eta_{\delta}^{\prime\prime}(r-\tilde{u})\sigma_{y_{j}}^{jk}(y,\tilde{u})\Big(\sigma_{r}^{ik}(x,\tilde{u})-\sigma_{r}^{ik}(x,r)\Big)\mathrm{d}r\\
 & \leq\mathbb{E}\int_{t,x,y}\partial_{y_{i}}\phi_{\varsigma}\eta_{\delta}^{\prime}(u-\tilde{u})\sigma_{y_{j}}^{jk}(y,\tilde{u})(y_{l}-x_{l})\int_{0}^{1}\sigma_{rx_{l}}^{ik}(x+\theta(y-x),\tilde{u})\mathrm{d}\theta\\
 & \quad+C\delta\varsigma^{-1}\mathbb{E}(1+\Vert u\Vert_{L_{1}(Q_{T})}+\Vert\tilde{u}\Vert_{L_{1}(Q_{T})}),
\end{align*}
and 
\begin{align*}
 & N_{3,3,1}+N_{7,3}\\
 & =\mathbb{E}\int_{t,x,y}\partial_{y_{i}}\phi_{\varsigma}\int_{\tilde{u}}^{u}\eta_{\delta}^{\prime\prime}(u-\tilde{r})\sigma_{x_{j}}^{jk}(x,u)\Big(\sigma_{r}^{ik}(x,u)-\sigma_{r}^{ik}(y,u)\Big)\mathrm{d}\tilde{r}\\
 & \quad+\mathbb{E}\int_{t,x,y}\partial_{y_{i}}\phi_{\varsigma}\int_{\tilde{u}}^{u}\eta_{\delta}^{\prime\prime}(u-\tilde{r})\sigma_{x_{j}}^{jk}(x,u)\Big(\sigma_{r}^{jk}(y,u)-\sigma_{r}^{ik}(y,\tilde{r})\Big)\mathrm{d}\tilde{r}\\
 & \leq-\mathbb{E}\int_{t,x,y}\partial_{y_{i}}\phi_{\varsigma}\eta_{\delta}^{\prime}(u-\tilde{u})\sigma_{x_{j}}^{jk}(x,u)(y_{l}-x_{l})\int_{0}^{1}\sigma_{rx_{l}}^{ik}(x+\theta(y-x),u)\mathrm{d}\theta\\
 & \quad+C\delta\varsigma^{-1}\mathbb{E}(1+\Vert u\Vert_{L_{1}(Q_{T})}+\Vert\tilde{u}\Vert_{L_{1}(Q_{T})}),
\end{align*}
we have
\begin{align*}
 & N_{3,3}+N_{7,3}+N_{7,7}\\
 & \leq\mathbb{E}\int_{t,x,y}\partial_{y_{i}}\phi_{\varsigma}\eta_{\delta}^{\prime}(u-\tilde{u})\sigma_{y_{j}}^{jk}(y,\tilde{u})(y_{l}-x_{l})\int_{0}^{1}\sigma_{rx_{l}}^{ik}(x+\theta(y-x),\tilde{u})\mathrm{d}\theta\\
 & \quad-\mathbb{E}\int_{t,x,y}\partial_{y_{i}}\phi_{\varsigma}\eta_{\delta}^{\prime}(u-\tilde{u})\sigma_{x_{j}}^{jk}(x,u)(y_{l}-x_{l})\int_{0}^{1}\sigma_{rx_{l}}^{ik}(x+\theta(y-x),u)\mathrm{d}\theta\\
 & \quad+C\delta\varsigma^{-1}\mathbb{E}(1+\Vert u\Vert_{L_{1}(Q_{T})}+\Vert\tilde{u}\Vert_{L_{1}(Q_{T})}).
\end{align*}
Since
\begin{align*}
 & \Big|\sigma_{y_{j}}^{jk}(y,\tilde{u})\int_{0}^{1}\sigma_{rx_{l}}^{ik}(x+\theta(y-x),\tilde{u})\mathrm{d}\theta-\sigma_{x_{j}}^{jk}(x,u)\int_{0}^{1}\sigma_{rx_{l}}^{ik}(x+\theta(y-x),u)\mathrm{d}\theta\Big|\\
 & \leq C\varsigma(1+|u|+|\tilde{u}|)+\Big|\sigma_{y_{j}}^{jk}(y,\tilde{u})\sigma_{rx_{l}}^{ik}(x,\tilde{u})-\sigma_{x_{j}}^{jk}(x,u)\sigma_{rx_{l}}^{ik}(x,u)\Big|\\
 & \leq C\varsigma(1+|u|+|\tilde{u}|)+\Big|\sigma_{y_{j}}^{jk}(y,\tilde{u})\sigma_{rx_{l}}^{ik}(x,\tilde{u})-\sigma_{x_{j}}^{jk}(x,\tilde{u})\sigma_{rx_{l}}^{ik}(x,\tilde{u})\Big|+C|u-\tilde{u}|\\
 & \leq C\varsigma(1+|u|+|\tilde{u}|)+C|u-\tilde{u}|,
\end{align*}
we have
\begin{align*}
N_{3,3}+N_{7,3}+N_{7,7} & \leq C(\delta\varsigma^{-1}+\varsigma)\mathbb{E}(1+\Vert u\Vert_{L_{1}(Q_{T})}+\Vert\tilde{u}\Vert_{L_{1}(Q_{T})})\\
 & \quad+C\mathbb{E}\int_{t,x,y}\varphi(t)\Big(\varsigma\sum_{i=1}^{d}|\partial_{y_{i}}\varrho_{\varsigma}(x-y)|\Big)\cdot|u-\tilde{u}|.
\end{align*}
For $N_{7,5}$, using Assumption \ref{assu:assumption for sigma},
we have
\begin{align*}
N_{7,5} & \leq C\mathbb{E}\int_{t,x,y}\varphi(t)\varrho_{\varsigma}(x-y)\cdot|u-\tilde{u}|.
\end{align*}
Similarly, we have
\begin{align*}
N_{7,6}+N_{7,8} & \leq\mathbb{E}\int_{t,x,y}\phi_{\varsigma}\int_{u}^{\tilde{u}}\eta_{\delta}^{\prime\prime}(u-\tilde{r})\sigma_{x_{j}}^{jk}(x,u)\sigma_{ry_{i}}^{ik}(y,u)\mathrm{d}\tilde{r}\\
 & \quad+\mathbb{E}\int_{t,x,y}\phi_{\varsigma}\int_{\tilde{u}}^{u}\eta_{\delta}^{\prime\prime}(r-\tilde{u})\sigma_{rx_{j}}^{jk}(x,\tilde{u})\sigma_{y_{i}}^{ik}(y,\tilde{u})\mathrm{d}r\\
 & \quad+C\delta\mathbb{E}(1+\Vert u\Vert_{L_{1}(Q_{T})}+\Vert\tilde{u}\Vert_{L_{1}(Q_{T})})\\
 & \leq\mathbb{E}\int_{t,x,y}\phi_{\varsigma}\eta_{\delta}^{\prime}(u-\tilde{u})|\sigma_{rx_{j}}^{jk}(x,\tilde{u})\sigma_{y_{i}}^{ik}(y,\tilde{u})-\sigma_{x_{j}}^{jk}(x,u)\sigma_{ry_{i}}^{ik}(y,u)|\\
 & \quad+C\delta\mathbb{E}(1+\Vert u\Vert_{L_{1}(Q_{T})}+\Vert\tilde{u}\Vert_{L_{1}(Q_{T})})\\
 & \leq C(\delta+\varsigma)\mathbb{E}(1+\Vert u\Vert_{L_{1}(Q_{T})}+\Vert\tilde{u}\Vert_{L_{1}(Q_{T})})+\mathbb{E}\int_{t,x,y}\varphi(t)\varrho_{\varsigma}(x-y)|u-\tilde{u}|.
\end{align*}
Using Assumption \ref{assu:assumption for sigma}, we also have
\begin{align*}
N_{4} & =-\mathbb{E}\int_{t,x,y}\mathbf{1}_{\{\tilde{u}\leq u\}}\partial_{x_{i}}\phi_{\varsigma}\int_{\tilde{u}}^{u}\int_{\tilde{u}}^{u}\mathbf{1}_{\{\tilde{r}\leq r\}}\eta_{\delta}^{\prime\prime}(r-\tilde{r})\Big(b_{r}^{i}(y,\tilde{r})-b_{r}^{i}(x,r)\Big)\mathrm{d}r\mathrm{d}\tilde{r}\\
 & \quad-\mathbb{E}\int_{t,x,y}\mathbf{1}_{\{u\leq\tilde{u}\}}\partial_{x_{i}}\phi_{\varsigma}\int_{u}^{\tilde{u}}\int_{u}^{\tilde{u}}\mathbf{1}_{\{r\leq\tilde{r}\}}\eta_{\delta}^{\prime\prime}(r-\tilde{r})\Big(b_{r}^{i}(y,\tilde{r})-b_{r}^{i}(x,r)\Big)\mathrm{d}r\mathrm{d}\tilde{r}\\
 & \quad-\mathbb{E}\int_{t,x,y}\phi_{\varsigma}\eta_{\delta}^{\prime}(u-\tilde{u})\big(b_{x_{i}}^{i}(x,u)-b_{y_{i}}^{i}(y,\tilde{u})\big)\phi_{\varsigma}\\
 & \quad+\mathbb{E}\int_{t,x,y}\phi_{\varsigma}\int_{\tilde{u}}^{u}b_{rx_{i}}^{i}(x,r)\eta_{\delta}^{\prime}(r-\tilde{u})\mathrm{d}r-\mathbb{E}\int_{t,x,y}\phi_{\varsigma}\int_{u}^{\tilde{u}}b_{ry_{i}}^{i}(y,\tilde{r})\eta_{\delta}^{\prime}(u-\tilde{r})\mathrm{d}\tilde{r}\\
 & \leq C(\varsigma+\delta\varsigma^{-1})\mathbb{E}(1+\Vert u\Vert_{L_{1}(Q_{T})}+\Vert\tilde{u}\Vert_{L_{1}(Q_{T})})\\
 & \quad+C\mathbb{E}\int_{t,x,y}\varphi(t)\Big(\varsigma\sum_{i=1}^{d}|\partial_{x_{i}}\varrho_{\varsigma}(x-y)|+\varrho_{\varsigma}(x-y)\Big)|u-\tilde{u}|.
\end{align*}
Combining these estimates of $N_i$, $i=1,\ldots,7$ and
\[
\Big|\mathbb{E}\int_{t,x,y}\eta_{\delta}(u-\tilde{u})\partial_{t}\phi_{\varsigma}-\mathbb{E}\int_{t,x,y}|u-\tilde{u}|\partial_{t}\phi_{\varsigma}\Big|\leq C\delta,
\]
we obtain 
\begin{equation}
\begin{aligned} & -\mathbb{E}\int_{t,x,y}|u-\tilde{u}|\partial_{t}\phi_{\varsigma}\\
 & \leq C\mathbb{E}\int_{t,x,y}|u-\tilde{u}|\varphi(t)\Big(\varsigma^{2}\sum_{i,j=1}^{d}|\partial_{y_{i}x_{j}}\varrho_{\varsigma}(x-y)|+\varsigma\sum_{i=1}^{d}|\partial_{x_{i}}\varrho_{\varsigma}(x-y)|+\varrho_{\varsigma}(x-y)\Big)\\
 & \quad+C(\delta^{-1}\varsigma^{2}+\varsigma+\varsigma^{-2}\delta^{2\alpha}+\delta\varsigma^{-1}+\delta^{2}\varsigma^{-2})\mathbb{E}\Big[1+\Vert u\Vert_{L_{m+1}(Q_{T})}^{m+1}+\Vert\tilde{u}\Vert_{L_{m+1}(Q_{T})}^{m+1}\Big].
\end{aligned}
\label{eq:diffu tildeu-1-1}
\end{equation}
Let $\vartheta\in(m^{-1}\lor2^{-1},1)$, $\alpha\in(1/(2\vartheta),1\land(m/2))$, 
and $\delta=\varsigma^{2\vartheta}$. Note that $\varsigma^{2}\sum_{i,j}|\partial_{y_{i}x_{j}}\varrho_{\varsigma}|$
and $\varsigma|\partial_{x_{i}}\varrho_{\varsigma}|$ are approximations
of the identity up to a constant. Taking $\varsigma\downarrow0$,
with the continuity of translations in $L_{1}$, we have
\begin{equation}
-\mathbb{E}\int_{t,x}|u(t,x)-\tilde{u}(t,x)|\partial_{t}\varphi(t)\leq C\mathbb{E}\int_{t,x}|u(t,x)-\tilde{u}(t,x)|\varphi(t),\label{eq:L1_with phi}
\end{equation}
where the constant $C$ only depends on $\Vert\varphi\Vert_{L_{\infty}(0,T)}$,
$\Vert\partial_{t}\varphi\Vert_{L_{1}(0,T)}$, $K$, $d$, and $T$,
and increases with their growth.

Let $0<s<\tau<T$ be Lebesgue points of the function 
\[
t\rightarrow\mathbb{E}\int_{x}|u(t,x)-\tilde{u}(t,x)|.
\]
Fix a constant $\gamma\in(0,(\tau-s)\lor(T-\tau))$. We take a sequence
of functions $\{\varphi_{n}\}_{n\in\mathbb{N}}$ satisfying $\varphi_{n}\in C_{c}^{\infty}((0,T))$ and
\[
\max\{\Vert\varphi\Vert_{L_{\infty}(0,T)},\Vert\partial_{t}\varphi\Vert_{L_{1}(0,T)}\}\leq1,
\]
 such that
\[
\lim_{n\rightarrow\infty}\Vert\varphi_{n}-V_{(\gamma)}\Vert_{H_{0}^{1}(0,T)}=0,
\]
where $V_{(\gamma)}:[0,T]\rightarrow\mathbb{R}$ satisfies $V_{(\gamma)}(0)=0$
and $V_{(\gamma)}^{\prime}=\gamma^{-1}\mathbf{1}_{[s,s+\gamma]}-\gamma^{-1}\mathbf{1}_{[\tau,\tau+\gamma]}$.
Taking $\varphi=\varphi_{n}$ in (\ref{eq:L1_with phi}) and passing
to the limit $n\rightarrow\infty$, we have 
\begin{align*}
 & \frac{1}{\gamma}\mathbb{E}\int_{\tau}^{\tau+\gamma}\int_{x}|u(t,x)-\tilde{u}(t,x)|\mathrm{d}t\\
 & \qquad\leq\frac{1}{\gamma}\mathbb{E}\int_{s}^{s+\gamma}\int_{x}|u(t,x)-\tilde{u}(t,x)|\mathrm{d}t+C\mathbb{E}\int_{0}^{\tau+\gamma}\int_{x}|u(t,x)-\tilde{u}(t,x)|\mathrm{d}t.
\end{align*}
Let $\gamma\downarrow0$ and we have
\[
\mathbb{E}\int_{x}|u(\tau,x)-\tilde{u}(\tau,x)|\leq\mathbb{E}\int_{x}|u(s,x)-\tilde{u}(s,x)|+C\mathbb{E}\int_{0}^{\tau}\int_{x}|u(t,x)-\tilde{u}(t,x)|\mathrm{d}t
\]
holds for almost all $s\in(0,\tau)$. Then, for each $\tilde{\gamma}\in(0,\tau)$,
by averaging over $s\in(0,\tilde{\gamma})$, we have
\begin{align*}
 & \mathbb{E}\int_{x}|u(\tau,x)-\tilde{u}(\tau,x)|\\
 & \qquad\leq\frac{1}{\tilde{\gamma}}\mathbb{E}\int_{0}^{\tilde{\gamma}}\int_{x}|u(t,x)-\tilde{u}(t,x)|\mathrm{d}t+C\mathbb{E}\int_{0}^{\tau}\int_{x}|u(t,x)-\tilde{u}(t,x)|\mathrm{d}t.
\end{align*}
Taking the limit $\tilde{\gamma}\downarrow0$ and using Lemma
\ref{lem:appro initial} and Gr\"onwall's inequality, we obtain (\ref{eq:L1_for diff solution u}).
\end{proof}

Now, we prove
the uniqueness of the compensation measure $\nu$.
\begin{prop}
\label{prop:Uniquenessfornu}Let Assumptions \ref{assu:assumption for phi}
and \ref{assu:assumption for sigma} be satisfied, and $\psi\in C^{2}_x(\bar{Q}_{T})$.
Given $\psi(0,\cdot)\le\xi,\tilde{\xi}\in L_{m+1}(\Omega,\mathcal{F}_{0};L_{m+1}(\mathbb{T}^{d}))$,
we suppose that $(u,\nu)$ and $(\tilde{u},\tilde{\nu})$ are entropy
solutions to $\Pi_{\psi}(\Phi,f,\xi)$ and $\Pi_{\psi}(\Phi,f,\tilde{\xi})$,
respectively. If $u=\tilde{u}$ for almost all $(\omega,t,x)\in\Omega\times Q_{T}$,
we have $\nu=\tilde{\nu}$ as an Radon measure on $Q_{T}$, almost
surely.
\end{prop}

\begin{proof}
We apply the entropy inequality (\ref{eq:entropy formula-1}) by taking $\eta(r):=r$ and $\eta(r):=-r$, respectively. Combining
these two inequalities, we have
\[
\begin{aligned} & -\int_{Q_{T}}\phi\nu(\mathrm{d}x\mathrm{d}t)\\
 & =\int_{Q_{T}}u\partial_{t}\phi\mathrm{d}x\mathrm{d}t+\int_{\mathbb{T}^{d}}\xi\phi(0)\mathrm{d}x+\int_{Q_{T}}\Phi(u)\Delta\phi\mathrm{d}x\mathrm{d}t\\
 & \quad+\int_{Q_{T}}\llbracket a^{ij}\rrbracket(u)\partial_{x_{i}x_{j}}\phi\mathrm{d}x\mathrm{d}t\\
 & \quad+\int_{Q_{T}}\Big(\llbracket a_{x_{j}}^{ij}+b_{r}^{i}\rrbracket(x,u)-2b^{i}(x,u)\Big)\partial_{x_{i}}\phi\mathrm{d}x\mathrm{d}t\\
 & \quad+\int_{Q_{T}}\Big(-b_{x_{i}}^{i}(x,u)+\llbracket b_{rx_{i}}^{i}\rrbracket(x,u)+f(t,x,u)\Big)\phi\mathrm{d}x\mathrm{d}t\\
 & \quad+\int_{Q_{T}}\Big(\phi\sigma_{x_{i}}^{ik}(x,u)-\llbracket\sigma_{rx_{i}}^{ik}\rrbracket(x,u)\phi-\llbracket\sigma_{r}^{ik}\rrbracket(x,u)\partial_{x_{i}}\phi\Big)\mathrm{d}x\mathrm{d}W_{t}^{k}.
\end{aligned}
\]
With the uniqueness of $u$, if there exists another entropy solution
$(u,\tilde{\nu})$ to the obstacle problem $\Pi_{\psi}(\Phi,f,\xi)$,
we have
\begin{equation}
\int_{Q_{T}}\phi\nu(\mathrm{d}x\mathrm{d}t)=\int_{Q_{T}}\phi\tilde{\nu}(\mathrm{d}x\mathrm{d}t)\quad\text{a.s.}\ \omega\in\Omega\label{eq:equal for =00005Cnu},
\end{equation}
for all test function $\phi:=\varphi\varrho\geq0$, where $(\varphi,\varrho)\in C_{c}^{\infty}([0,T))\times C^{\infty}(\mathbb{T}^{d})$. 

Now, we weaken the requirement on $\phi$ to $C(\bar{Q}_T)$. Fix a function
$\phi\in C(\bar{Q}_T)$ satisfying $\phi(T,\cdot)\equiv0$.
For any $\bar{\varepsilon}>0$, there exist $L\in\mathbb{N}$, series
$\{z_{l}\}_{l=1}^{L}$ and smooth functions $\{\phi_{l}\}_{l=1}^{L}$
satisfying $\phi_{l}=\varphi_{l}\rho_{l}\geq0$ where $\varphi_{l}\in C_{c}^{\infty}([0,T))$
and $\rho_{l}\in C^{\infty}(\mathbb{T}^{d})$, such that
\[
\Big\Vert\phi-\sum_{l=1}^{L}z_{l}\phi_{l}\Big\Vert_{C(\bar{Q}_T)}\leq\bar{\varepsilon}.
\]
Then, we have 
\begin{equation}
\begin{aligned} & \mathbb{E}\bigg|\int_{Q_{T}}\phi\tilde{\nu}(\mathrm{d}x\mathrm{d}t)-\int_{Q_{T}}\phi\nu(\mathrm{d}x\mathrm{d}t)\bigg|\\
 & \leq\mathbb{E}\bigg|\int_{Q_{T}}\phi\tilde{\nu}(\mathrm{d}x\mathrm{d}t)-\int_{Q_{T}}\sum_{l=1}^{L}z_{l}\phi_{l}\tilde{\nu}(\mathrm{d}x\mathrm{d}t)\bigg|\\
 & \quad+\sum_{l=1}^{L}z_{l}\mathbb{E}\bigg|\int_{Q_{T}}\phi_{l}\tilde{\nu}(\mathrm{d}x\mathrm{d}t)-\int_{Q_{T}}\phi_{l}\nu(\mathrm{d}x\mathrm{d}t)\bigg|\\
 & \quad+\mathbb{E}\bigg|\int_{Q_{T}}\phi\nu(\mathrm{d}x\mathrm{d}t)-\int_{Q_{T}}\sum_{l=1}^{L}z_{l}\phi_{l}\nu(\mathrm{d}x\mathrm{d}t)\bigg|\\
 & \leq\Big(\mathbb{E}\tilde{\nu}(Q_T)+\mathbb{E}\nu(Q_{T})\Big)\cdot\Big\Vert\phi-\sum_{l=1}^{L}z_{l}\phi_{l}\Big\Vert_{C(\bar{Q}_T)}\leq C\bar{\varepsilon}.
\end{aligned}
\label{eq:unique for nu}
\end{equation}
Therefore, we have (\ref{eq:equal for =00005Cnu}) holds for all $\phi\in C(\bar{Q}_T)$
satisfying $\phi(T,\cdot)\equiv0$, which
means $\nu=\tilde{\nu}$ as an Radon measure on $Q_{T}$, almost surely.
\end{proof}

\section{Approximation and the penalization scheme\label{sec:Approximation}}

Sections \ref{sec:Approximation}--\ref{sec:Existence} are devoted to the proof of
the existence of the entropy solution. 
The basic idea
is approximating $\Phi$ with the smooth functions $\Phi_{n}$ as
in \cite{dareiotis2019entropy} and using the penalization method. 
\begin{lem}[Proposition 5.1 in \cite{dareiotis2019entropy}]
\label{lem:Assumptio for coef}
Let $\Phi$ satisfy Assumption~\ref{assu:assumption for phi} with
a constant $K>1$. Then, for all $n\in\mathbb{N}$, there exists an
increasing function $\Phi_{n}\in C^{\infty}(\mathbb{R})$ with bounded
derivatives, satisfying Assumption~\ref{assu:assumption for phi}
with constant $3K$, such that $\sqrt{\Phi_{n}^{\prime}}(r)\geq2/n$,
and
\[
\sup_{|r|\leq n}\,|\sqrt{\Phi^{\prime}}(r)-\sqrt{\Phi_{n}^{\prime}}(r)|\leq\frac{4}{n}.
\]
\end{lem}

For $\epsilon>0$, define $P_{\epsilon}(r,b):=(r-b)^{-}/\epsilon$ and $\xi_{n}:=(-n)\lor(\xi\land n)$. 
We introduce
the approximating penalized equations
$\Pi(\Phi_{n},f+P_{\epsilon}(\cdot,\psi),\xi_{n})$:
 $u_{n,\epsilon}(0,x)=\xi_n(x)$ and
\[
\mathrm{d}u_{n,\epsilon}=[\Delta\Phi_n(u_{n,\epsilon})+f(t,x,u_{n,\epsilon})+\epsilon^{-1}(u_{n,\epsilon}-\psi)^{-}]\mathrm{d}t+\nabla\cdot\sigma^{k}(x,u_{n,\epsilon})\circ\mathrm{d}{W}_{t}^{k},
\]
Let $u_{n,\epsilon}$ be the $L_2$-solution of
of $\Pi(\Phi_{n},f+P_{\epsilon}(\cdot,\psi),\xi_{n})$, which means it is a weak solution such that $u_{n,\epsilon},\Phi_{n}(u_{n,\epsilon})\in L_{2}(\Omega_{T},H^{1}(\mathbb{T}^{d}))$. 
The well-posedness of these equations can be proved using a priori estimates (Proposition \ref{prop:estimate for u_n,epsilon}) and Galerkin method as in \cite{dareiotis2020nonlinear,du2024entropy}. 

\begin{prop}
\label{prop:estimate for u_n,epsilon}Let Assumptions \ref{assu:assumption for phi}
and \ref{assu:assumption for sigma} hold. Suppose $\xi\in L_{m+1}(\Omega,\mathcal{F}_{0};L_{m+1}(\mathbb{T}^{d}))$
and  $\psi\in C(\bar{Q}_{T})$. For all $n\in\mathbb{N}$,
$\epsilon>0$,  and $p\in[2,\infty)$, an $L_{2}$-solution $u_{n,\epsilon}$
satisfies estimates
\begin{equation}
\begin{gathered}\mathbb{E}\sup_{t\leq T}\,\Vert u_{n,\epsilon}(t)\Vert_{L_{2}(\mathbb{T}^{d})}^{p}+\mathbb{E}\big\Vert\nabla\llbracket\sqrt{\Phi_{n}^{\prime}}\rrbracket(u_{n,\epsilon})\big\Vert_{L_{2}(Q_{T})}^{p}+\epsilon^{-\frac{p}{2}}\mathbb{E}\Vert(u_{n,\epsilon}-\psi)^{-}\Vert_{L_{2}(Q_{T})}^{p}\\
+\mathbb{E}\Vert\epsilon^{-1}(u_{n,\epsilon}-\psi)^{-}\Vert_{L_{1}(Q_{T})}^{\frac{p}{2}}\leq C\Big(1+\mathbb{E}\Vert\xi_{n}\Vert_{L_{2}(\mathbb{T}^{d})}^{p}\Big),
\end{gathered}
\label{eq:priori estimate1}
\end{equation}
and
\begin{equation}
\begin{gathered}\mathbb{E}\sup_{t\leq T}\,\Vert u_{n,\epsilon}(t)\Vert_{L_{m+1}(\mathbb{T}^{d})}^{m+1}+\frac{1}{\epsilon}\mathbb{E}\int_{Q_{T}}|(u_{n,\epsilon}-\psi)^{-}|^{2}|u_{n,\epsilon}|^{m-1}\mathrm{d}x\mathrm{d}s\\
+\Vert\nabla\Phi_{n}(u_{n,\epsilon})\Vert_{L_{2}(\Omega\times Q_{T})}^{2}\leq C\Big(1+\Vert\xi_{n}\Vert_{L_{m+1}(\Omega\times\mathbb{T}^{d})}^{m+1}\Big),
\end{gathered}
\label{eq:priori estimate2}
\end{equation}
where the constant C is independent of $n$ and $\epsilon$.
\end{prop}

\begin{rem}
Proposition \ref{prop:estimate for u_n,epsilon} provides a uniform estimate for the penalty term $\epsilon^{-1}(u_{n,\epsilon}-\psi)^{-}$,
which is essential in establishing the existence of $\nu$.
\end{rem}

\begin{proof}[Proof of Proposition \ref{prop:estimate for u_n,epsilon}]
Since $\psi\in C(\bar{Q}_{T})$, we can assume that $\psi\leq M$
for a constant $M$ big enough. Applying It\^o's formula (cf. \cite[Lemma 2]{dareiotis2015boundedness}),
we have
\begin{align*}
&\Vert u_{n,\epsilon}(t)-M-1\Vert_{L_{2}(\mathbb{T}^{d})}^{2}\\
 & =\Vert\xi_{n}-M-1\Vert_{L_{2}(\mathbb{T}^{d})}^{2}-2\int_{0}^{t}\langle\partial_{x_{i}}\Phi_{n}(u_{n,\epsilon}),\partial_{x_{i}}u_{n,\epsilon}\rangle_{L_{2}(\mathbb{T}^{d})}\mathrm{d}s\\
 & \quad-2\int_{0}^{t}\langle a^{ij}(x,u_{n,\epsilon})\partial_{x_{j}}u_{n,\epsilon}+b^{i}(x,u_{n,\epsilon}),\partial_{x_{i}}u_{n,\epsilon}\rangle_{L_{2}(\mathbb{T}^{d})}\mathrm{d}s\\
 & \quad+2\int_{0}^{t}\langle f(s,\cdot,u_{n,\epsilon})+\frac{1}{\epsilon}(u_{n,\epsilon}-\psi)^{-},u_{n,\epsilon}-M-1\rangle_{L_{2}(\mathbb{T}^{d})}\mathrm{d}s\\
 & \quad-2\int_{0}^{t}\langle\sigma^{ik}(x,u_{n,\epsilon}),\partial_{x_{i}}u_{n,\epsilon}\rangle_{L_{2}(\mathbb{T}^{d})}\mathrm{d}W_{s}^{k}\\
 & \quad+\int_{0}^{t}\sum_{k=1}^{\infty}\Vert\sigma_{r}^{ik}(x,u_{n,\epsilon})\partial_{x_{i}}u_{n,\epsilon}+\sigma_{x_{i}}^{ik}(x,u_{n,\epsilon})\Vert_{L_{2}(\mathbb{T}^{d})}^{2}\mathrm{d}s,\quad\textrm{a.e.}\ t\in[0,T].
\end{align*}
Since
\begin{equation}
\begin{aligned} & \sum_{k=1}^{\infty}|\sigma_{r}^{ik}(u_{n,\epsilon})\partial_{x_{i}}u_{n,\epsilon}+\sigma_{x_{i}}^{ik}(u_{n,\epsilon})|^{2}\\
 & =2a^{ij}(x,u_{n,\epsilon})\partial_{x_{j}}u_{n,\epsilon}\partial_{x_{i}}u_{n,\epsilon}+4b^{i}(x,u_{n,\epsilon})\partial_{x_{i}}u_{n,\epsilon}+\sum_{k=1}^{\infty}\sigma_{x_{i}}^{ik}(x,u_{n,\epsilon})\sigma_{x_{j}}^{jk}(x,u_{n,\epsilon}),
\end{aligned}
\label{eq:couple term in L2}
\end{equation}
and
\begin{equation}
\begin{aligned}\langle b^{i}(x,u_{n,\epsilon}),\partial_{x_{i}}u_{n,\epsilon}\rangle_{L_{2}(\mathbb{T}^{d})} & =-\int_{x}\llbracket b_{x_{i}}^{i}\rrbracket(x,u_{n,\epsilon}),\\
\langle\sigma^{ik}(x,u_{n,\epsilon}),\partial_{x_{i}}u_{n,\epsilon}\rangle_{L_{2}(\mathbb{T}^{d})} & =-\int_{x}\llbracket\sigma_{x_{i}}^{ik}\rrbracket(x,u_{n,\epsilon}),
\end{aligned}
\label{eq:divergence integral}
\end{equation}
in view of the definition of $\sqrt{\Phi_{n}^{\prime}}$ and Assumption
\ref{assu:assumption for sigma}, we have
\begin{align*}
\Vert u_{n,\epsilon}(t)\Vert_{L_{2}(\mathbb{T}^{d})}^{2} & \leq C+\Vert\xi_{n}\Vert_{L_{2}(\mathbb{T}^{d})}^{2}+C\int_{0}^{t}\Vert u_{n,\epsilon}\Vert_{L_{2}(\mathbb{T}^{d})}^{2}\mathrm{d}s\\
 & \quad-2\int_{0}^{t}\Vert\nabla\llbracket\sqrt{\Phi_{n}^{\prime}}\rrbracket(u_{n,\epsilon})\Vert_{L_{2}(\mathbb{T}^{d})}^{2}\mathrm{d}s\\
 &\quad+\frac{2}{\epsilon}\int_{0}^{t}\langle(u_{n,\epsilon}-\psi)^{-},u_{n,\epsilon}-M-1\rangle_{L_{2}(\mathbb{T}^{d})}\mathrm{d}s\\
 & \quad+2\int_{0}^{t}\int_{x}\llbracket\sigma_{x_{i}}^{ik}\rrbracket(x,u_{n,\epsilon})\mathrm{d}W_{s}^{k}.
\end{align*}
Since $\psi\leq M$, we have
\begin{align*}
\frac{2}{\epsilon}\int_{0}^{t}\langle(u_{n,\epsilon}-\psi)^{-},u_{n,\epsilon}-M-1\rangle_{L_{2}(\mathbb{T}^{d})}\mathrm{d}s & \leq-\frac{2}{\epsilon}\int_{0}^{t}\Vert(u_{n,\epsilon}-\psi)^{-}\Vert_{L_{2}(\mathbb{T}^{d})}^{2}\mathrm{d}s\\
 & \quad-\frac{2}{\epsilon}\int_{0}^{t}\int_{x}(u_{n,\epsilon}-\psi)^{-}\mathrm{d}s.
\end{align*}
Raising to the power $p/2$, taking supreme up to time $\tau$ and
expectations, since
\begin{align*}
 & \mathbb{E}\sup_{t\leq\tau}\,\Big|\sum_{k=1}^{\infty}\int_{0}^{t}\int_{x}\llbracket\sigma_{x_{i}}^{ik}\rrbracket(x,u_{n,\epsilon})\mathrm{d}W_{s}^{k}\Big|^{\frac{p}{2}}\\
 & \leq\mathbb{E}\Big|\int_{0}^{\tau}\sum_{k=1}^{\infty}\Big(\int_{x}\llbracket\sigma_{x_{i}}^{ik}\rrbracket(x,u_{n,\epsilon})\Big)^{2}\mathrm{d}s\Big|^{\frac{p}{4}}\\
 & \leq C\mathbb{E}\Big|\int_{0}^{\tau}\Big(\int_{x}\int_{0}^{u_{n,\epsilon}}\big(\sum_{k=1}^{\infty}|\sigma_{x_{i}}^{ik}(x,r)|^{2}\big)^{\frac{1}{2}}\mathrm{d}r\Big)^{2}\mathrm{d}s\Big|^{\frac{p}{4}}\\
 & \leq C\mathbb{E}\Big|\int_{0}^{\tau}1+\Vert u_{n,\epsilon}\Vert_{L_{2}(\mathbb{T}^{d})}^{4}\mathrm{d}s\Big|^{\frac{p}{4}}\\
 & \leq C+\bar{\varepsilon}\mathbb{E}\sup_{t\leq\tau}\,\Vert u_{n,\epsilon}(t)\Vert_{L_{2}(\mathbb{T}^{d})}^{p}+\bar{\varepsilon}^{-1}C\int_{0}^{\tau}\mathbb{E}\sup_{t\leq s}\,\Vert u_{n,\epsilon}(t)\Vert_{L_{2}(\mathbb{T}^{d})}^{p}\mathrm{d}s,
\end{align*}
applying Gr\"onwall's inequality, we have inequality (\ref{eq:priori estimate1}).

To prove inequality (\ref{eq:priori estimate2}) without the term
$\mathbb{E}\Vert\nabla\Phi_{n}(u_{n,\epsilon})\Vert_{L_{2}(Q_{T})}^{2}$,
applying It\^o's formula (cf. \cite[Lemma 2]{dareiotis2015boundedness}),
we have 
\begin{align*}
 & \Vert u_{n,\epsilon}(t)-M\Vert_{L_{m+1}(\mathbb{T}^{d})}^{m+1}\\
 & =\Vert\xi_{n}-M\Vert_{L_{m+1}(\mathbb{T}^{d})}^{m+1}-(m^{2}+m)\int_{0}^{t}\int_{x}\partial_{x_{i}}\Phi_{n}(u_{n,\epsilon})|u_{n,\epsilon}-M|^{m-1}\partial_{x_{i}}u_{n,\epsilon}\mathrm{d}s\\
 & \quad-(m^{2}+m)\int_{0}^{t}\int_{x}\Big(a^{ij}(x,u_{n,\epsilon})\partial_{x_{j}}u_{n,\epsilon}+b^{i}(x,u_{n,\epsilon})\Big)|u_{n,\epsilon}-M|^{m-1}\partial_{x_{i}}u_{n,\epsilon}\mathrm{d}s\\
 & \quad+(m+1)\int_{0}^{t}\int_{x}\big[f(s,x,u_{n,\epsilon})+\frac{1}{\epsilon}(u_{n,\epsilon}-\psi)^{-}\big]|u_{n,\epsilon}-M|^{m-1}(u_{n,\epsilon}-M)\mathrm{d}s\\
 & \quad-(m^{2}+m)\int_{0}^{t}\int_{x}\sigma^{ik}(x,u_{n,\epsilon})|u_{n,\epsilon}-M|^{m-1}\partial_{x_{i}}u_{n,\epsilon}\mathrm{d}W_{s}^{k}\\
 & \quad+\frac{(m^{2}+m)}{2}\int_{0}^{t}\int_{x}\sum_{k=1}^{\infty}\big(\sigma_{r}^{ik}(x,u_{n,\epsilon})\partial_{x_{i}}u_{n,\epsilon}+\sigma_{x_{i}}^{ik}(x,u_{n,\epsilon})\big)^{2}|u_{n,\epsilon}-M|^{m-1}\mathrm{d}s.
\end{align*}
With (\ref{eq:couple term in L2}) and 
\begin{align*}
\int_{x}b^{i}(x,u_{n,\epsilon})|u_{n,\epsilon}-M|^{m-1}\partial_{x_{i}}u_{n,\epsilon} & =-\int_{x}\llbracket b_{x_{i}}^{i}|\cdot-M|^{m-1}\rrbracket(x,u_{n,\epsilon}),\\
\int_{x}\sigma^{ik}(x,u_{n,\epsilon})|u_{n,\epsilon}-M|^{m-1}\partial_{x_{i}}u_{n,\epsilon} & =-\int_{x}\llbracket\sigma_{x_{i}}^{ik}|\cdot-M|^{m-1}\rrbracket(x,u_{n,\epsilon}),
\end{align*}
in view of the monotonicity of $\Phi_{n}$, Assumption \ref{assu:assumption for sigma}, 
and $\psi\leq M$, we have
\begin{equation}
\begin{aligned} & \Vert u_{n,\epsilon}(t)\Vert_{L_{m+1}(\mathbb{T}^{d})}^{m+1}+\frac{m+1}{\epsilon}\int_{0}^{t}\int_{x}|(u_{n,\epsilon}-\psi)^{-}|^{2}|u_{n,\epsilon}|^{m-1}\mathrm{d}s\\
 & \leq C+C\Vert\xi_{n}\Vert_{L_{m+1}(\mathbb{T}^{d})}^{m+1}+C\int_{0}^{t}\Vert u_{n,\epsilon}\Vert_{L_{m+1}(\mathbb{T}^{d})}^{m+1}\mathrm{d}s\\
 & \quad+(m^{2}+m)\int_{0}^{t}\int_{x}\llbracket\sigma_{x_{i}}^{ik}|\cdot-M|^{m-1}\rrbracket(x,u_{n,\epsilon})\mathrm{d}W_{s}^{k}.
\end{aligned}
\label{eq:second estimate ito formula}
\end{equation}
Taking supreme up to time $\tau$ and taking expectations, since
\begin{align*}
 & \mathbb{E}\sup_{t\leq\tau}\,\Big|\sum_{k=1}^{\infty}\int_{0}^{t}\int_{x}\llbracket|\cdot-M|^{m-1}\sigma_{x_{i}}^{ik}\rrbracket(x,u_{n,\epsilon})\mathrm{d}W_{s}^{k}\Big|\\
 & \leq C\mathbb{E}\Big|\int_{0}^{\tau}\sum_{k=1}^{\infty}\Big(\int_{x}\llbracket|\cdot-M|^{m-1}\sigma_{x_{i}}^{ik}\rrbracket(x,u_{n,\epsilon})\Big)^{2}\mathrm{d}s\Big|^{\frac{1}{2}}\\
 & \leq C\mathbb{E}\Big|\int_{0}^{\tau}\Big(\int_{x}\int_{0}^{u_{n,\epsilon}}\big|\sum_{k=1}^{\infty}|\sigma_{x_{i}}^{ik}(x,r)|^{2}\big|^{\frac{1}{2}}|r-M|^{m-1}\mathrm{d}r\Big)^{2}\mathrm{d}s\Big|^{\frac{1}{2}}\\
 & \leq C\mathbb{E}\Big|\int_{0}^{\tau}1+\Vert u_{n,\epsilon}\Vert_{L_{m+1}(\mathbb{T}^{d})}^{2(m+1)}\mathrm{d}s\Big|^{\frac{1}{2}}\\
 & \leq C+\bar{\varepsilon}\mathbb{E}\sup_{t\leq\tau}\,\Vert u_{n,\epsilon}(t)\Vert_{L_{m+1}(\mathbb{T}^{d})}^{m+1}+\bar{\varepsilon}^{-1}C\int_{0}^{\tau}\mathbb{E}\sup_{t\leq s}\,\Vert u_{n,\epsilon}(t)\Vert_{L_{m+1}(\mathbb{T}^{d})}^{m+1}\mathrm{d}s,
\end{align*}
applying Gr\"onwall's inequality, we obtain the desired inequality.

To estimate the term $\mathbb{E}\Vert\nabla\Phi_{n}(u_{n,\epsilon})\Vert_{L_{2}(Q_{T})}^{2}$,
using It\^o's formula (cf. \cite{krylov2013relatively}), we have
\begin{align*}
 & \int_{x}\int_{0}^{u_{n,\epsilon}(t,x)}\Phi_{n}(r)\mathrm{d}r\\
 & =\int_{x}\int_{0}^{\xi_{n}}\Phi_{n}(r)\mathrm{d}r-\int_{0}^{t}\int_{x}\partial_{x_{i}}\Phi_{n}(u_{n,\epsilon})\partial_{x_{i}}\Phi_{n}(u_{n,\epsilon})\mathrm{d}s\\
 & \quad-\int_{0}^{t}\int_{x}\Big(\partial_{x_{i}}a^{ij}(x,u_{n,\epsilon})\partial_{y_{j}}u_{n,\epsilon}+b^{i}(x,u_{n,\epsilon})\Big)\partial_{x_{i}}\Phi_{n}(u_{n,\epsilon})\mathrm{d}s\\
 & \quad+\int_{0}^{t}\int_{x}\Big(f(s,x,u_{n,\epsilon})+\frac{1}{\epsilon}(u_{n,\epsilon}-\psi)^{-}\Big)\Phi_{n}(u_{n,\epsilon})\mathrm{d}s\\
 & \quad+\frac{1}{2}\int_{0}^{t}\int_{x}\sum_{k=1}^{\infty}|\sigma_{x_{i}}^{ik}(x,u_{n,\epsilon})+\sigma_{r}^{ik}(x,u_{n,\epsilon})\partial_{x_{i}}u_{n,\epsilon}|^{2}\Phi_{n}^{\prime}(u_{n,\epsilon})\mathrm{d}s\\
 & \quad-\sum_{k=1}^{\infty}\int_{0}^{t}\int_{x}\sigma^{ik}(x,u_{n,\epsilon})\partial_{x_{i}}\Phi_{n}(u_{n,\epsilon})\mathrm{d}W_{s}^{k}.
\end{align*}
Since (\ref{eq:couple term in L2}) and
\begin{align*}
\int_{x}b^{i}(x,u_{n,\epsilon})\Phi_{n}^{\prime}(u_{n,\epsilon})\partial_{x_{i}}u_{n,\epsilon} & =-\int_{x}\llbracket b_{x_{i}}^{i}\Phi_{n}^{\prime}\rrbracket(x,u_{n,\epsilon}),\\
\int_{x}\sigma^{ik}(x,u_{n,\epsilon})\Phi_{n}^{\prime}(u_{n,\epsilon})\partial_{x_{i}}u_{n,\epsilon} & =-\int_{x}\llbracket\sigma_{x_{i}}^{ik}\Phi_{n}^{\prime}\rrbracket(x,u_{n,\epsilon}),
\end{align*}
in view of Assumption \ref{assu:assumption for phi} and the monotonicity
of $\Phi_{n}$, we have
\begin{equation}
\begin{aligned} & \mathbb{E}\Vert\nabla\Phi_{n}(u_{n,\epsilon})\Vert_{L_{2}(Q_{T})}^{2}\\
 & \leq C+C\mathbb{E}\Vert\xi_{n}\Vert_{L_{m+1}(\mathbb{T}^{d})}^{m+1}+C\mathbb{E}\Vert u_{n,\epsilon}\Vert_{L_{m+1}(Q_{T})}^{m+1}+\int_{t,x}\frac{1}{\epsilon}(u_{n,\epsilon}-\psi)^{-}\Phi_{n}(M)\\
 & \leq C+C\mathbb{E}\Vert\xi_{n}\Vert_{L_{m+1}(\mathbb{T}^{d})}^{m+1}+C\int_{t,x}\frac{1}{\epsilon}(u_{n,\epsilon}-\psi)^{-}\\
 & \leq C+C\mathbb{E}\Vert\xi_{n}\Vert_{L_{m+1}(\mathbb{T}^{d})}^{m+1}.
\end{aligned}
\label{eq:second estimate ito formula-1}
\end{equation}
Therefore, this proposition is proved.
\end{proof}
Next, we prove the uniform ($\star$)-property of $L_{2}$-solution
$u_{n,\epsilon}$. 
\begin{prop}
\label{prop:star property}Let Assumptions \ref{assu:assumption for phi}
and \ref{assu:assumption for sigma} hold. For any $n\in\mathbb{N}$,
$\epsilon>0$,  $\xi\in L_{m+1}(\Omega,\mathcal{F}_{0};L_{m+1}(\mathbb{T}^{d}))$, 
and  $\psi\in C(\bar{Q}_{T})$, let $u_{n,\epsilon}$ be
the $L_{2}$-solution to $\Pi(\Phi_{n},f+P_{\epsilon}(\cdot,\psi),\xi_{n})$.
Then, the function $u_{n,\epsilon}$ has the ($\star$)-property.
If further $\xi\in L_{4}(\Omega;L_{2}(\mathbb{T}^{d}))$, the constant
$C$ in Definition \ref{def:star property} is independent of $n$
and $\epsilon$.
\end{prop}

\begin{proof}
Fix $\theta>0$ small enough. We will apply the approximation method
in the proof of \cite[Lemma 5.3]{dareiotis2020nonlinear}. For a function
$g\in L_{2}(\mathbb{T}^{d})$ and a constant $\gamma>0$, let $g^{(\gamma)}:=\varrho_{\gamma}*g$
be the mollification. Then, the function $u_{n,\epsilon}^{(\gamma)}$
satisfies (pointwise) the equation
\begin{align*}
du_{n,\epsilon}^{(\gamma)} & =\bigg[\Big(\Delta\Phi_{n}(u_{n,\epsilon})+\partial_{x_{i}}\big(a^{ij}(\cdot,u_{n,\epsilon})\partial_{x_{j}}u_{n,\epsilon}+b^{i}(\cdot,u_{n,\epsilon})\big)\\
 & \quad+f(t,\cdot,u_{n,\epsilon})+\frac{1}{\epsilon}(u_{n,\epsilon}-\psi)^{-}\Big)^{(\gamma)}\bigg]\mathrm{d}t+\big(\partial_{x_{i}}\sigma^{ik}(\cdot,u_{n,\epsilon})\big)^{(\gamma)}\mathrm{d}W_{t}^{k}.
\end{align*}
Applying It\^{o}'s formula, we have 
\[
\int_{t,x,z}H(t,x,z|\tilde{u},\phi{}_{\theta},h)\big(\rho_{\lambda}(u_{n,\epsilon}^{(\gamma)}(t,x)-z)-\rho_{\lambda}(u_{n,\epsilon}^{(\gamma)}(t-\theta,x)-z)\big)=\sum_{l=1}^{6}N_{\lambda,\gamma}^{(l)},
\]
where
\begin{align*}
N_{\lambda,\gamma}^{(1)} & :=\int_{t,x,z}H(t,x,z|\tilde{u},\phi{}_{\theta},h)\int_{t-\theta}^{t}\rho_{\lambda}^{\prime}(u_{n,\epsilon}^{(\gamma)}(s,x)-z)\Delta\big(\Phi_{n}(u_{n,\epsilon})\big)^{(\gamma)}\mathrm{d}s,\\
N_{\lambda,\gamma}^{(2)} & :=\int_{t,x,z}H(t,x,z|\tilde{u},\phi{}_{\theta},h)\int_{t-\theta}^{t}\rho_{\lambda}^{\prime}(u_{n,\epsilon}^{(\gamma)}(s,x)-z)\Big(\partial_{x_{i}}\big(a^{ij}(x,u_{n,\epsilon})\partial_{x_{j}}u_{n,\epsilon}\\
&\qquad\qquad\qquad\qquad\qquad\qquad+2b^{i}(x,u_{n,\epsilon})\big)\Big)^{(\gamma)}\mathrm{d}s,\\
N_{\lambda,\gamma}^{(3)} & :=\int_{t,x,z}H(t,x,z|\tilde{u},\phi{}_{\theta},h)\int_{t-\theta}^{t}\rho_{\lambda}^{\prime}(u_{n,\epsilon}^{(\gamma)}(s,x)-z)\Big(-\partial_{x_{i}}b^{i}(x,u_{n,\epsilon})\Big)^{(\gamma)}\mathrm{d}s,\\
N_{\lambda,\gamma}^{(4)} & :=\int_{t,x,z}H(t,x,z|\tilde{u},\phi{}_{\theta},h)\int_{t-\theta}^{t}\rho_{\lambda}^{\prime}(u_{n,\epsilon}^{(\gamma)}(s,x)-z)\Big(\partial_{x_{i}}\sigma^{ik}(x,u_{n,\epsilon})\Big)^{(\gamma)}\mathrm{d}W_{s}^{k},\\
N_{\lambda,\gamma}^{(5)} & :=\frac{1}{2}\int_{t,x,z}H(t,x,z|\tilde{u},\phi{}_{\theta},h)\int_{t-\theta}^{t}\rho_{\lambda}^{\prime\prime}(u_{n,\epsilon}^{(\gamma)}(s,x)-z)\sum_{k=1}^{\infty}\Big|\Big(\partial_{x_{i}}\sigma^{ik}(x,u_{n,\epsilon})\Big)^{(\gamma)}\Big|^{2}\mathrm{d}s,\\
N_{\lambda,\gamma}^{(6)} & :=\int_{t,x,z}H(t,x,z|\tilde{u},\phi{}_{\theta},h)\int_{t-\theta}^{t}\rho_{\lambda}^{\prime}(u_{n,\epsilon}^{(\gamma)}(s,x)-z)\big(f(s,\cdot,u_{n,\epsilon})+\frac{1}{\epsilon}(u_{n,\epsilon}-\psi)^{-}\big)^{(\gamma)}\mathrm{d}s.
\end{align*}
For $N_{\lambda,\gamma}^{(6)}$, using the integration by parts formula
in $z$, we have
\[
\mathbb{E}|N_{\lambda,\gamma}^{(6)}|\leq N_{1}^{(6)}+N_{2}^{(6)},
\]
where
\begin{align*}
N_{1}^{(6)} & :=\mathbb{E}\big|\int_{t,x,z}\partial_{z}H(t,x,z|\tilde{u},\phi{}_{\theta},h)\int_{t-\theta}^{t}\rho_{\lambda}(u_{n,\epsilon}^{(\gamma)}(s,x)-z)\big(f(s,\cdot,u_{n,\epsilon})\big)^{(\gamma)}\mathrm{d}s\big|,\\
N_{2}^{(6)} & :=\mathbb{E}\big|\int_{t,x,z}\partial_{z}H(t,x,z|\tilde{u},\phi{}_{\theta},h)\int_{t-\theta}^{t}\rho_{\lambda}(u_{n,\epsilon}^{(\gamma)}(s,x)-z)\big(\frac{1}{\epsilon}(u_{n,\epsilon}-\psi)^{-}\big)^{(\gamma)}\mathrm{d}s\big|.
\end{align*}
Applying Assumption \ref{assu:assumption for sigma}, \cite[Remark 3.2 and Lemma 3.8]{dareiotis2020nonlinear}
and Lemma \ref{prop:estimate for u_n,epsilon}, we have
\begin{equation}
\begin{aligned}N_{1}^{(6)} & \leq\mathbb{E}\Big|\Vert\partial_{z}H\Vert_{L_{\infty}(Q_{T}\times\mathbb{R})}\int_{t,x}\int_{t-\theta}^{t}\big(f(s,\cdot,u_{n,\epsilon})\big)^{(\gamma)}\int_{z}\rho_{\lambda}(u_{n,\epsilon}^{(\gamma)}(s,x)-z)\mathrm{d}s\Big|\\
 & \leq2\mathbb{E}\Big|\Vert\partial_{z}H\Vert_{L_{\infty}(Q_{T}\times\mathbb{R})}\theta\int_{t,x}\big(f(t,\cdot,u_{n,\epsilon})\big){}^{(\gamma)}\Big|\\
 & \leq C\theta\Big(\mathbb{E}\Vert\partial_{z}H\Vert_{L_{\infty}(Q_{T}\times\mathbb{R})}^{2}\Big)^{1/2}\Big(\Vert f(\cdot,\cdot,u_{n,\epsilon})\Vert_{L_{2}(\Omega\times Q_{T})}^{2}\Big)^{1/2}\\
 & \leq C\theta\Big(\mathbb{E}\Vert\partial_{z}H\Vert_{L_{\infty}(Q_{T}\times\mathbb{R})}^{2}\Big)^{1/2}\Big(1+\Vert u_{n,\epsilon}\Vert_{L_{2}(\Omega\times Q_{T})}^{2}\Big)^{1/2}\leq C\theta^{1-\mu},
\end{aligned}
\label{eq:star estimate for F}
\end{equation}
where $\mu:=(3m+5)/(4m+4)<1$, and the constant $C$ is independent
with $n$ and $\epsilon$. Similarly, using estimate (\ref{eq:priori estimate1}),
we have 
\begin{align}
N_{2}^{(6)} & \leq C\theta\Big(\mathbb{E}\Vert\partial_{z}H\Vert_{L_{\infty}(Q_{T}\times\mathbb{R})}^{2}\Big)^{1/2}\Big(\mathbb{E}\Vert P_{\epsilon}(u_{n,\epsilon},\psi)\Vert_{L_{1}(Q_{T})}^{2}\Big)^{1/2}\leq C\theta^{1-\mu}.\label{eq:star estimate for G}
\end{align}
The estimates for $N_{\lambda,\gamma}^{(i)}$, $i=1,\ldots,5$ can
be obtained as in the proof of \cite[Lemma 5.3]{dareiotis2020nonlinear}.
Combining these estimates and following the proof of \cite[Lemma 5.3]{dareiotis2020nonlinear},
we have
\begin{align*}
 & \mathbb{E}\int_{t,x,a}H(t,x,u_{n,\epsilon}(t,x)|\tilde{u},\phi{}_{\theta},h)\\
 & \leq\limsup_{\lambda\rightarrow0}\limsup_{\gamma\rightarrow0}\mathbb{E}\big[\big|N_{\lambda,\gamma}^{(1)}\big|+\big|N_{\lambda,\gamma}^{(3)}\big|+\big|N_{\lambda,\gamma}^{(6)}\big|\big]+\limsup_{\lambda\rightarrow0}\limsup_{\gamma\rightarrow0}\mathbb{E}\big|N_{\lambda,\gamma}^{(2)}+N_{\lambda,\gamma}^{(5)}\big|\\
 & \quad+\lim_{\lambda\rightarrow0}\lim_{\gamma\rightarrow0}\mathbb{E}N_{\lambda,\gamma}^{(4)}\leq C\theta^{1-\mu}+\mathcal{B}(u_{n,\epsilon},\tilde{u}|\phi{}_{\theta},h),
\end{align*}
Moreover, when $\xi\in L_{4}(\Omega;L_{2}(\mathbb{T}^{d}))$, the
constant $C$ is independent with $n$ and $\epsilon$, which concludes the proof of 
the proposition.
\end{proof}
With the help of the following result, we can obtain the ($\star$)-property of $u$.
\begin{lem}[Corollary 3.9 in \cite{dareiotis2020nonlinear}]
\label{lem:appro for star}Let $\{u_{n}\}_{n\in\mathbb{N}}$ be a
sequence bounded in $L_{m+1}(\Omega_{T}\times\mathbb{T}^{d})$ satisfying
the ($\star$)-property uniformly in $n$, which means that the constant
$C$ in Definition \ref{def:star property} is independent of $n$.
If $u_{n}$ converges to a function $u$ almost surely on $\Omega\times Q_{T}$,
then $u$ has the ($\star$)-property.
\end{lem}

\section{Comparison principle\label{sec:Comparison-principle}}

In this section, we shall prove the monotonocity of $u_{n,\epsilon}$
with respect to $\epsilon$. Since the doubling-variables technique
in $L_{1}$-estimate can deal with the nonlinear term $\Delta\Phi(u_{n,\epsilon})$,
we use it instead of directly applying the It\^o's formula to $(u_{\epsilon_{1}}-u_{\epsilon_{2}})^{+}$.

The method here basically
follows the ones in the proof of \cite[Lemma 5.1]{liu2024obstacle}.
Note that in view of the definition of $L_{2}$-solution, the smoothness
of $\sqrt{\Phi_{n}^{\prime}}$ and inequality (\ref{eq:priori estimate2}),
we have that the $L_{2}$-solution $u_{n}$ is also an entropy solution
to $\Pi(\Phi_{n},f,\xi_{n})$ using It\^o's formula. 

\begin{prop}
\label{prop:comparison}Fix $n\in\mathbb{N}$. Let Assumptions \ref{assu:assumption for phi}
and \ref{assu:assumption for sigma} hold for both $f$ and $\tilde{f}$.
For any $\xi_{n}\in L_{m+1}(\Omega,\mathcal{F}_{0};L_{m+1}(\mathbb{T}^{d}))$,
suppose that $f(\cdot,\cdot,r)\leq\tilde{f}(\cdot,\cdot,r)$ on $Q_T$ for each $r\in\mathbb{R}$, and $u_{n}$ and $\tilde{u}_{n}$ are
entropy solutions to $\Pi(\Phi_{n},f,\xi_{n})$ and $\Pi(\Phi_{n},\tilde{f},\xi_{n})$,
respectively. Then, we have $u_{n}\leq\tilde{u}_{n}$ almost surely
on $\Omega\times Q_{T}$.
\end{prop}

\begin{proof}
The proof is similar to the ones of Proposition \ref{prop:uniqueentropy}.
The different lies in the selection of $\eta_{\delta}$. For each
$\delta>0$, we define the function $\eta_{\delta}\in C^{2}(\mathbb{R})$
by
\[
\eta_{\delta}(0)=\eta_{\delta}^{\prime}(0)=0,\quad\eta_{\delta}^{\prime\prime}(r)=\rho_{\delta}(r).
\]
Thus, we have
\[
|\eta_{\delta}(r)-r^{+}|\leq\delta,\quad\mathrm{supp}\,\eta_{\delta}^{\prime\prime}\subset[0,\delta],\quad\int_{\mathbb{R}}|\eta_{\delta}^{\prime\prime}(r)|\mathrm{d}r\leq2,\quad|\eta_{\delta}^{\prime\prime}|\leq2\delta^{-1}.
\]
Then, fix $(z,s,y)\in\mathbb{R}\times Q_{T}$. We use the entropy
inequality of $u_{n}$ with $\eta(r):=\eta_{\delta}(r-z)$ and $\phi(t,x):=\phi_{\theta,\varsigma}(t,x,s,y)$.
By taking $z=\tilde{u}_{n}(s,y)$, integrating over $(s,y)\in Q_{T}$, and taking expectations and $\theta\downarrow0$, we have the
estimate from the aspect of $u_{n}$. Note that the stochastic integral
is zero as the one in (\ref{eq:entropy doubling}).

Similarly, for each $(z,t,x)\in\mathbb{R}\times Q_{T}$, we apply
the entropy inequality of $\tilde{u}_{n}$ with $\eta(r):=\eta_{\delta}(z-r)$
and $\phi(s,y):=\phi_{\theta,\varsigma}(t,x,s,y)$. Then, substituting
$z=u_{n}(t,x)$, integrating over $(t,x)\in Q_{T}$ and taking expectations,
with the $(\star)$-property of $u_{n}$ where $h(r):=-\eta_{\delta}^{\prime}(-r)$
and $g(x,y):=\varrho_{\varsigma}(x-y)$, we finally take $\theta\downarrow0$
to obtain the estimate from the aspect of $\tilde{u}_{n}$. Adding
these two estimates together, we have
\begin{equation}
-\mathbb{E}\int_{t,x,y}\eta_{\delta}(u_{n}-\tilde{u}_{n})\partial_{t}\phi_{\varsigma}\leq\sum_{l=1}^{7}N_{l},\label{eq:inequalitywitht-1}
\end{equation}
where
\begin{align*}
N_{1} & :=\mathbb{E}\int_{t,x,y}\eta_{\delta}^{\prime}(u_{n}-\tilde{u}_{n})\big(f(t,x,u_{n})-\tilde{f}(t,y,\tilde{u}_{n})\big)\phi_{\varsigma},\\
N_{2} & :=\mathbb{E}\int_{t,x,y}\llbracket\Phi_{n}^{\prime}\eta_{\delta}^{\prime}(\cdot-\tilde{u}_{n})\rrbracket(u_{n})\Delta_{x}\phi_{\varsigma}-\mathbb{E}\int_{t,x,s,y}\llbracket\Phi_{n}^{\prime}\eta_{\delta}^{\prime}(u_{n}-\cdot)\rrbracket(\tilde{u}_{n})\Delta_{y}\phi_{\varsigma},\\
N_{3} & :=\mathbb{E}\int_{t,x,y}\llbracket a^{ij}\eta_{\delta}^{\prime}(\cdot-\tilde{u}_{n})\rrbracket(x,u_{n})\partial_{x_{i}x_{j}}\phi_{\varsigma}-\mathbb{E}\int_{t,x,y}\llbracket a^{ij}\eta_{\delta}^{\prime}(u_{n}-\cdot)\rrbracket(y,\tilde{u}_{n})\partial_{y_{i}y_{j}}\phi_{\varsigma}\\
 & \quad+\mathbb{E}\int_{t,x,y}\llbracket a_{x_{j}}^{ij}\eta_{\delta}^{\prime}(\cdot-\tilde{u}_{n})\rrbracket(x,u_{n})\partial_{x_{i}}\phi_{\varsigma}-\mathbb{E}\int_{t,x,y}\llbracket a_{y_{j}}^{ij}\eta_{\delta}^{\prime}(u_{n}-\cdot)\rrbracket(y,\tilde{u}_{n})\partial_{y_{i}}\phi_{\varsigma}\\
 & \quad-2\mathbb{E}\int_{t,x,y}\eta_{\delta}^{\prime}(u_{n}-\tilde{u}_{n})b^{i}(x,u_{n})\partial_{x_{i}}\phi_{\varsigma}+2\mathbb{E}\int_{t,x,y}\eta_{\delta}^{\prime}(u_{n}-\tilde{u}_{n})b^{i}(y,\tilde{u}_{n})\partial_{y_{i}}\phi_{\varsigma},\\
N_{4} & :=\mathbb{E}\int_{t,x,y}\llbracket b_{r}^{i}\eta_{\delta}^{\prime}(\cdot-\tilde{u}_{n})\rrbracket(x,u_{n})\partial_{x_{i}}\phi_{\varsigma}-\mathbb{E}\int_{t,x,y}\llbracket b_{r}^{i}\eta_{\delta}^{\prime}(u_{n}-\cdot)\rrbracket(y,\tilde{u}_{n})\partial_{y_{i}}\phi_{\varsigma}\\
 & \quad-\mathbb{E}\int_{t,x,y}\eta_{\delta}^{\prime}(u_{n}-\tilde{u}_{n})b_{x_{i}}^{i}(x,u_{n})\phi_{\varsigma}+\mathbb{E}\int_{t,x,y}\eta_{\delta}^{\prime}(u_{n}-\tilde{u}_{n})b_{y_{i}}^{i}(y,\tilde{u}_{n})\phi_{\varsigma}\\
 & \quad+\mathbb{E}\int_{t,x,y}\llbracket b_{rx_{i}}^{i}\eta_{\delta}^{\prime}(\cdot-\tilde{u}_{n})\rrbracket(x,u_{n})\phi_{\varsigma}-\mathbb{E}\int_{t,x,y}\llbracket b_{ry_{i}}^{i}\eta_{\delta}^{\prime}(u_{n}-\cdot)\rrbracket(y,\tilde{u}_{n})\phi_{\varsigma},\\
N_{5} & :=\mathbb{E}\int_{t,x,y}\frac{1}{2}\eta_{\delta}^{\prime\prime}(u_{n}-\tilde{u}_{n})\sum_{k=1}^{\infty}\Big(|\sigma_{x_{i}}^{ik}(x,u_{n})|^{2}+|\sigma_{y_{i}}^{ik}(y,\tilde{u}_{n})|^{2}\Big)\phi_{\varsigma},\\
N_{6} & :=\mathbb{E}\int_{t,x,y}-\eta_{\delta}^{\prime\prime}(u_{n}-\tilde{u}_{n})\Big(|\nabla_{x}\llbracket\sqrt{\Phi_{n}^{\prime}}\rrbracket(u_{n})|^{2}+|\nabla_{y}\llbracket\sqrt{\Phi_{n}^{\prime}}\rrbracket(\tilde{u}_{n})|^{2}\Big)\phi_{\varsigma},\\
N_{7} & :=-\mathbb{E}\int_{t,x,y}\partial_{y_{i}x_{j}}\phi_{\varsigma}\int_{\tilde{u}_{n}}^{u_{n}}\int_{r}^{\tilde{u}_{n}}\eta_{\delta}^{\prime\prime}(r-\tilde{r})\sigma_{r}^{jk}(x,r)\sigma_{r}^{ik}(y,\tilde{r})\mathrm{d}\tilde{r}\mathrm{d}r\\
 & \quad-\mathbb{E}\int_{t,x,y}\partial_{y_{i}}\phi_{\varsigma}\int_{\tilde{u}_{n}}^{u_{n}}\int_{r}^{\tilde{u}_{n}}\eta_{\delta}^{\prime\prime}(r-\tilde{r})\sigma_{rx_{j}}^{jk}(x,r)\sigma_{r}^{ik}(y,\tilde{r})\mathrm{d}\tilde{r}\mathrm{d}r\\
 & \quad+\mathbb{E}\int_{t,x,y}\partial_{y_{i}}\phi_{\varsigma}\int_{u_{n}}^{\tilde{u}_{n}}\eta_{\delta}^{\prime\prime}(u_{n}-\tilde{r})\sigma_{x_{j}}^{jk}(x,u_{n})\sigma_{r}^{ik}(y,\tilde{r})\mathrm{d}\tilde{r}\\
 & \quad-\mathbb{E}\int_{t,x,y}\partial_{x_{j}}\phi_{\varsigma}\int_{\tilde{u}_{n}}^{u_{n}}\int_{r}^{\tilde{u}_{n}}\eta_{\delta}^{\prime\prime}(r-\tilde{r})\sigma_{r}^{jk}(x,r)\sigma_{ry_{i}}^{ik}(y,\tilde{r})\mathrm{d}\tilde{r}\mathrm{d}r\\
 & \quad-\mathbb{E}\int_{t,x,y}\phi_{\varsigma}\int_{\tilde{u}_{n}}^{u_{n}}\int_{r}^{\tilde{u}_{n}}\eta_{\delta}^{\prime\prime}(r-\tilde{r})\sigma_{rx_{j}}^{jk}(x,r)\sigma_{ry_{i}}^{ik}(y,\tilde{r})\mathrm{d}\tilde{r}\mathrm{d}r\\
 & \quad+\mathbb{E}\int_{t,x,y}\phi_{\varsigma}\int_{u_{n}}^{\tilde{u}_{n}}\eta_{\delta}^{\prime\prime}(u_{n}-\tilde{r})\sigma_{x_{j}}^{jk}(x,u_{n})\sigma_{ry_{i}}^{ik}(y,\tilde{r})\mathrm{d}\tilde{r}\\
 & \quad-\mathbb{E}\int_{t,x,y}\partial_{x_{j}}\phi_{\varsigma}\int_{\tilde{u}_{n}}^{u_{n}}\eta_{\delta}^{\prime\prime}(r-\tilde{u}_{n})\sigma_{r}^{jk}(x,r)\sigma_{y_{i}}^{ik}(y,\tilde{u}_{n})\mathrm{d}r\\
 & \quad+\mathbb{E}\int_{t,x,y}\phi_{\varsigma}\int_{\tilde{u}_{n}}^{u_{n}}\eta_{\delta}^{\prime\prime}(r-\tilde{u}_{n})\sigma_{rx_{j}}^{jk}(x,r)\sigma_{y_{i}}^{ik}(y,\tilde{u}_{n})\mathrm{d}r\\
 & \quad-\mathbb{E}\int_{t,x,y}\phi_{\varsigma}\eta_{\delta}^{\prime\prime}(u_{n}-\tilde{u}_{n})\sigma_{x_{j}}^{jk}(x,u_{n})\sigma_{y_{i}}^{ik}(y,\tilde{u}_{n}).
\end{align*}
where $u_{n}=u_{n}(t,x)$ and $\tilde{u}_{n}=\tilde{u}_{n}(t,y)$.
To estimate term $N_{1}$, using Assumption \ref{assu:assumption for sigma}
and $f\leq\tilde{f}$, we have
\begin{align*}
N_{1} & =\mathbb{E}\int_{t,x,y}\eta_{\delta}^{\prime}(u_{n}-\tilde{u}_{n})\Big(f(t,x,u_{n})-\tilde{f}(t,x,u_{n})\Big)\phi_{\varsigma}\\
 & \quad+\mathbb{E}\int_{t,x,y}\eta_{\delta}^{\prime}(u_{n}-\tilde{u}_{n})\Big(\tilde{f}(t,x,u_{n})-\tilde{f}(t,y,\tilde{u}_{n})\Big)\phi_{\varsigma}\\
 & \leq C\varsigma^{\kappa}+C\mathbb{E}\int_{t,x,y}(u_{n}-\tilde{u}_{n})^{+}\varrho_{\varsigma}(x-y)\varphi(t).
\end{align*}
Other terms can be estimated as the proof of Proposition \ref{prop:uniqueentropy}.
Then, we obtain 
\begin{align*}
 & -\mathbb{E}\int_{t,x,y}(u_{n}-\tilde{u}_{n})^{+}\varrho_{\varsigma}(x-y)\partial_{t}\varphi(t)\\
 & \leq C\big(\varsigma^{\kappa}+\delta\varsigma^{-1}+\varsigma^{-2}\delta^{2\alpha}+\varsigma^{-2}\delta^{2}+\varsigma^{2}\delta^{-1}\big)\cdot\mathbb{E}\Big[1+\Vert u_{n}\Vert_{L_{m+1}(Q_{T})}^{m+1}+\Vert\tilde{u}_{n}\Vert_{L_{m+1}(Q_{T})}^{m+1}\Big]\\
 & \quad+C\mathbb{E}\int_{t,x,y}\varphi(t)(u_{n}-\tilde{u}_{n})^{+}\Big(\varsigma^{2}\sum_{i,j}|\partial_{y_{i}x_{j}}\varrho_{\varsigma}(x-y)|+\varsigma|\partial_{x_{i}}\varrho_{\varsigma}(x-y)|+\varrho_{\varsigma}(x-y)\Big).
\end{align*}
We choose $\vartheta\in(m^{-1}\lor2^{-1},1)$ and $\alpha\in(1/(2\vartheta),1\land(m/2))$.
Let $\delta=\varsigma^{2\vartheta}$ and $\varsigma\downarrow0$.
We have $\mathfrak{C}(\varsigma,\delta)\downarrow0$. Note that
$\varsigma^{2}\sum_{i,j}|\partial_{y_{i}x_{j}}\varrho_{\varsigma}|$
and $\varsigma|\partial_{x_{i}}\varrho_{\varsigma}|$ are approximations
of the identity up to a constant. With the continuity of translations
in $L_{1}$, we have
\[
\mathbb{E}\int_{t,x}(u_{n}(\tau,x)-\tilde{u}_{n}(\tau,x))^{+}\partial_{t}\varphi(t)\leq C\mathbb{E}\int_{t,x}(u_{n}(t,x)-\tilde{u}_{n}(t,x))^{+}\varphi(t).
\]
Following the proof of Proposition \ref{prop:uniqueentropy} and using
Gr\"onwall's inequality, we have
\[
\mathbb{E}\int_{x}(u_{n}(\tau,x)-\tilde{u}_{n}(\tau,x))^{+}=0,\quad\text{a.s.}\ \tau\in[0,T],
\]
which proves this lemma. 
\end{proof}

\section{Existence\label{sec:Existence}}

We first introduce an existence result on $\Pi(\Phi,f+P_{\epsilon}(\cdot,\psi),\xi)$,
the solution of which is a almost surely limit of the subsequence
of $\{u_{n,\epsilon}\}_{n\in\mathbb{N}}$ when $n\to\infty$. Then,
taking the limit $\epsilon\downarrow0$, with the the monotonicity
of $u_{\epsilon}$ in $\epsilon$, we acquire the existence of the
entropy solution $(u,\nu)$. 
\begin{lem}
\label{lem:exist for u_=00005Cepsilon}Fix $\epsilon>0$. Let Assumption
\ref{assu:assumption for phi} hold. For any $\xi\in L_{4}(\Omega,\mathcal{F}_{0};L_{2}(\mathbb{T}^{d}))\cap L_{m+1}(\Omega,\mathcal{F}_{0};L_{m+1}(\mathbb{T}^{d}))$,  $\psi\in C_x^{\kappa}(\bar{Q}_{T})$ satisfying (\ref{eq:assumption on xi}),
and $n\in\mathbb{N}$, the problem $\Pi(\Phi_{n},f+P_{\epsilon}(\cdot,\psi),\xi_{n})$
has a unique $L_{2}$-solution $u_{n,\epsilon}$. Moreover, by taking
subsequence, we can assume that when $n\rightarrow\infty$, $u_{n,\epsilon}$
almost surely converge to $u_{\epsilon}$ which is an entropy solution
of $\Pi(\Phi,f+P_{\epsilon}(\cdot,\psi),\xi)$. The function $u_{\epsilon}$
has the $(\star)$-property and satisfies 
\begin{equation}
\begin{aligned}\mathbb{E}\sup_{t\leq T}\,\Vert u_{\epsilon}(t)\Vert_{L_{2}(\mathbb{T}^{d})}^{2}+\mathbb{E}\Vert\nabla\llbracket\sqrt{\Phi^{\prime}}\rrbracket(u_{\epsilon})\Vert_{L_{2}(Q_{T})}^{2} & +\mathbb{E}\Vert\epsilon^{-1}(u_{\epsilon}-\psi)^{-}\Vert_{L_{1}(Q_{T})}^{(m+1)/2}\\
 & \leq C\Big(1+\mathbb{E}\Vert\xi\Vert_{L_{2}(\mathbb{T}^{d})}^{m+1}\Big),
\end{aligned}
\label{eq:priori estimate1_u_=00005Cepsilon}
\end{equation}
and
\begin{align}
 & \mathbb{E}\sup_{t\leq T}\,\Vert u_{\epsilon}(t)\Vert_{L_{m+1}(\mathbb{T}^{d})}^{m+1}\leq C\Big(1+\Vert\xi\Vert_{L_{m+1}(\Omega\times\mathbb{T}^{d})}^{m+1}\Big).\label{eq:priori estimate2_u_=00005Cepsilon}
\end{align}
\end{lem}

The proof of existence in Lemma \ref{lem:exist for u_=00005Cepsilon}
follows a standard procedure, wherein we prove the existence of the
$L_{2}$-solution $u_{n,\epsilon}$ using the Galerkin method, and
then taking $n\to\infty$ (cf. \cite{dareiotis2020nonlinear,du2024entropy,liu2024obstacle}).
Further treatment of the inhomogeneous function $f+P_{\epsilon}(\cdot,\psi)$
is required in our case, and this is easily proved due to the H\"older
continuity in $x$ and Lipschitz continuity in $r$. Moreover,
a priori estimates (\ref{eq:priori estimate1_u_=00005Cepsilon}) and
(\ref{eq:priori estimate2_u_=00005Cepsilon}) result from Proposition
\ref{prop:estimate for u_n,epsilon}, along with the almost sure convergence
of $u_{n,\epsilon}$ and uniform integrability of $\{|u_{n,\epsilon}|^{q}\}_{n\in\mathbb{N}}$
for any $0<q<m+1$. The $(\star)$-property of $u_{\epsilon}$ can
be obtained using Proposition \ref{prop:star property} and Lemma
\ref{lem:appro for star}. Therefore, we omit the proof.
\begin{rem}
The $(\star)$-property of $u_{\epsilon}$ is uniform on $\epsilon$,
if $\xi\in L_{4}(\Omega,\mathcal{F}_{0};L_{2}(\mathbb{T}^{d}))$.

We are now in a positive to conclude the proof of Theorem \ref{thm:existence_main thm}.
\end{rem}

\begin{proof}[Proof of Theorem \ref{thm:existence_main thm}]
First, we assume additionally $\xi\in L_{4}(\Omega,\mathcal{F}_{0};L_{2}(\mathbb{T}^{d}))$.
Let $\{\epsilon_{l}\}_{i\in\mathbb{N}}$ be a monotone sequence decreasing
to $0$. Using Lemma \ref{lem:exist for u_=00005Cepsilon}, the penalized
equation $\Pi(\Phi,f+P_{\epsilon_{l}}(\cdot,\psi),\xi)$ has an entropy
solution $u_{\epsilon_{l}}$. From Proposition \ref{prop:comparison},
the functions $u_{\epsilon_{l}}$ almost surely increase as $l\rightarrow\infty$.
Set the pointwise limit to be $u$ which takes value in $[-\infty,\infty]$.
With the estimates (\ref{eq:priori estimate2_u_=00005Cepsilon}) and
Fatous lemma, we have $u\in L_{m+1}(\Omega\times Q_{T})$, which implies
$|u|<\infty$ almost sure on $\Omega\times Q_{T}$. Furthermore, this
convergence is strong in $L_{m+1}(\Omega\times Q_{T})$ based on the
dominated convergence theorem, and we have
\[
\mathbb{E}\sup_{t\leq T}\,\Vert u(t)\Vert_{L_{m+1}(\mathbb{T}^{d})}^{m+1}\leq C\Big(1+\Vert\xi\Vert_{L_{m+1}(\Omega\times\mathbb{T}^{d})}^{m+1}\Big).
\]
Furthermore, the $(\star)$-property of $u$ is a direct consequence
from Proposition \ref{prop:star property} and Lemma \ref{lem:appro for star}.

Define $\nu_{l}:=P_{\epsilon_{l}}(u_{\epsilon_{l}},\psi)$. We verify the existence and properties of $\nu$ through three steps:

\textit{Step 1}. There is a subsequence of $\{\nu_{l}\}_{l\in\mathbb{N}}$ that is weakly convergent in the space $L_{(m+1)/2}(\Omega;H_{q}^{-1}(Q_{T}))$ with some $q\in(1,1/d+1)$. 
This is actually a consequence of the following uniform estimate:
\begin{equation}\label{eq:H-1}
\sup_{l\in\mathbb{N}}\Vert\nu_{l}\Vert_{L_{(m+1)/2}(\Omega;H_{q}^{-1}(Q_{T}))}\leq C\sup_{l\in\mathbb{N}}\Vert\nu_{l}\Vert_{L_{(m+1)/2}(\Omega;L_{1}(Q_{T}))}<\infty,
\end{equation}
which follows from the Sobolev embedding and duality.
Indeed, when $p>d+1$, the
space $H_{p,0}^{1}(Q_{T})$ is continuously embedding into the space
$C_{0}^{\alpha}(\bar{Q}_{T})$, which means for all $g\in H^1_{p,0}(Q_T):=\{u\in H^1_{p}(Q_T):u(T,\cdot)=0\}$, we have
\begin{equation*}
\Vert g\Vert_{L_{\infty}(Q_{T})}\leq\Vert g\Vert_{C^{\alpha}(Q_{T})}\leq C\Vert g\Vert_{H_{p}^{1}(Q_{T})}.
\end{equation*}
Then, for all $h\in L_1(Q_T)$, we have
\begin{align*}
\Big|\int_{Q_T}g h \mathrm{d}x\mathrm{d}t\Big|\leq \Vert g\Vert_{L_\infty(Q_T)}\Vert h\Vert_{L_1(Q_T)}
\leq C\Vert g\Vert_{H^1_p(Q_T)} \Vert h\Vert_{L_1(Q_T)}.
\end{align*}
Choosing $q$ such that $1/p+1/q=1$ we have
\[
\begin{aligned}
\Vert h \Vert_{H^{-1}_q(Q_T)}\leq C \Vert h\Vert_{L_1(Q_T)}.
\end{aligned}
\]
This confirms the claim \eqref{eq:H-1}.

Let $\nu\in L_{(m+1)/2}(\Omega;H_{q}^{-1}(Q_{T}))$
be the weak limit of a subsequence of $\{\nu_{l}\}_{l\in\mathbb{N}}$ (also denoted as $\{\nu_{l}\}_{l\in\mathbb{N}}$ by abuse of notation).
Since $\nu_{l}\geq0$ as a function in the space $L_{(m+1)/2}(\Omega;L_{1}(Q_{T}))$,
we know that $\nu$ is a.s. a nonnegative distribution, and therefore it is actually a.s. a locally finite Radon measure from a well-known result \cite[Page 84]{strichartz2003guide}.

\textit{Step 2}. For each $\phi\in H^1_{p,0}(Q_T)$, there is a full probability subset $\Omega_{\phi}\subset \Omega$ such that for all $\omega\in \Omega_\phi$,
\begin{equation}\label{eq:aslimit}
\lim_{l\to\infty}\int_{Q_T}\nu_l(\omega)\phi\mathrm{d}x\mathrm{d}t=\int_{Q_T}\nu(\omega)\phi\mathrm{d}x\mathrm{d}t.
\end{equation}
To show this, using the weak convergence of $\{\nu_{l}\}_{l\in\mathbb{N}}$, we have for all $\phi\in H^1_{p,0}(Q_T)$ and $Z\in L_{(m+1)/(m-1)}(\Omega)$, 
\[
\mathbb{E}\Big[Z\int_{Q_T}\nu_{l}\phi\mathrm{d}x\mathrm{d}t\Big]\xrightarrow{l\to\infty}\mathbb{E}\Big[Z\int_{Q_T}\nu\phi\mathrm{d}x\mathrm{d}t\Big],
\]
which means $\int_{Q_T}\nu_{l}\phi\,\mathrm{d}x\mathrm{d}t$ converges to $\int_{Q_T}\nu\phi\,\mathrm{d}x\mathrm{d}t$ weakly in $L_{(m+1)/2}(\Omega)$.

On the other hand, recalling that $\nu_{l}=P_{\epsilon_{l}}(u_{\epsilon_{l}},\psi)$, and $u_{\epsilon}$ is the entropy solution (so it is also a weak solution)
of $\Pi(\Phi,f+P_{\epsilon}(\cdot,\psi),\xi)$, one has that for all $\phi\in H^1_{p,0}(Q_T)$,
\begin{equation}
\begin{aligned} & -\int_{Q_T}\nu_{l} \phi \mathrm{d}x\mathrm{d}t\\
 & =\int_{Q_T}u_{\epsilon_{l}}\partial_{t}\phi\mathrm{d}x\mathrm{d}t+\int_{\mathbb{T}^{d}}\xi\phi(0)\mathrm{d}x+\int_{Q_T}\Phi(u_{\epsilon_{l}})\Delta\phi\mathrm{d}x\mathrm{d}t\\
 & \quad+\int_{Q_T}\llbracket a^{ij}\rrbracket(x,u_{\epsilon_{l}})\partial_{x_{i}x_{j}}\phi\mathrm{d}x\mathrm{d}t\\
 & \quad+\int_{Q_T}\Big(\llbracket a_{x_{j}}^{ij}+b_{r}^{i}\rrbracket(x,u_{\epsilon_{l}})-2b^{i}(x,u_{\epsilon_{l}})\Big)\partial_{x_{i}}\phi\mathrm{d}x\mathrm{d}t\\
 & \quad+\int_{Q_T}\Big(-b_{x_{i}}^{i}(x,u_{\epsilon_{l}})+\llbracket b_{rx_{i}}^{i}\rrbracket(x,u_{\epsilon_{l}})+f(t,x,u_{\epsilon_{l}})\Big)\phi\mathrm{d}x\mathrm{d}t\\
 & \quad+\int_{Q_T}\Big(\phi\sigma_{x_{i}}^{ik}(x,u_{\epsilon_{l}})-\llbracket\sigma_{rx_{i}}^{ik}\rrbracket(x,u_{\epsilon_{l}})\phi-\llbracket\sigma_{r}^{ik}\rrbracket(x,u_{\epsilon_{l}})\partial_{x_{i}}\phi\Big)\mathrm{d}x\mathrm{d}W_{t}^{k}.
\end{aligned}
\label{eq:entropy_equality}
\end{equation}
Thanks to the almost sure convergence of $u_{\epsilon_l}$ (to $u$), 
the right-hand side of (\ref{eq:entropy_equality}) is convergent almost surely.
Therefore, $\int_{Q_T}\nu_{l}\phi\,\mathrm{d}x\mathrm{d}t$ is convergent almost surely, and the limit must coincide with its weak limit $\int_{Q_T}\nu\phi\mathrm{d}x\mathrm{d}t$.

\textit{Step 3}.  We show that $\nu$ is almost surely a finite Radon measure on $Q_T$. Specifically, there exists a random variable $\beta\in L_{(m+1)/2}(\Omega)$ such that
\begin{equation}\label{eq:beta}
\Big|\int_{Q_T}\nu\phi\mathrm{d}x\mathrm{d}t\Big|\leq \beta\Vert\phi\Vert_{C({Q}_T)}\quad \text{a.s.}
\end{equation}
for any $\phi\in C(\bar{Q}_T)$ with $\phi(T,\cdot)=0$. This verifies the condition (ii) in Definition \ref{def:entropy solution with ob}.

Since we do not require $\phi(0,\cdot)=0$, we need to do an even continuation of $\nu$ and $\nu_l$:
define $\tilde{\nu}_l(t,\cdot):=\nu_l(|t|,\cdot)$ and
$\tilde{\nu}(t,\cdot):=\nu(|t|,\cdot)$ for $t\in(-T,T)$, 
and $\tilde{Q}_T:=(-T,T)\times\mathbb{T}^d$. 

Let $\{\phi_{k'}\}_{k'\in\mathbb{N}}\subset H^1_{p,0}(\tilde{Q}_T)$ be a countable dense subset of $C_0(\tilde{Q}_T):=\{u\in C(\tilde{Q}_T):u(T,\cdot)=u(-T,\cdot)=0\}$. 
Define the nonnegative random variables $\beta_l:=\tilde{\nu}_l(\tilde{Q}_T)\in L_{(m+1)/2}(\Omega)$. Since $\Vert\beta_l\Vert_{L_{(m+1)/2}(\Omega)}$ is uniformly bounded for $l\in\mathbb{N}$, we know that there is a subsequene of $\beta_l$ (still denoted by itself) weakly converging to a $\beta\in L_{(m+1)/2}(\Omega)$ as $l\to\infty$. 
It follows from Step 2 that there is a full probability subset $\Omega_1\subset \Omega$ such that
\[
\Big|\int_{Q_T}\nu_l(\omega)\phi_{k'}\mathrm{d}x\mathrm{d}t\Big|\leq \beta_l(\omega)\Vert\phi_{k'} \Vert_{C({Q}_T)}
\]
for all  $\phi_{k'}$ and $\omega\in\Omega_1$. 
This along with the Banach--Saks theorem yields \eqref{eq:beta}.
\medskip

Finally, we verify that $(u,\nu)$ is an entropy solution to the obstacle
problem $\Pi_{\psi}(\Phi,f,\xi)$. 
Actually, it remains to verify the condition (iii) in Definition \ref{def:entropy solution with ob}.
Recall the test functions $\eta$ and $\phi$ in Definition \ref{def:entropy solution with ob}. 
From the definition of the entropy
solution $u_{\epsilon_{l}}$, for all $B\in\mathcal{F}$, we have 
\begin{equation}
\begin{aligned} & -\mathbb{E}\bigg[\mathbf{1}_{B}\bigg(\int_{Q_T}\eta(u_{\epsilon_l})\partial_{t}\phi\mathrm{d}x\mathrm{d}t+\int_{Q_T}\eta^{\prime}(\psi)\nu_{\epsilon_l}\phi\mathrm{d}x\mathrm{d}t\bigg)\bigg]\\
 & \leq\mathbb{E}\Bigg[\mathbf{1}_{B}\bigg[\int_{\mathbb{T}^{d}}\eta(\xi)\phi(0)\mathrm{d}x+\int_{Q_T}\llbracket\Phi^{\prime}\eta^{\prime}\rrbracket(u_{\epsilon_l})\Delta\phi\mathrm{d}x\mathrm{d}t\\
 & \quad+\int_{Q_T}\llbracket a^{ij}\eta^{\prime}\rrbracket(u_{\epsilon_l})\partial_{x_{i}x_{j}}\phi\mathrm{d}x\mathrm{d}t\\
 & \quad+\int_{Q_T}\Big(\llbracket a_{x_{j}}^{ij}\eta^{\prime}+b_{r}^{i}\eta^{\prime}\rrbracket(x,u_{\epsilon_l})-2\eta^{\prime}(u_{\epsilon_l})b^{i}(x,u_{\epsilon_l})\Big)\partial_{x_{i}}\phi\mathrm{d}x\mathrm{d}t\\
 & \quad+\int_{Q_T}\Big(-\eta^{\prime}(u)b_{x_{i}}^{i}(x,u_{\epsilon_l})+\llbracket b_{rx_{i}}^{i}\eta^{\prime}\rrbracket(x,u_{\epsilon_l})+\eta^{\prime}(u_{\epsilon_l})f(t,x,u_{\epsilon_l})\Big)\phi\mathrm{d}x\mathrm{d}t\\
 & \quad+\int_{Q_T}\Big(\frac{1}{2}\eta^{\prime\prime}(u_{\epsilon_l})\sum_{k=1}^{\infty}|\sigma_{x_{i}}^{ik}(x,u_{\epsilon_l})|^{2}-\eta^{\prime\prime}(u_{\epsilon_l})|\nabla\llbracket\sqrt{\Phi^{\prime}}\rrbracket(u_{\epsilon_l})|^{2}\Big)\phi\mathrm{d}x\mathrm{d}t\\
 & \quad+\int_{Q_T}\Big(\eta^{\prime}(u_{\epsilon_l})\phi\sigma_{x_{i}}^{ik}(x,u_{\epsilon_l})-\llbracket\sigma_{rx_{i}}^{ik}\eta^{\prime}\rrbracket(x,u_{\epsilon_l})\phi-\llbracket\sigma_{r}^{ik}\eta^{\prime}\rrbracket(x,u_{\epsilon_l})\partial_{x_{i}}\phi\Big)\mathrm{d}x\mathrm{d}W_{t}^{k}\bigg]\Bigg].
\end{aligned}
\label{eq:Itofornepsilon}
\end{equation}
For any $\eta^{\prime\prime}\in C(\mathbb{R})$, using
\[
|\llbracket\sqrt{\Phi^{\prime}}\eta^{\prime\prime}\rrbracket(r)|\leq C\Vert \eta^{\prime\prime}\Vert_{L_{\infty}}|r|^{(m+1)/2},\quad\forall r\in\mathbb{R},
\]
Proposition \ref{prop:estimate for u_n,epsilon}, and $\partial_{x_{i}}\llbracket\sqrt{\Phi^{\prime}}\eta^{\prime\prime}\rrbracket(u_{\epsilon_l})=\eta^{\prime\prime}(u_{\epsilon_l})\partial_{x_{i}}\llbracket\sqrt{\Phi^{\prime}}\rrbracket(u_{\epsilon_l})$, we have
\[
\sup_{\epsilon_l}\,\mathbb{E}\int_{t}\Vert\llbracket\sqrt{\Phi^{\prime}}\eta^{\prime\prime}\rrbracket(u_{\epsilon_l})\Vert_{H^{1}(\mathbb{T}^{d})}^{2}<\infty.
\]
By taking a non-relabelled subsequence, we have that $\llbracket\sqrt{\Phi^{\prime}}\eta^{\prime\prime}\rrbracket(u_{\epsilon_l})$
converges weakly to some $v$ in $L_{2}(\Omega_{T};H^{1}(\mathbb{T}^{d}))$.
Then, with the pointwise convergence and uniform integrability of
$u_{\epsilon_l}$, we have $v=\llbracket\sqrt{\Phi^{\prime}}\eta^{\prime\prime}\rrbracket(u)$, which implies
\begin{align*}
 & \mathbb{E}\Bigg[\mathbf{1}_{B}\int_{Q_T}\eta^{\prime\prime}(u)|\nabla\llbracket\sqrt{\Phi^{\prime}}\rrbracket(u)|^{2}\phi\mathrm{d}x\mathrm{d}t\Bigg]\\
 & \leq\liminf_{l\rightarrow\infty}\mathbb{E}\Bigg[\mathbf{1}_{B}\int_{Q_T}\eta^{\prime\prime}(u_{\epsilon_l})|\nabla\llbracket\sqrt{\Phi^{\prime}}\rrbracket(u_{\epsilon_l})|^{2}\phi\mathrm{d}x\mathrm{d}t\Bigg].
\end{align*}
Therefore, taking inferior limit on (\ref{eq:Itofornepsilon}) and
using Assumptions \ref{assu:assumption for phi}-\ref{assu:assumption for sigma}
and the convergence of $u_{\epsilon_l}$, we acquire the condition (iii) in Definition \ref{def:entropy solution with ob}. 

For general $\xi\in L_{m+1}(\Omega\times\mathbb{T}^{d})$, we define
$\xi_{n}:=(-n)\lor(\xi\land n)$. From above, the problem $\Pi_{\psi}(\Phi,f,\xi_{n})$
has an entropy solution $(u_{n},\nu_{n})$, and $u_{n}$ has the $(\star)$-property.
Using Proposition \ref{prop:uniqueentropy}, we have
\[
\underset{t\in[0,T]}{\mathrm{ess\ sup}}\,\Vert u_{n}(t,\cdot)-u_{\tilde{n}}(t,\cdot)\Vert_{L_{1}(\Omega\times\mathbb{T}^{d})}\leq C\Vert\xi_{n}-\xi_{\tilde{n}}\Vert_{L_{1}(\Omega\times\mathbb{T}^{d})}.
\]
As the above proof, when $n,\tilde{n}\to\infty$, by taking subsequence, the
solution $(u_{n},\nu_{n})$ has an limit $(u,\nu)$, which is an entropy
solution to $\Pi_{\psi}(\Phi,f,\xi)$. This complete the proof.
\end{proof}

\section{The porous medium equation\label{sec:Deterministic}}
As an application, our result can apply to the obstacle problem of deterministic porous medium equations.
Let us consider
\begin{equation}
\begin{aligned} \label{eq:deteministic PME}
\partial_t u&=\Delta(|u|^{m-1}u)+f(t,x,u)+\nu,\\
u&\geq \psi,\quad
u(0,\cdot)=\xi,
\end{aligned}
\end{equation}
where $m>1$.
This obstacle problem without the reaction term $f$ has been studied in the framework of variational inequality by many works such as \cite{alt1983quasilinear,kinnunen2008definition,bogelein2015obstacle}, 
which mainly focus on lower obstacles and non-negative solutions.
The uniqueness of solutions are usually absent in those works.

The entropy formulation of the problem \eqref{eq:deteministic PME} is the following.
\begin{defn}\label{def:entropy solution of PME}
An entropy solution to \eqref{eq:deteministic PME} is a pair $(u,\nu)$ such that
\begin{itemize}
\item[(i)]  $u\ge\psi$
for almost all $(t,x)\in Q_{T}$ and
\[
u\in L_{m+1}(Q_T),\quad|u|^{\frac{m-1}{2}}u\in L_{2}(0,T;H^{1}(\mathbb{T}^{d}));
\]
\item[(ii)] $\nu$ is a finite Radon measure on $Q_{T}$;
\item[(iii)] for all $(\eta,\varphi,\varrho)\in\mathcal{E}\times C_{c}^{\infty}[0,T)\times C^{\infty}(\mathbb{T}^{d})$
and $\phi:=\varphi\varrho\geq0$, it holds almost surely that
\begin{equation}
\begin{aligned}  &- \int_{Q_{T}}\eta(u)\partial_{t}\phi\,\mathrm{d}x\mathrm{d}t-\int_{Q_{T}}\eta^{\prime}(\psi)\phi\nu(\mathrm{d}x\mathrm{d}t)\\
 & \leq\int_{\mathbb{T}^{d}}\eta(\xi)\phi(0)\mathrm{d}x+m\int_{Q_{T}}\bigg(\int_0^{u(t,x)} |r|^{m-1}\eta^\prime(r)\mathrm{d}r\bigg)\Delta\phi\,\mathrm{d}x\mathrm{d}t\\
 & \quad-\frac{4m}{(m+1)^2}\int_{Q_{T}}\eta^{\prime\prime}(u)\big|\nabla(|u|^{\frac{m-1}{2}}u)\big|^{2}\phi\,\mathrm{d}x\mathrm{d}t
 +\int_{Q_{T}}\eta^{\prime}(u)f(t,x,u)\phi\,\mathrm{d}x\mathrm{d}t.
\end{aligned}
\label{eq:entropy formula for deterministic PME}
\end{equation}
\end{itemize}
\end{defn}

In view of Theorems \ref{thm:existence_main thm} and \ref{thm:unque_main thm}, we have the following well-posedness result.
\begin{thm}\label{thm:well-posedness of PME}
Given $\psi\in C^{\kappa}_x(\bar{Q}_{T})$ and $f\in C(\bar{Q}_{T}\times\mathbb{R})$ satisfying \eqref{eq:assumption for f},
there exists an entropy solution $(u,\nu)$ to \eqref{eq:deteministic PME}
for any initial data $\xi\in L_{m+1}(\mathbb{T}^{d})$ with $\xi\ge\psi(0,\cdot)$.
If $\psi\in C^{2}_x(\bar{Q}_T)$ in addition, and $(\tilde{u},\tilde{\nu})$
is the entropy solution to \eqref{eq:deteministic PME} with the initial data $\tilde{\xi}\in L_{m+1}(\mathbb{T}^{d})$, then
\[
\underset{t\in[0,T]}{\mathrm{ess\ sup}}\,\Vert u(t)-\tilde{u}(t)\Vert_{L_{1}(\mathbb{T}^{d})}\leq C\Vert\xi-\tilde{\xi}\Vert_{L_{1}(\mathbb{T}^{d})},
\]
where the constant $C$ depends only on $K$, $d$, and $T$. 
\end{thm}
Next, we demonstrate the relation between our entropy solution and the variational solution. 
For clarity, we focus on the case $f=0$, and the following argument still works for non-zero $f$. 
Define
\begin{align*}
K^{\prime}_{\psi}:=\big{\{} v\in C([0,T];L_{m+1}(\mathbb{T}^d));\ &|v|^{m-1}v\in L_2(0,T;H^1(\mathbb{T}^d)),\\
&\ \partial_t(|v|^{m-1}v)\in L^{\frac{m+1}{m}}(Q_T),\text{ and}\ v\geq\psi\ \text{a.e. on }Q_T \big{\}}.
\end{align*}

\begin{prop}\label{prop:connection with variational solution}
Let $\psi\in C^{2}_x(\bar{Q}_{T})$ and $\xi\in L_{m+1}(\mathbb{T}^{d})$ satisfying (\ref{eq:assumption on xi}). 
The entropy solution to \eqref{eq:deteministic PME} satisfies the variational inequality \begin{align}\label{eq:weak variational inequality}
\langle\!\langle\partial_t u,(|v|^{m-1}v-|u|^{m-1}u)\tilde{\phi}\rangle\!\rangle+\int_{Q_{T}}\nabla\big(|u|^{m-1}u\big)\cdot\nabla[(|v|^{m-1}v-|u|^{m-1}u)\tilde{\phi}]\,\mathrm{d}x\mathrm{d}t\geq0,
\end{align}
for all $v\in K^\prime_\psi$ and $\tilde\phi\in H^{1}_{\infty}([0,T];\mathbb{R}_+)$ with $\tilde{\phi}(T)=0$, where
\begin{align*}  
&\langle\!\langle\partial_t u_\epsilon,(|v|^{m-1}v-|u_\epsilon|^{m-1}u_\epsilon)\tilde{\phi}\rangle\!\rangle:=
\int_{Q_T}\partial_t\tilde{\phi}\Big[\frac{1}{m+1}|u_\epsilon|^{m+1}-u_\epsilon |v|^{m-1}v\Big] \\
&\qquad\qquad-\tilde{\phi} u_\epsilon \partial_t(|v|^{m-1}v)\mathrm{d}x\mathrm{d}t+\int_{Q_T}\tilde{\phi}(0)\Big[\frac{1}{m+1}|\xi|^{m+1}-\xi |v(0)|^{m-1}v(0)\Big]\mathrm{d}x.
\end{align*}
\end{prop}
The solution characterized by the variational inequality \eqref{eq:weak variational inequality} is called a weak solution to the obstacle problem \eqref{eq:deteministic PME} (cf.~\cite{bogelein2015obstacle}). 
The proof of Proposition \ref{prop:connection with variational solution} 
is based on the following lemma that reveals more regularity of the entropy solution $u$. 
For simplicity, we consider
$$\Phi(r):=|r|^{m-1}r.$$
\begin{lem}\label{lem:more-regularity}
Under the assumption of Proposition \ref{prop:connection with variational solution}, for each $\epsilon>0$, the entropy solution $u_\epsilon$ of penalized equation $\Pi(\Phi,P_{\epsilon}(\cdot,\psi),\xi)$ satisfies 
\begin{equation*}
\Vert\Phi(u_{\epsilon})\Vert_{L_2(0,T;H^1(\mathbb{T}^d))}\leq C(1+\Vert\xi\Vert_{L_{m+1}(\mathbb{T}^d)}),
\end{equation*}
where the constant $C$ is independent with $\epsilon$. 
Consequently, for the entropy solution $u$ to the obstacle problem $\Pi_\psi(\Phi,0,\xi)$, we have $\Phi(u)\in L_2(0,T;H^1(\mathbb{T}^d))$.
\end{lem}
\begin{proof}
Since the entropy solution $u_\epsilon$ is approximated by the solution $u_{n,\epsilon}$, we first consider the approximating penalized equation $\Pi(\Phi_n,P_\epsilon(\cdot,\psi),\xi_n)$,  where $\Phi_n$ is the approximation of $\Phi$ defined in Lemma \ref{lem:Assumptio for coef}.  From a priori estimate \eqref{eq:priori estimate2} and Poincar\'e inequality, we have
\begin{align*}
\Vert\Phi_n(u_{n,\epsilon})\Vert_{L_2(Q_T)}&\leq \Vert\nabla\Phi_n(u_{n,\epsilon})\Vert_{L_2(Q_T)}+C\Vert\Phi_n(u_{n,\epsilon})\Vert_{L_1(Q_T)}\\
&\leq C+\Vert\nabla\Phi_n(u_{n,\epsilon})\Vert_{L_2(Q_T)} +C \Vert u_{n,\epsilon}\Vert_{L_{m+1}(Q_T)}^{m+1}\leq C + C\Vert\xi\Vert_{L_{m+1}(\mathbb{T}^d)}^{m+1},
\end{align*}
which means $\{\Phi_n(u_{n,\epsilon})\}_{n\in\mathbb{N},\epsilon>0}$ is uniformly bounded in $L_2(0,T;H^{1}(\mathbb{T}^d))$. For fixed $\epsilon$ and a subsequence, we have $\Phi_n(u_{n,\epsilon})$ weakly converges to some $w_\epsilon\in L_2(0,T;H^{1}(\mathbb{T}^d))$ when $n\rightarrow \infty$. Since it is also a weak convergence in $L_{(m+2)/{(m+1)}}(Q_T)$, in order to verify that $w_\epsilon=\Phi(u_{\epsilon})$, we only need to prove that $\Phi_n(u_{n,\epsilon})$ weakly converges to $\Phi(u_{\epsilon})$ in $L_{(m+2)/(m+1)}(Q_T)$.

Note that by taking subsequence, $u_{n,\epsilon}$ almost surely converges $u_{\epsilon}$. Using the fact that $\{|u_{n,\epsilon}|^q\}_{n\in\mathbb{N}}$ is uniformly integrable on $\Omega_T\times\mathbb{T}^d$ for all $0\leq q< {m+1}$, we have
\begin{equation}\label{eq:converge of Phi when nto inf}
\lim_{n\rightarrow\infty} \big\Vert \Phi(u_{n,\epsilon}) - \Phi(u_{\epsilon})\big\Vert_{L_{(m+2)/(m+1)}(Q_T)}=0.
\end{equation}
Moreover, since
\begin{align*}
&|\Phi_n(u_{n,\epsilon})-\Phi(u_{n,\epsilon})|\\
&\leq \Big|\int_0^{u_{n,\epsilon}}\big(\sqrt{\Phi_n^\prime(r)}-\sqrt{\Phi^\prime(r)}\big)\big(\sqrt{\Phi_n^\prime(r)}+\sqrt{\Phi^\prime(r)}\big)\mathrm{d}r\Big|\\
&\leq C\Big|\int_0^{|u_{n,\epsilon}|}\big|\sqrt{\Phi_n^\prime(r)}-\sqrt{\Phi^\prime(r)}\big|\big(1+|r|^{\frac{m-1}{2}}\big)\mathrm{d}r\Big|\\
&\leq C\Big|\int_0^{n\land|u_{n,\epsilon}|}n^{-2}\big(1+|r|^{m-1}\big)\mathrm{d}r\Big|^{\frac{1}{2}}\cdot\Big|\int_0^{|u_{n,\epsilon}|}(1+|r|^{m-1})\mathrm{d}r\Big|^{\frac{1}{2}}\\
&\quad+C\mathbf{1}_{|u_{n,\epsilon}|>n}\big(1+|u_{n,\epsilon}|^{m}\big)\\
&\leq C (n^{-1}+\mathbf{1}_{|u_{n,\epsilon}|>n})(1+|u_{n,\epsilon}|^m),
\end{align*}
we have
\begin{align}
&\Vert\Phi_n(u_{n,\epsilon})-\Phi(u_{n,\epsilon})\Vert_{L_{(m+2)/(m+1)}(Q_T)}\label{eq:converge of Phi when nto inf2}\\
&\leq C \Vert (n^{-1}+\mathbf{1}_{|u_{n,\epsilon}|>n}) (1+|u_{n,\epsilon}|^m)\Vert_{L_{(m+2)/(m+1)}(Q_T)}\stackrel{n\to\infty}{\longrightarrow} 0.\nonumber
\end{align} 
Combining \eqref{eq:converge of Phi when nto inf} and \eqref{eq:converge of Phi when nto inf2}, we have $\Phi_n(u_{n,\epsilon})$ strongly converges to $\Phi(u_{\epsilon})$ in $L_{(m+2)/(m+1)}(Q_T)$, which indicates $w_\epsilon=\Phi(u_\epsilon)$. Then, by taking subsequence, $\Phi_n(u_{n,\epsilon})$ weakly converges to $\Phi(u_\epsilon)$ in $L_2(0,T;H^1(\mathbb{T}^d))$. With a priori estimate \eqref{eq:priori estimate2}, the sequence $\{\Phi(u_{\epsilon})\}_{\epsilon>0}$ is uniformly bounded in $L_2(0,T;H^1(\mathbb{T}^d))$. When taking $\epsilon\downarrow 0$, with the above procedure, we also have $\Phi(u)\in L_2(0,T;H^1(\mathbb{T}^d))$.
\end{proof}
\begin{proof}[Proof of Proposition \ref{prop:connection with variational solution}]
We first verify that $u_\epsilon$ of the penalized equation satisfies the variational inequality \cite[(2.7)]{bogelein2015obstacle}. Taking $\eta(r)=r$ and $\eta(r)=-r$ in the entropy inequality of $u_\epsilon$ respectively to obtain that
\begin{equation*}
\begin{aligned}  & -\int_{Q_{T}}u_\epsilon\partial_{t}\phi\,\mathrm{d}x\mathrm{d}t-\int_{Q_{T}}\phi \epsilon^{-1}(u_\epsilon-\psi)^-\mathrm{d}x\mathrm{d}t\\
 & =\int_{\mathbb{T}^{d}}\xi\phi(0)\mathrm{d}x-\int_{Q_{T}}\nabla\big(|u_\epsilon|^{m-1}u_\epsilon\big)\cdot\nabla\phi\,\mathrm{d}x\mathrm{d}t.
\end{aligned}
\end{equation*}
Since $|u_\epsilon|^{m-1}u_\epsilon \in L_2 (0,T;H^1(\mathbb{T}^d))$, we can extend the test function $\phi$ to the ones in $L^2(0,T;H^1(\mathbb{T}^d))$ satisfying $\partial_t(\phi^m)\in L^{\frac{m+1}{m}}(Q_T)$ and $\phi(T,\cdot)\equiv 0$. Using time mollifications of $|u_\epsilon|^{m-1}u_\epsilon$ to taking the  test function $\phi=(|v|^{m-1}v-|u_\epsilon|^{m-1}u_\epsilon)\tilde{\phi}$, we have
\begin{equation*}
\begin{aligned}  
&\int_{Q_T}\partial_t\tilde{\phi}\Big[\frac{1}{m+1}|u_\epsilon|^{m+1}-u_\epsilon |v|^{m-1}v\Big]-\tilde{\phi} u_\epsilon \partial_t(|v|^{m-1}v)\mathrm{d}x\mathrm{d}t \\
&+\int_{Q_T}\tilde{\phi}(0)\Big[\frac{1}{m+1}|\xi|^{m+1}-\xi |v(0)|^{m-1}v(0)\Big]\mathrm{d}x\\
&=\langle\partial_t u_\epsilon,(|v|^{m-1}v-|u_\epsilon|^{m-1}u_\epsilon)\tilde{\phi}\rangle_\xi\\
&=\int_{Q_{T}}(|v|^{m-1}v-|u_\epsilon|^{m-1}u_\epsilon)\tilde{\phi} \epsilon^{-1}(u_\epsilon-\psi)^-\mathrm{d}x\mathrm{d}t\\
 &\quad-\int_{Q_{T}}\nabla\big(|u_\epsilon|^{m-1}u_\epsilon\big)\cdot\nabla[(|v|^{m-1}v-|u_\epsilon|^{m-1}u_\epsilon)\tilde{\phi}]\,\mathrm{d}x\mathrm{d}t
\end{aligned}
\end{equation*}
holds for all $v\in K^\prime_\psi$ and $\tilde\phi\in W^{1,\infty}([0,T];\mathbb{R}_+)$ with $\tilde{\phi}(T)=0$.
Due to $v\in K_\psi$, we have $v\geq \psi$, which indicates
\begin{align}
&\langle\!\langle\partial_t u_\epsilon,(|v|^{m-1}v-|u_\epsilon|^{m-1}u_\epsilon)\tilde{\phi}\rangle\!\rangle\nonumber\\
&+\int_{Q_{T}}\nabla\big(|u_\epsilon|^{m-1}u_\epsilon\big)\cdot\nabla[(|v|^{m-1}v-|u_\epsilon|^{m-1}u_\epsilon)\tilde{\phi}]\,\mathrm{d}x\mathrm{d}t\geq0.\label{eq:variational for penalty equation}
\end{align}
Since $|u_\epsilon|^{m-1}u_\epsilon\in L_2(0,T;H^1(\mathbb{T}^d))$, using a priori estimate \eqref{eq:priori estimate2} and noting that $u_{n,\epsilon}$ almost surely converges to $u_{\epsilon}$ in Lemma \ref{lem:exist for u_=00005Cepsilon}, we have $\nabla(|u_\epsilon|^{m-1}u_\epsilon)$ is uniformly bounded in $L_2(Q_T)$. Similarly, using the fact $|u|^{m-1}u\in L_2(0,T;H^1(\mathbb{T}^d))$, with the almost surely convergence of $u_{\epsilon}$ in the proof of Theorem \ref{thm:existence_main thm}, there exists a subsequence $u_{\tilde{\epsilon}}$ such that $\nabla(|u_{\tilde{\epsilon}}|^{m-1}u_{\tilde{\epsilon}})$ weakly converges to $\nabla(|u|^{m-1}u)$ in $L_2(Q_T)$. Then, we have
\begin{equation*}
\underset{\tilde{\epsilon}\downarrow 0}{\underline{\lim}}\int_{Q_T}|\nabla(|u_{\tilde{\epsilon}}|^{m-1}u_{\tilde{\epsilon}})|^2\mathrm{d}x\mathrm{d}t\geq \int_{Q_T}|\nabla(|u|^{m-1}u)|^2\mathrm{d}x\mathrm{d}t.
\end{equation*}
In \eqref{eq:variational for penalty equation}, we take the limit $\tilde{\epsilon}\downarrow 0$. Based on the monotone convergence of $u_{\tilde\epsilon}$ and $u\in L_{m+1}(Q_T)$, we have \eqref{eq:weak variational inequality} holds for all $v\in K^{\prime}_\psi$ and $\tilde\phi\in H^{1}_{\infty}([0,T];\mathbb{R}_+)$ with $\tilde{\phi}(T)=0$, which completes the proof.
\end{proof}

On the other hand, it seems difficult to prove that the variational solution is an entropy solution in general. 
Nevertheless, when the obstacle function is strictly positive, 
one may show that the the strong variational solution defined in \cite[Definition 2.2]{bogelein2015obstacle} is also an entropy solution.
Let us give a brief discussion as follows.

First of all, we recall the definition of the strong solution (cf.~\cite{bogelein2015obstacle}):
letting
\[
K_{\psi}:=\{u\in C(0,T;L_{{m+1}}(\mathbb{T}^d)): u^m\in L_2(0,T;H^1(\mathbb{T}^d)),\ u\geq\psi\ \text{a.e. on}\ Q_T\}.
\]
A functions $u\in K_{\psi}$ is a strong variational solution to \eqref{eq:deteministic PME} (with $f\equiv 0$) if $\partial_t u\in L_2(0,T;H^{-1}(\mathbb{T}^d))$ and
\begin{equation}
\int_{t,x}\partial_{t} u \cdot(v^m-u^m){\phi}+\int_{t,x}\nabla u^m\cdot\nabla[(v^m-u^m){\phi}]\geq0,\label{eq:bogelein_varia_inequality}
\end{equation}
holds for all $v\in K_{\psi}$ and every cut-off function ${\phi}\in H^1_\infty([0,T],\mathbb{R}_{\geq0})$ with $\tilde{\phi}(T)=0$.

Assume that the obstacle $\psi(t,x)\ge\lambda>0$ for all $(t,x)\in Q_T$.
Firstly, we show that a strong variational solution $u$ to~\eqref{eq:deteministic PME}
satisfies the following inequality
\begin{equation}
\int_{t,x}\partial_{t}u\cdot(\eta^{\prime}(\psi)-\eta^{\prime}(u)){\phi}+\int_{t,x}\nabla u^m\cdot\nabla[(\eta^{\prime}(\psi)-\eta^{\prime}(u)){\phi}]\ge 0, \label{eq:entropy varia_inequality}
\end{equation}  
for any ${\phi}\in H^1_\infty([0,T],\mathbb{R}_{\geq0})$ with ${\phi}(T)=0$.
Indeed, to verify this from~\eqref{eq:bogelein_varia_inequality}, one can take 
\[
v^m := u^m+\rho^{-1}[\eta^{\prime}({\psi})-\eta^{\prime}(u)],
\]
and carefully select the constant $\rho$ such that $v\in K_\psi$.
With $\tilde{\rho}:=\lambda^{-1} m^{-\frac{m}{m-1}}$, one has
\begin{align*}
 v^m & \geq(u^m-\psi^m)+(\rho\tilde{\rho})^{-1}|\eta^{\prime\prime}|_{L_\infty}\Big([(\tilde{\rho}\psi)^m)]^{\frac{1}{m}}-[({\tilde{\rho}u)^m}]^{\frac{1}{m}}\Big)+\psi^m\\
  & \geq(u^m-\psi^m)+\rho^{-1}\tilde{\rho}^{m-1}\|\eta^{\prime\prime}\|_{L_\infty} (\psi^m-u^m )+\psi^m\\
 & =\big(1-\rho^{-1}\tilde{\rho}^{m-1}\|\eta^{\prime\prime}\|_{L_\infty}\big) (u^m-\psi^m)+\psi^m \\
 & \ge \psi^m,
\end{align*}
as long as $\rho\geq\tilde{\rho}^{m-1}|\eta^{\prime\prime}|_{L_\infty}=m^{-m}\lambda^{-(m-1)}|\eta^{\prime\prime}|_{L_\infty}$.

Next, letting $\nu:=\partial_t u - \Delta u^m$, one has
\begin{equation*}
-\int_{t,x}\partial_t u\phi+\int_{t,x}\phi\nu(\mathrm{d}x\mathrm{d}t)-\int_{t,x}\nabla u^m\nabla\phi=0.
\end{equation*}
Replacing the test $\phi$ by $\eta^\prime(\psi) {\phi}$, and combining this with \eqref{eq:entropy varia_inequality}, one has that
\begin{equation*}
-\int_{t,x}\eta^\prime(u)\partial_t u\phi+\int_{t,x}\eta^\prime(\psi)\phi\nu(\mathrm{d}x\mathrm{d}t)-\int_{t,x}\nabla u^m\nabla[\eta^\prime(u)\phi]\geq0,
\end{equation*}
which implies the entropy inequality~\eqref{eq:entropy formula for deterministic PME}.

\noindent\textbf{Acknowledgements} This work is supported
by the National Science and Technology Major Project (2022ZD0116401).
The first author is also supported by the National Natural Science Foundation of China (No.~12222103),
and by Shanghai Institute for Mathematics and Interdisciplinary Sciences (Grand No.~SIMIS-ID-2024-WE),
and by LMNS at Fudan University.

\noindent\textbf{Author Contributions} All authors wrote the main manuscript text and reviewed the manuscript.

\noindent\textbf{Data availability} Not applicable.

\section*{Declarations}\noindent\textbf{Conflict of interest}  On behalf of all authors, the corresponding author states that there is no conflict of interest.

\begin{appendix}
\section{Auxiliary lemma\label{sec:Auxiliary}}

In this section, we introduce a lemma, which is a version of \cite[Lemma 3.2]{dareiotis2019entropy}, but extended to the Radon measure case.

\begin{lem}
\label{lem:appro initial}Let Assumptions \ref{assu:assumption for phi}
and \ref{assu:assumption for sigma} hold. Suppose  $\xi\in L_{m+1}(\Omega,\mathcal{F}_{0};L_{m+1}(\mathbb{T}^{d}))$
and  $\psi\in C^(\bar{Q}_{T})$ satisfying (\ref{eq:assumption on xi}).
Let $\nu$
be a Radon measure on $\bar{Q}_{T}$ for almost all $\omega\in\Omega$ satisfying
$\mathbb{E}\nu(\bar{Q}_{T})<\infty$. If $u\in L_{m+1}(\Omega\times Q_{T})$
and $\nu$ satisfy the entropy inequality
(\ref{eq:entropy formula-1}), we have
\[
\lim_{\tau\downarrow0}\frac{1}{\tau}\mathbb{E}\int_{0}^{\tau}\int_{\mathbb{T}^{d}}|u(t,x)-\xi(x)|^{2}\mathrm{d}x\mathrm{d}t=0.
\]
\end{lem}

\begin{proof}
For $\varsigma>0$ and $\varrho_{\varsigma}:=\rho_{\varsigma}^{\otimes d}$,
we have
\begin{align*}
\frac{1}{h}\mathbb{E}\int_{0}^{h}\int_{\mathbb{T}^{d}}|u(t,x)-\xi(x)|^{2}\mathrm{d}x\mathrm{d}t & \leq2\mathbb{E}\int_{\mathbb{T}^{d}}|\xi(y)-\xi(x)|^{2}\varrho_{\varsigma}(x-y)\mathrm{d}x\\
 & \quad+\frac{2}{h}\mathbb{E}\int_{0}^{h}\int_{\mathbb{T}^{d}}|u(t,x)-\xi(y)|^{2}\varrho_{\varsigma}(x-y)\mathrm{d}x\mathrm{d}t.
\end{align*}
We first estimate the second term on the right-hand side for $h\in[0,T]$.
Take a decreasing non-negative function $\gamma\in C^{\infty}([0,T])$,
such that
\[
\gamma(0)=2,\quad\gamma\leq2\mathbf{1}_{[0,2h]},\quad\partial_{t}\gamma\leq-\frac{1}{h}\mathbf{1}_{[0,h]}.
\]
For each $\delta>0$, take $\eta_{\delta}\in C^{2}(\mathbb{R})$ by
\[
\eta_{\delta}(0)=\eta_{\delta}^{\prime}(0)=0,\quad\eta_{\delta}^{\prime\prime}(r)\coloneqq2\mathbf{1}_{[0,\delta^{-1})}(|r|)+(-|r|+\delta^{-1}+2)\mathbf{1}_{[\delta^{-1},\delta^{-1}+2]}(|r|).
\]
Notice that $\eta_{\delta}(r)\rightarrow r^{2}$ as $\delta\downarrow0$.
Fix $(y,z)\in\mathbb{T}^{d}\times\mathbb{R}$. Using the entropy inequality
(\ref{eq:entropy formula-1}) with $\phi(t,x)=\gamma(t)\varrho_{\varsigma}(x-y)$
and $\eta(r)=\eta_{\delta}(r-z)$, we obtain
\begin{align*}
 & -\int_{Q_T}\eta_{\delta}(u-z)\partial_{t}\gamma(t)\varrho_{\varsigma}(x-y)\mathrm{d}x\mathrm{d}t\\
 & \leq2\int_{\mathbb{T}^{d}}\eta_{\delta}(\xi-z)\varrho_{\varsigma}(x-y)\mathrm{d}x\\
 & \quad+C\int_{Q_T}(1+|u|^{m+1}+|z|^{m+1})\Big(|\partial_{x_{i}x_{j}}\varrho_{\varsigma}(x-y)|+|\partial_{x_{i}}\varrho_{\varsigma}(x-y)|+\varrho_{\varsigma}(x-y)\Big)\\
 &\quad\quad\cdot\gamma(t)\mathrm{d}x\mathrm{d}t\\
 & \quad+\int_{Q_T}\eta_{\delta}^{\prime}(\psi-z)\varrho_{\varsigma}(x-y)\gamma(t)\nu(\mathrm{d}x\mathrm{d}t)\\
 & \quad+\int_{Q_T}\Big(\eta_{\delta}^{\prime}(u-z)\varrho_{\varsigma}(x-y)\sigma_{x_{i}}^{ik}(x,u)\\
 & \quad-\llbracket\sigma_{rx_{i}}^{ik}\eta_{\delta}^{\prime}(\cdot-z)\rrbracket(x,u)\varrho_{\varsigma}(x-y)-\llbracket\sigma_{r}^{ik}\eta_{\delta}^{\prime}(\cdot-z)\rrbracket(x,u)\partial_{x_{i}}\varrho_{\varsigma}(x-y)\Big)\gamma(t)\mathrm{d}x\mathrm{d}W_{t}^{k}.
\end{align*}
Note that all the terms are continuous in $z\in\mathbb{R}$. By substituting
$z=\xi(y)$, integrating over $y\in\mathbb{T}^{d}$, and taking expectations,
with the bounds on $\gamma$, we have
\begin{align*}
 & \frac{1}{h}\mathbb{E}\int_{0}^h\int_{\mathbb{T}^{d}}\int_{\mathbb{T}^{d}}\eta_{\delta}(u-\xi(y))\varrho_{\varsigma}(x-y)\mathrm{d}y\mathrm{d}x\mathrm{d}t\\
 & \leq2\mathbb{E}\int_{\mathbb{T}^{d}}\int_{\mathbb{T}^{d}}\eta_{\delta}(\xi(x)-\xi(y))\varrho_{\varsigma}(x-y)\mathrm{d}y\mathrm{d}x\\
 & \quad+\frac{C}{\varsigma^{2}}\mathbb{E}\int_{0}^{2h}\int_{\mathbb{T}^{d}}(1+|u(t,x)|^{m+1}+|\xi(x)|^{m+1})\mathrm{d}x\mathrm{d}t\\
 & \quad+\mathbb{E}\int_{\mathbb{T}^{d}}\int_{Q_T}\eta_{\delta}^{\prime}(\psi(t,x)-\xi(y))\varrho_{\varsigma}(x-y)\gamma(t)\nu(\mathrm{d}x\mathrm{d}t)\mathrm{d}y.
\end{align*}
With the monotonicity of $\eta_{\delta}^{\prime}$, (\ref{eq:assumption on xi}), 
and $\psi\in C(\bar{Q}_T)$, we have
\begin{align*}
 & \mathbb{E}\int_{\mathbb{T}^{d}}\int_{Q_T}\eta_{\delta}^{\prime}(\psi(t,x)-\xi(y))\varrho_{\varsigma}(x-y)\gamma(t)\nu(\mathrm{d}x\mathrm{d}t)\mathrm{d}y\\
 & \leq\mathbb{E}\int_{\mathbb{T}^{d}}\int_{Q_T}\eta_{\delta}^{\prime}(\psi(t,x)-\psi(0,y))\varrho_{\varsigma}(x-y)\gamma(t)\nu(\mathrm{d}x\mathrm{d}t)\mathrm{d}y\\
 & \leq\mathbb{E}\int_{\mathbb{T}^{d}}\int_{Q_T}\eta_{\delta}^{\prime}(\varpi(h)+\varpi(\varsigma))\varrho_{\varsigma}(x-y)\gamma(t)\nu(\mathrm{d}x\mathrm{d}t)\mathrm{d}y\\
 & \leq C(\varpi(h)+\varpi(\varsigma))\mathbb{E}\int_{\mathbb{T}^{d}}\int_{Q_T}\varrho_{\varsigma}(x-y)\gamma(t)\nu(\mathrm{d}x\mathrm{d}t)\mathrm{d}y\\
 & \leq C(\varpi(h)+\varpi(\varsigma))\mathbb{E}\int_{\mathbb{T}^{d}}\int_{Q_T}\varrho_{\varsigma}(x-y)\nu(\mathrm{d}x\mathrm{d}t)\mathrm{d}y\\
 & =C(\varpi(h)+\varpi(\varsigma))\mathbb{E}\nu(Q_{T}),
\end{align*}
where $\varpi$ is the modulus of continuity of $\psi$. Then, taking $\delta\downarrow0$ and then $h\downarrow0$,
we have
\begin{align*}
 & \limsup_{h\downarrow0}\frac{1}{h}\mathbb{E}\int_{0}^h\int_{\mathbb{T}^{d}}\int_{\mathbb{T}^{d}}|u-\xi(y)|^{2}\varrho_{\varsigma}(x-y)\mathrm{d}y\mathrm{d}x\mathrm{d}t\\
 & \leq2\mathbb{E}\int_{\mathbb{T}^{d}}\int_{\mathbb{T}^{d}}|\xi(x)-\xi(y)|^{2}\varrho_{\varsigma}(x-y)\mathrm{d}y\mathrm{d}x+C\varpi(\varsigma)\mathbb{E}\nu(Q_{T}).
\end{align*}
Then, we have
\begin{align*}
\lim_{h\downarrow0}\frac{1}{h}\mathbb{E}\int_{0}^{h}\int_{\mathbb{T}^{d}}|u(t,x)-\xi(x)|^{2}\mathrm{d}x\mathrm{d}t & \leq4\mathbb{E}\int_{\mathbb{T}^{d}}|\xi(y)-\xi(x)|^{2}\varrho_{\varsigma}(x-y)\mathrm{d}x\\
 & \quad+C\varpi(\varsigma)\mathbb{E}\nu(Q_{T}).
\end{align*}
With the continuity of translations in $L_{2}(\mathbb{T}^{d})$, the
right-hand side goes to $0$ as $\varsigma\downarrow0$.
\end{proof}
When $\nu\equiv0$, we have the following corollary.
\begin{cor}
\label{cor:appro initial_without ob}Let Assumptions \ref{assu:assumption for phi}
and \ref{assu:assumption for sigma} hold. If $\xi\in L_{m+1}(\Omega,\mathcal{F}_{0};L_{m+1}(\mathbb{T}^{d}))$,
and $u$ is an entropy solution of $\Pi(\Phi,f,\xi)$, we have 
\[
\lim_{\tau\downarrow0}\frac{1}{\tau}\mathbb{E}\int_{0}^{\tau}\int_{\mathbb{T}^{d}}|u(t,x)-\xi(x)|^{2}\mathrm{d}x\mathrm{d}t=0.
\]
\end{cor}
\end{appendix}

\bibliographystyle{plain}
\bibliography{bi}

\end{document}